\documentclass[Numbered]{my-sn-jnl}

\usepackage{graphicx}%
\usepackage{amsmath,amssymb,amsfonts}%
\usepackage{amsthm}%
\usepackage{mathrsfs}%
\usepackage[title]{appendix}%
\usepackage{xcolor}%
\usepackage{manyfoot}%
\usepackage{enumitem}%

\numberwithin{equation}{section}%
\newtheorem{theorem}{Theorem}[section]%
\newtheorem{proposition}[theorem]{Proposition}%
\newtheorem{definition}[theorem]{Definition}%
\newtheorem{remark}[theorem]{Remark}%
\newtheorem{corollary}[theorem]{Corollary}%

\begin{document}

\title[Unitary rational best approximations to the exponential function]{Unitary rational best approximations to the exponential function}

\author*[1]{\fnm{Tobias} \sur{Jawecki}}\email{tobias.jawecki@gmail.com}

\author[2]{\fnm{Pranav} \sur{Singh}}\email{ps2106@bath.ac.uk}

\affil*[1]{\orgdiv{Institute for Theoretical Physics}, \orgname{TU-Wien}, \orgaddress{\street{Wiedner Hauptstra{\ss}e 8-10}, \postcode{A-1040 Wien}, \country{Austria}}}

\affil[2]{\orgdiv{Department of Mathematical Sciences}, \orgname{University of Bath}, \orgaddress{\postcode{Bath BA2 7AY}, \country{UK}}}

\abstract{Rational best approximations (in a Chebyshev sense) to real functions are characterized by an equioscillating approximation error. Similar results do not hold true for rational best approximations to complex functions in general. In the present work, we consider unitary rational approximations to the exponential function on the imaginary axis, which map the imaginary axis to the unit circle. In the class of unitary rational functions, best approximations are shown to exist, to be uniquely characterized by equioscillation of a phase error, and to possess a super-linear convergence rate. Furthermore, the best approximations have full degree (i.e., non-degenerate), achieve their maximum approximation error at points of equioscillation, and interpolate at intermediate points. Asymptotic properties of poles, interpolation nodes, and equioscillation points of these approximants are studied. Three algorithms, which are found very effective to compute unitary rational approximations including candidates for best approximations, are sketched briefly. Some consequences to numerical time-integration are discussed. In particular, time propagators based on unitary best approximants are unitary, symmetric and A-stable.}



\maketitle

\section{Introduction and overview}\label{sec:intro}

We consider unitary rational approximations to the exponential function on the imaginary axis, i.e.,
\begin{equation}\label{eq:rpqapproxexp}
r(\mathrm{i} x) = \frac{p(\mathrm{i} x) }{q(\mathrm{i} x)} \approx \mathrm{e}^{\mathrm{i} \omega x},~~~x\in [-1,1]\subset\mathbb{R},~~ \omega>0,~~ q\not\equiv0,
\end{equation}
where $\omega$ is the {\em frequency}, $p$ and $q$ are polynomials and $r$ is {\em unitary}, i.e.,
\begin{equation}\label{eq:runitary}
|r(\mathrm{i}x)| = 1,~~~x\in\mathbb{R}.
\end{equation}
If $p$ and $q$ are of degree $\leq n$, the rational function $r$ is considered to be of degree $(n,n)$.

In this manuscript we are interested in {\em unitary rational best approximation}, i.e., the unitary rational function of degree $(n,n)$ that minimizes the approximation error of~\eqref{eq:rpqapproxexp} in a Chebyshev sense, i.e.,
\begin{equation}\label{eq:maxerrinSec1}
\max_{x\in[-1,1]} |r(\mathrm{i}x) - \mathrm{e}^{\mathrm{i}\omega x} |.
\end{equation}

\begin{remark}
\label{rmk:unitary best approx}
Since the only unitary approximations of interest in this manuscript are rational functions, we will frequently use the term `unitary best approximations' as a shorter form of `unitary rational best approximations'.
\end{remark}

\begin{remark}
\label{rmk:eiwx vs ewz}
While $r$ in~\eqref{eq:rpqapproxexp} might be most accurately referred to as an approximation to $\mathrm{e}^{\omega z}$ for $z\in\mathrm{i}[-1,1]$, 
we abuse notation slightly and refer to $r$ as an approximation to $\mathrm{e}^{\mathrm{i}\omega x}$ and $\mathrm{e}^{\omega z}$ interchangeably throughout the present work, as convenient in a particular context.    
\end{remark}

\begin{remark}
\label{rmk:ix}
We consider the approximation problem~\eqref{eq:rpqapproxexp} as $r(z)\approx\mathrm{e}^{\omega z}$ over an interval of the imaginary axis, $z \in \mathrm{i}[-1,1]$, throughout this paper. Thus, interpolation nodes and points where the approximation error attains a maximum in magnitude are considered as imaginary, i.e., $\mathrm{i}x$ for some $x\in\mathbb{R}$. However, for convenience, we occasionally use the real-valued $x\in \mathbb{R}$ to refer to the corresponding interpolation nodes and points of maximum error.
\end{remark}

\subsection{Contributions of the present work} 
A prominent feature of polynomial best approximations and rational best approximations to real-valued functions is the equioscillatory nature of the approximation error. This property characterizes the best approximations in various settings, provides important theoretical properties, e.g., uniqueness of best approximations, and motivates the use of Remez and rational Remez algorithms~\cite{PT09} to compute the best approximation in practice. For rational approximations to complex-valued functions, however, there are no equivalent results for equioscillation of the approximation error in general. In fact, uniqueness of best approximations, a typical consequence of equioscillation, has even shown to be false in some complex settings~\cite{GT83,RV89}.

{\bf Equioscillation and best approximation.}
In this manuscript we prove that unitary rational best approximations exist and are characterized by an equioscillation property of the {\em phase error} which guarantees uniqueness. 
The phase error has a zero between each pair of equioscillation points. 
The equioscillation points and zeros of the phase error (cf. Remark~\ref{rmk:ix}) correspond to points where the {\em approximation error} is maximal in magnitude and nodes where the best approximant interpolates $\mathrm{e}^{\mathrm{i} \omega x}$, respectively. The unitary best approximation of degree $(n,n)$ has distinct poles and minimal degree $n$ (i.e., for $r=p/q$ the polynomials $p$ and $q$ have degree exactly $n$ and no common zeros). The interpolatory property is in contrast to some complex rational best approximations that do not have any interpolation nodes on the approximation interval and show circular error curves~\cite{Tre81,NT20}. In particular, the error curve of unitary best approximants of degree $(n,n)$ have a floral pattern similar to {\em rose} curves with $2n$ petals~\cite{La95}, albeit with non-uniformly spaced and (possibly) nested petals.

{\bf Algorithms.} While the contributions of this manuscript are primarily of a theoretical nature and a full treatment of an algorithm for computing the unitary best approximation is beyond the scope of the present work, we briefly sketch three algorithms to compute approximants that are found to be good in practice:~(i) unitary rational interpolation at Chebyshev nodes (cf. Remark~\ref{rmk:ix}) using the rotated Loewner matrix approach of \cite{JS23}, (ii) a modified \cite{JS23} version of the AAA--Lawson algorithm \cite{NT20} with Chebyshev support nodes, and~(iii) a modified version of the BRASIL algorithm \cite{Ho21} initialized with Chebyshev interpolation nodes. While the algorithm~(i) is a very simple but effective algorithm for producing good approximants, the algorithms~(ii) and~(iii) produce approximants that seem to have an equioscillatory property in practice.
In particular, algorithm~(iii) is used in this manuscript to produce numerical illustrations to accompany theoretical results such as convergence rate, equioscillation property etc..

{\bf Convergence.} We show that unitary rational best approximations achieve super-linear convergence in the degree $n$. Our convergence proof is based on $(n,n)$-Pad\'e approximations, which are a prominent example of unitary rational approximations to $\mathrm{e}^z$. Pad\'e approximations are asymptotic in nature and are derived by matching the Taylor series expansion of $\mathrm{e}^z$ to the highest degree possible. In practice, for uniform approximation to $\mathrm{e}^{\mathrm{i}\omega x}$ on the interval $[-1,1]$, unitary rational best approximations of degree $(n,n)$ are found to be superior, achieving an error which is $2^{2n}$ times smaller for a fixed $n$ and $\omega$, or achieving the same accuracy for a given $\omega$ as an $(n,n)$-Pad\'e approximation achieves for $\omega/2$. While in practice this advantage seems to be maintained at reasonably large $\omega$, we are only able to prove this property in the asymptotic regime, $\omega \rightarrow 0^+$, where the behavior of the best approximations approaches that of the unitary rational interpolation at Chebyshev nodes. To this end we analyze the asymptotic errors of rational interpolants and the unitary best approximation for $\omega\to0^+$. For the latter, the leading-order term in its asymptotic expansion is also found to be a practical error bound in our numerical tests, even outside of an asymptotic regime.

A mild condition for the results on unitary best approximation presented in this manuscript to hold is that the approximation error needs to be strictly smaller than two, which certainly holds when considering accurate approximations. In particular, for a given $n$ it holds true if and only if $\omega \in (0, (n+1)\pi)$. Since $\mathrm{e}^{\mathrm{i}\omega x}$ for $x\in [-1,1]$ has a winding number of $\frac{\omega}{\pi}$, i.e. $\frac{\omega}{\pi}$ full oscillations or {\em wavelengths} in the interval $[-1,1]$, this corresponds to the restriction that a degree $(n,n)$ unitary rational approximant should not be used for approximating $n+1$ or more wavelengths. We prove that super-linear convergence begins for $\omega \leq 1.47(n+1/2)$, but in practice super-linear convergence seems to begin around $\omega \leq 2.94(n+1/2)$. Effectively, the best approximation goes from having the maximal error of two to super-linear convergence very shortly after gaining any amount of accuracy. For $\omega\geq (n+1)\pi$ all unitary rational functions are best approximations but achieve the maximum error of two. 

{\bf Continuity.} The unitary best approximation to $\mathrm{e}^{\mathrm{i}\omega x}$ is continuous in $\omega$ for $\omega \in (0, (n+1)\pi)$. Continuity of the rational approximation carries over to the poles, interpolation nodes and equioscillation points. In the limit $\omega \to 0^+$, the unitary best approximation converges to $r\equiv 1$. For $\omega=0$, this approximation is exact but degenerate, i.e. it has a degree $0$ which is strictly smaller than $n$. Due to the degeneracy at $\omega=0$, the poles, interpolation nodes and equioscillation points for $\omega \to 0^+$ cannot be studied directly by considering the $\omega=0$ case,  $r\equiv 1$. A similar observation holds for the limit $\omega \to (n+1)\pi^-$. The behavior of the poles, interpolation nodes and equioscillation points for $\omega \to (n+1)\pi^-$ cannot be studied directly by considering $\omega=(n+1)\pi$ since unitary best approximations are not uniquely defined in this case.

{\bf Poles, interpolation nodes and equioscillation points.} 
Poles of unitary best approximations are either real or come in complex conjugate pairs, are distinct, and live in the right complex plane, i.e., $\{z\in \mathbb{C}: \operatorname{Re} z >0\}$, for $\omega \in (0, (n+1)\pi)$, while its interpolation and equioscillation points are mirrored around the origin.

In the asymptotic limit $\omega \to 0^+$, the poles of any rational interpolant to $\mathrm{e}^{\mathrm{i} \omega x}$ (i.e. to $\mathrm{e}^{\omega z}$ with interpolation nodes on the imaginary axis) are unbounded, however, re-scaled by the factor $\omega$, they converge to the poles of the Pad\'{e} approximation. Since the unitary best approximation is also interpolatory at intermediate points between the points of equioscillation, this also applies to the poles of the unitary best approximations. We show that the interpolation nodes of the unitary best approximation converge to Chebyshev nodes (cf. Remark~\ref{rmk:ix}) in the limit $\omega \to 0^+$, justifying rational interpolation at Chebyshev nodes as an effective algorithm for $\omega$ in an asymptotic regime, i.e. the algorithm (i), as well as their use in algorithms (ii) and (iii) as support nodes and initial guess for interpolation nodes, respectively. 

In the other extreme, $\omega \to (n+1)\pi^-$, poles of the unitary best approximation converge to $\mathrm{i}(-1+2j/(n+1))$, for $j=1,\ldots,n$, while staying in the right complex plane. The zeros of the phase error and, equivalently the interpolation nodes for $r(\mathrm{i}x)\approx \mathrm{e}^{\mathrm{i}\omega x}$, converge to $-1+j/(n+1)$ for $j=1,\ldots,2n+1$. Of the $2n+2$ equioscillation points, two are at the boundaries $x=-1$ and $x=1$, and the remaining $2n$ are split into $n$ pairs, with the $j$th pair approaching $-1+2j/(n+1)$ while enclosing the $2j$th zero of the phase error from below and above. With two equioscillation points and a zero converging to $-1+2j/(n+1)$, for $j=1,\ldots,n$, the phase error approaches a sawtooth function. 

{\bf Symmetry and stability.}
The unitary best approximation has the property of symmetry,
\begin{equation}\label{eq:defsym}
r(-z) = r(z)^{-1},~~~ z\in \mathbb{C},
\end{equation}
and stability,
\begin{equation}\label{eq:rzless1leftplane}
|r(z)| < 1,~~~\text{for $z\in\mathbb{C}$ with $\operatorname{Re} z < 0$}. 
\end{equation}

\subsection{Relevance of unitary rational best approximations}

{\bf Geometric properties for time-integration.} 
An important application for the approximation to the exponential function is in the time-propagation of ordinary differential equations (ODEs), cf.~\cite{ML03}, including those arising from spatial discretization of initial-value problems in partial differential equations (PDEs). Rational approximations have some advantages over polynomial approximations in this context. Of particular interest to us is the fact that unitary rational functions exist. Being unitary~\eqref{eq:runitary} by design and {\em symmetric}~\cite{HW02} due to~\eqref{eq:defsym}, time-integrators based on unitary best approximations to the exponential function are suitable for geometric numerical integration~\cite{HWL02}. In particular, these are relevant for the time-integration of skew-Hermitian systems that occur in equations of quantum mechanics such as the Schr{\"o}dinger equation~\cite{Lu08} where they lead to conservation of unitarity, norm and energy.

Time-integrators based on unitary best approximations are also {\em A-stable} due to~\eqref{eq:rzless1leftplane}, making them suitable for weakly dissipative near-skew-Hermitian problems such as open quantum systems~\cite{DLTCZ19,RJ19} and the Schr{\"o}dinger equation in the presence of certain artificial boundary conditions~\cite{MPNE04}. 

Arguably the most well known unitary rational approximations are Pad{\'e} approximations to the exponential function. These are utilized widely in applications to skew-Hermitian cases, weakly dissipative cases, as well as dissipation dominated cases. Pad\'{e} approximations share all the geometric properties of unitary best approximations stated above. 

{\bf Uniform approximation.} 
While Pad{\'e} approximations are asymptotic in nature, with good approximation properties close to the origin at the expense of a poorer performance further from the origin, unitary best approximations are designed to achieve a prescribed level of error uniformly over a specified interval.

An alternative approach for achieving uniform error on an interval of interest is provided by Chebyshev polynomial approximations. However, polynomial approximations cannot conserve unitarity, norm or energy. Moreover, they are necessarily unbounded on the full imaginary axis, making them unstable in the presence of noise outside the interval of interest (due to numerical rounding errors, for instance).
In this sense, unitary best approximations combine the geometric properties of Pad\'{e} approximations with the flexibility of uniform approximation offered by Chebyshev polynomial approximations.

We emphasize that rational best approximation to $\mathrm{e}^{\mathrm{i}\omega x}$ on $x \in [-1,1]$ as well as rational interpolation to it at Chebyshev nodes should not be confused with Chebyshev rational approximation~\cite{CMV69} which yields uniform approximations to $\mathrm{e}^{-x}$ for $x\geq 0$. While the exponential on the left half of the real axis naturally decays for $x\to \infty$, and this behavior can be imitated by rational approximations, the exponential oscillates on the imaginary axis. Thus, in contrast to the real case where it is possible to find a best approximation to $\mathrm{e}^{-x}$ with uniform error on the half-line $x\in [0,\infty)$, we need to restrict ourselves to approximating the exponential $\mathrm{e}^z$ on a sub-interval of the imaginary axis, $z \in \mathrm{i}\omega[-1,1]$.

\medskip
{\bf Unitarity as a natural consequence to rational best approximation?}
While unitary rational approximation to the exponential on an interval on the imaginary axis are clearly beneficial for applications, we argue that restriction to unitarity is not critical in practice. Even when considering asymptotic best approximation to $\mathrm{e}^z$ around the origin, Pad\'e approximations end up being unitary without a specific restriction. A similar observation is made in the context of approximations to $\mathrm{e}^{\mathrm{i}\omega x}$ that  minimize a linearized error  or when considering rational interpolation at a specific number of points.
For instance, the AAA~\cite{NST18} and AAA--Lawson methods~\cite{NT20}, which (among other applications) aim to find uniform best approximations and were not designed with any considerations of unitarity, have been shown to generate unitary approximations~\cite{JS23}. The same has been shown for interpolation at $2n+1$ points by a degree $(n,n)$ rational function~\cite{JS23}. These observations suggest that unitarity might be a natural consequence of rational best approximation to $\mathrm{e}^{\mathrm{i}\omega x}$ on $x\in [-1,1]$.  

Moreover, we briefly sketch algorithms (i)--(iii) that make unitary rational approximations practical, achieving unitarity to machine precision and super-linear convergence rates. 

\subsection{Outline}

In Section~\ref{sec:unitaryratfct} we introduce the set of unitary rational functions $\mathcal{U}_n$ and show in Proposition~\ref{prop:unitaryridentities} their equivalence with rational functions that satisfy~\eqref{eq:runitary}.
In Subsection~\ref{subsec:someunitaryratfct}, we give some candidates for unitary rational approximations to $\mathrm{e}^{\mathrm{i} \omega x}$, including algorithms~(i)--(iii). The unitary best approximation in $\mathcal{U}_n$ is defined in Subsection~\ref{subsec:unitarybest}.
In Section~\ref{sec:exist} we show the existence of unitary best approximations.

In Section~\ref{sec:phaserror} we introduce the {\em phase function} and {\em phase error}, and establish some correspondences with the approximation error. In particular, in propositions~\ref{prop:approxertophaseerrinequ} and \ref{prop:omega0} we show that the approximation error of the best approximation in $\mathcal{U}_n$ to $\textrm{e}^{\textrm{i}\omega x}$ is non-maximal for frequencies $\omega\in(0,(n+1)\pi)$, and, in this frequency range, minimizing the approximation error corresponds to minimizing the phase error. 

In Section~\ref{sec:main} we state the equioscillation theorem, Theorem~\ref{thm:bestapprox}, which shows that the unitary best approximation is uniquely characterized by an equioscillating phase error, provided $\omega\in(0,(n+1)\pi)$. Moreover, Corollary~\ref{cor:errattainsmax} provides some results on the interpolation property of the unitary best approximation.

In Section~\ref{sec:symmetry} we show that unitary best approximations are symmetric (Proposition~\ref{prop:sym}), i.e., satisfy~\eqref{eq:defsym}, and state some consequences for its poles and the phase function in Corollary~\ref{cor:gsym}.
Moreover, we show that its interpolation and equioscillation points are mirrored around the origin due to symmetry (Proposition~\ref{prop:inodessym}).

In Section~\ref{sec:cont} we show that the unitary best approximation to $\mathrm{e}^{\mathrm{i} \omega x}$ depends continuously on $\omega $ for $\omega\in(0,(n+1)\pi)$ (Proposition~\ref{prop:minimizercontinuousonw}). Moreover, continuity carries over to the poles (Proposition~\ref{prop:polescont}), the phase function and phase error (Corollary~\ref{cor:phaseerrcont}), and the interpolation and equioscillation points (Proposition~\ref{prop:inodescontinuous}). Furthermore in Proposition~\ref{prop:eo1max} we show that, as a consequence of continuity, the phase error attains a maximum (i.e. with a positive sign) at the first equioscillation point.

In subsections~\ref{subsec:errbound} and \ref{subsec:asymerr} we introduce an error bound based on Pad\'e approximation (propositions~\ref{prop:padeerrorbound} and~\ref{prop:errorbound}), an asymptotically correct expression for the error under $\omega\to0^+$ based on rational interpolation at Chebyshev nodes (propositions~\ref{prop:asymerrcheb} and~\ref{prop:asymerrorbest}), and a practical error estimate~\eqref{eq:asymerrest}. Convergence results are illustrated with numerical examples in Subsection~\ref{subsec:plots}.

In Section~\ref{sec:asymprop} we study the poles, equioscillation points and interpolation nodes of unitary best approximations under the limits $\omega\to0^+$ (Subsection~\ref{subsec:limitwtozero}) and $\omega\to(n+1)\pi^-$ (Subsection~\ref{subsec:limitwnpipi}).  These properties are illustrated by numerical experiments in Subsection~\ref{subsec:plotslimit}.
In Section~\ref{sec:poles} we show that poles reside in the right complex plane for $\omega \in (0,(n+1)\pi)$ (Corollary~\ref{cor:polestoinf}), which results in stability~\eqref{eq:rzless1leftplane} (Proposition~\ref{prop:stable}). 

In Appendix~\ref{sec:appendixA} we state some results for rational interpolation to $\mathrm{e}^{\mathrm{i} \omega x}$ in the limit $\omega\to0^+$. These results are required to derive asymptotic error representations (propositions~\ref{prop:asymerrcheb} and~\ref{prop:asymerrorbest}), and results in Subsection~\ref{subsec:limitwtozero}.
In Appendix~\ref{sec:proofswtonp1pi} we prove propositions~\ref{prop:rforwtonp1pi},~\ref{prop:polesconvwtonp1pi} and~\ref{prop:nodesconvwtonp1pi}. These proofs are closely related to the convergence of the arc tangent function to a step function.

{\bf Equioscillation points and interpolation nodes.}
While properties of the poles are summarized in Section~\ref{sec:poles}, the interpolation nodes and equioscillation points have no section for their own. We provide a brief outline of the relevant results here. 
Existence of these points is due to the equioscillation theorem, Theorem~\ref{thm:bestapprox}, and Corollary~\ref{cor:errattainsmax}. 
In Proposition~\ref{prop:eo1max} we show that the phase error attains a maximum at the first equioscillation point, which further specifies the equioscillation property~\eqref{eq:defeo2}, which supersedes~\eqref{eq:defeo}. 
The interpolation and equioscillation points are mirrored around the origin due to symmetry (Proposition~\ref{prop:inodessym}), depend continuously on $\omega$ (Proposition~\ref{prop:inodescontinuous}), and their convergence in the limits $\omega\to0^+$ and $\omega\to(n+1)\pi^-$ is investigated in propositions~\ref{prop:inodestoCheb} and~\ref{prop:nodesconvwtonp1pi}, respectively.

\section{Unitary rational functions}\label{sec:unitaryratfct}
We consider rational functions of the form
\begin{equation}\label{eq:pdagpintro}
r(z) = 
\frac{p^\dag(z) }{p(z)},~~~\text{where $p^\dag(z):=\overline{p}(-z)$ for $z\in\mathbb{C}$}, \quad p \not \equiv 0,
\end{equation}
and $p$ is a complex polynomial of degree $\leq n$.
In particular, when $p$ is a polynomial of degree exactly $m\leq n$, and~$c\in\mathbb{C}$ and~$s_1,\ldots,s_m\in\mathbb{C}$ denote the pre-factor and zeros of $p$, respectively, then $p$ and~$p^\dag$ have the following forms:
\begin{equation}\label{eq:pandpdag}
p(z) = c \prod_{j=1}^m (z-s_j),~~~\text{and}~~
p^\dag(z)=\overline{p}(-z) = (-1)^m \overline{c} \prod_{j=1}^m (z+\overline{s}_j),~~~z\in\mathbb{C}.
\end{equation}

\begin{proposition}\label{prop:unitaryridentities}
The following statements are equivalent.
\begin{enumerate}[label={(\roman*)}]
\item\label{item:1requiv} $r\in\mathcal{R}_n$ is unitary, i.e., it satisfies~\eqref{eq:runitary}, where 
\begin{equation*}
\mathcal{R}_n := \left\{ \frac{p}{q} ~:~ \text{where $p$ and $q$ are polynomials of degree $\leq n$ and $q\not\equiv 0$}\right\}
\end{equation*}
denotes the set of all degree $(n,n)$ rational functions, 
\item\label{item:2requiv} $r\in \mathcal{U}_n$, where
\begin{equation*}
\mathcal{U}_{n} := \left\{\frac{p^\dag}{p} ~:~ \text{where $p$ is a polynomial of degree $\leq n$ and $p\not\equiv 0$}\right\},
\end{equation*}
and
\item\label{item:3requiv} $r$ is a rational function of the form
\begin{equation}\label{eq:defrbyp}
r(z) = (-1)^m \mathrm{e}^{\mathrm{i}\theta}  \prod_{j=1}^{m} \frac{z+\overline{s}_j}{z-s_j},~~~z\in\mathbb{C},
\end{equation}
where $\theta\in\mathbb{R}$ corresponds to the phase, and $s_1,\ldots,s_m\in\mathbb{C}$ denote the poles of $r$ with~$m\leq n$.
\end{enumerate}
In particular, $\mathcal{U}_n \subseteq \mathcal{R}_n$ is the set of all unitary rational functions.
\end{proposition}

\medskip
We briefly introduce some notation before providing a proof to this proposition.
The zeros and poles of a rational function $r=p/q$ refer to the zeros of $p$ and $q$, respectively.
In the sequel we refer to a rational function $r$ as \textit{irreducible} when its numerator and denominator have no common zeros. Since such zeros correspond to removable singularities of $r$, we may assume rational functions to be given in irreducible form throughout the present work. 
Furthermore, the \textit{minimal degree} of $r\in\mathcal{R}_n$ refers to the smallest degree $m\leq n$ with~$r\in\mathcal{R}_m$ (thus,~$r\notin \mathcal{R}_{m-1}$).
In particular, any~$r\in\mathcal{R}_n$ with minimal degree~$n$ is irreducible.

\begin{proof}[Proof of Proposition~\ref{prop:unitaryridentities}] We first show that~\ref{item:2requiv} implies~\ref{item:3requiv}. Let $r=p^\dag/p\in\mathcal{U}_n$ where $p$ is a polynomial of degree exactly $m\leq n$. Let~$c\in\mathbb{C}$ and~$s_1,\ldots,s_m\in\mathbb{C}$ denote the pre-factor and zeros of $p$, respectively. Then $p$ and~$p^\dag$ correspond to~\eqref{eq:pandpdag} and $r=p^\dag/p$ corresponds to~\eqref{eq:defrbyp}, with $ \mathrm{e}^{\mathrm{i}\theta} =  \overline{c}/c $ for $\theta \in \mathbb{R}$. In particular, $p$ and $p^\dag$ are both polynomials of degree $m \leq n$ and $r=p^\dag/p$ is a rational function.

We proceed to show that~\ref{item:3requiv} implies~\ref{item:1requiv}, i.e., rational functions $r$ of the form~\eqref{eq:defrbyp} are unitary. For $z=\mathrm{i}x$ each fraction in~\eqref{eq:defrbyp} maps to the unit circle since $|\mathrm{i}x+\overline{s}_j| = |\mathrm{i}x - s_j|$, which proves this assertion.

We complete our proof by showing that~\ref{item:1requiv} implies~\ref{item:2requiv}.
Let $r=\widetilde{p}/\widetilde{q}\in\mathcal{R}_n$ be unitary.
We recall $ \overline{r}(-\mathrm{i}x)= \overline{r(\mathrm{i}x)}$ for $x\in\mathbb{R}$ which implies $r^\dag(\mathrm{i}x) = \overline{r(\mathrm{i}x)}$, where $r^\dag = \widetilde{p}^\dag/\widetilde{q}^\dag$. From unitarity \eqref{eq:runitary} we conclude that $r$ has no non-removable singularities on the imaginary axis, and since $r$ may be assumed to be given in an irreducible form, $\widetilde{q}$ has no zeros on the imaginary axis. This carries over to $\widetilde{q}^\dag$ as well.
Thus, unitarity corresponds to $r^\dag r =1$ and yields
\begin{equation}\label{eq:unitarytopdp}
\widetilde{p}^\dag(\mathrm{i}x) \widetilde{p}(\mathrm{i}x) = \widetilde{q}^\dag(\mathrm{i}x) \widetilde{q}(\mathrm{i}x),~~~\text{for $x\in\mathbb{R}$.}
\end{equation}
We assume that $\widetilde{p}$ and $\widetilde{q}$ have no common zeros. Any point $z\in\mathbb{C}$ is a zero of $\widetilde{p}$ if and only if $-\overline{z}$ is a zero of $\widetilde{p}^\dag$, which directly follows from~\eqref{eq:pandpdag}. The zeros of $\widetilde{q}$ and $\widetilde{q}^\dag$ satisfy a similar relation. Thus, it carries over from $\widetilde{p}$ and $\widetilde{q}$ that $\widetilde{p}^\dag$ and $\widetilde{q}^\dag$ have no common zeros.

The identity~\eqref{eq:unitarytopdp} entails that any zero of $\widetilde{p}$ is either a zero of $\widetilde{q}$ or $\widetilde{q}^\dag$. Since we assume that $\widetilde{p}$ and $\widetilde{q}$ have no common zeros, the latter holds true. In a similar manner, any zero of $\widetilde{q}^\dag$ is either a zero of $\widetilde{p}^\dag$ or $\widetilde{p}$, and with absence of common zeros of $\widetilde{p}^\dag$ and $\widetilde{q}^\dag$, the latter holds true. Thus, the zeros of $\widetilde{p}$ are exactly the zeros of $\widetilde{q}^\dag$. 
Let $s_1,\ldots,s_m$ denote the zeros of the denominator $\widetilde{q}$ for $m\leq n$, then the poles and zeros of $r=\widetilde{p}/\widetilde{q}$ correspond to $s_1,\ldots,s_m$ and $-\overline{s}_1,\ldots,-\overline{s}_m$, respectively.

Let $\widetilde{c}_1,\widetilde{c}_2 \in \mathbb{C}$ be the pre-factors of $\widetilde{p}$ and $\widetilde{q}$, respectively, i.e., $\widetilde{p}(z) =  \widetilde{c}_1 \prod_{j=1}^m (z+\overline{s}_j)$ and $\widetilde{q}(z)= \widetilde{c}_2 \prod_{j=1}^m (z-s_j)$. Substituting $z=0$, we observe $\widetilde{p}(0) =  \widetilde{c}_1 \prod_{j=1}^m \overline{s}_j$ and $\widetilde{q}(0)= (-1)^m \widetilde{c}_2 \prod_{j=1}^m s_j$. Unitarity of $r$ and the fact that $\widetilde{q}$ has no zeros on the imaginary axis implies that $|\widetilde{p}(0)|=|\widetilde{q}(0)| \neq 0$, and thus, $|\widetilde{c}_1|=|\widetilde{c}_2|$. As a consequence there exists $\psi\in\mathbb{R}$ with $\mathrm{e}^{\mathrm{i}\psi }= c_1/c_2$. Then, we may define a polynomial $p$ with zeros $s_1,\ldots,s_m$ and a pre-factor $c = \mathrm{e}^{-\mathrm{i}\psi/2}$, such that $r=\widetilde{p}/\widetilde{q} = p^\dag/ p \in\mathcal{U}_n$.
\end{proof}

\subsection*{Poles and zeros of unitary rational functions}
Whether a unitary rational function is irreducible is already specified by its poles.
\begin{proposition}\label{prop:fulldegree}
A unitary rational function $r\in \mathcal{U}_n$ is irreducible if and only if its poles satisfy $\overline{s}_j \neq -s_\ell$ for $j,\ell \in\{ 1,\ldots,m\}$ where $m\leq n$ denotes the number of poles of $r$.

In particular, a point $z\in\mathbb{C}$ is a pole of $r$ if and only if $-\overline{z}$ is a zero of $r$, and all poles of $r$ on the imaginary axis correspond to removable singularities.
\end{proposition}
\begin{proof}
This follows directly from~\eqref{eq:defrbyp}.
\end{proof}

\begin{remark}\label{rmk:cayley}
In the simplest case a unitary rational function $r\in\mathcal{U}_n$ with minimal degree $n=1$ is a $(1,1)$ rational function of the form
\begin{subequations}\label{eq:cayley12}
\begin{equation}
\label{eq:cayley1}
r(z) = -\frac{z+\overline{s}}{z-s}, \qquad \overline{s} \neq -s,
\end{equation}
which corresponds to a Cayley transform.
Considering $s=\xi+\textrm{i}\mu$, an equivalent representation for $r$ is
\begin{equation}
\label{eq:cayley2}
r(z) = \frac{1+(z-\textrm{i}\mu)/\xi}{1-(z-\textrm{i}\mu)/\xi}, \qquad \overline{s} \neq -s.
\end{equation}
\end{subequations}
In the context of Lie group methods for time-integration of skew-Hermitian problems, the Cayley transform commonly appears with $s=2$, i.e.  $\xi=2$ and $\mu=0$~\cite{Is01,MO01}.

Thus, due to~\eqref{eq:defrbyp} in Proposition~\ref{prop:unitaryridentities}, any unitary rational function of minimal degree $m$ can be expressed as a product of $m$ Cayley transforms~\eqref{eq:cayley1}, with an additional phase factor.  
\end{remark}

\subsection{Some unitary rational approximations}\label{subsec:someunitaryratfct}

\subsubsection{Diagonal Pad\'e approximation}\label{subsec:intropade}

We recall some classical results and notation concerning the diagonal Pad\'{e} approximation to $\mathrm{e}^{z}$, mostly referring to~\cite{BG96}. 
The diagonal Pad\'{e} approximation of degree $n$ to $\mathrm{e}^{z}$, i.e., the $(n,n)$-Pad\'{e} approximation, is a rational function $\widehat{r}=p/q$, where $p$ and $q$ are polynomials of degree~$\leq n$, s.t.\ $\widehat{r}(z)$ and $\mathrm{e}^{z}$ agree with the highest possible asymptotic order at $z=0$, namely,
\begin{equation}\label{eq:defPadepq}
\widehat{r}(z) := \frac{p(z)}{q(z)} = \mathrm{e}^{z} + \mathcal{O}(|z|^{2n+1}),~~~|z|\to 0.
\end{equation}
This approximation exists for all degrees $n$ and, when $q$ is normalized, e.g.,~$q(0)=1$, the diagonal Pad\'e approximation is unique,~\cite[Section~1.2]{BG96} for instance. For an explicit representation of $p$ and $q$ we refer to~\cite[eq.~(10.23)]{Hi08} or~\cite[Theorem~3.11]{HW02}. Since $p$ and $q$ are real polynomials with $p(z)=q(-z)$ we may write $\widehat{r}=\widehat{p}^\dag/\widehat{p}$, where $\widehat{p} := q$.
Thus,~\eqref{eq:defPadepq} corresponds to
\begin{equation}\label{eq:defPade}
\widehat{r}(z) = \frac{\widehat{p}^\dag(z)}{\widehat{p}(z)} = \mathrm{e}^{z} + \mathcal{O}(|z|^{2n+1}),~~~|z|\to 0.
\end{equation}
In particular, $\widehat{r}\in\mathcal{U}_n$ and $\widehat{r}$ has minimal degree $n$.

\subsubsection{Rational approximations based on polynomial approximations}\label{subsec:ratapproxfrompolcheb}
When considering uniform approximations to $\mathrm{e}^{\mathrm{i} \omega x}$ over $x\in[-1,1]$, the polynomial Chebyshev approximation might be the most common approach, cf.~\cite[Subsection~III.2.1]{Lu08}. While polynomial methods are not directly suitable when unitarity or related stability properties are desired, quotients of polynomial approximations can yield a unitary rational approximation. As an example, we consider polynomial Chebyshev approximations. In particular, a candidate for a unitary approximation to $\mathrm{e}^{\mathrm{i} \omega x}$ is
\begin{equation}\label{eq:rfrompcheb}
\check{r}:=\frac{p^\dag}{p},~~~\text{where $p$ is the polynomial Chebyshev approximation $p(\mathrm{i} x)\approx\mathrm{e}^{-\mathrm{i} \omega/2 x}$.}
\end{equation}
Certainly, when $p$ is the polynomial Chebyshev approximation of degree $n$, then $\check{r}=p^\dag/p\in\mathcal{U}_n$. Furthermore, the polynomial $p^\dag$ approximates $p^\dag(\mathrm{i} x) = \overline{p(\mathrm{i} x)}\approx\mathrm{e}^{\mathrm{i} \omega/2 x}$ and thus, $\check{r}=p^\dag/p$ approximates $\check{r}(\mathrm{i} x) \approx \mathrm{e}^{\mathrm{i} \omega x}$. While~\eqref{eq:rfrompcheb} provides a straightforward way for constructing unitary rational approximations with uniform error on an interval on the imaginary axis, it performs poorly compared to other rational approximations discussed in the present work, as evident through the numerical experiments in Section~\ref{sec:convergence}.

\subsubsection{Rational interpolation}\label{subsec:ratint}

Rational interpolation is also referred to as multipoint Pad\'e or Newton--Pad\'e approximation in the literature.
We first consider the case of distinct interpolation nodes, which is most relevant in the present work with the exception of Appendix~\ref{sec:appendixA}, and we recall some classical results which are thoroughly discussed in~\cite[Section~2]{Be70} and~\cite{MW60b, Gu90}, among others. Then, we briefly generalize these results allowing interpolation nodes of higher multiplicity, mostly referring to~\cite{Sa62, Cla78}. We refer to the general setting as {\em osculatory} rational interpolation, and we show that interpolants are unitary at the end of the present subsection.

\smallskip
{\bf Rational interpolation with distinct nodes.} We recall the previously introduced notation $\mathcal{R}_n$ for rational functions $r=p/q$ where $p$ and $q$ are polynomials of degree $\leq n$. Since functions in $\mathcal{R}_n$ are specified by $2n+1$ degrees of freedom, considering candidates in $\mathcal{R}_n$ for interpolation at $2n+1$ nodes seems natural. Let $\mathrm{i} x_1,\ldots,\mathrm{i} x_{2n+1}\in\mathrm{i}\mathbb{R}$ denote distinct and given interpolation nodes for the respective rational interpolation problem to $\mathrm{e}^{\omega z}$, i.e. finding $r\in\mathcal{R}_n$ s.t.
\begin{subequations}\label{eq:ratintdistinctnodes}
\begin{equation}\label{eq:ratintpq}
r(\mathrm{i} x_j)=\mathrm{e}^{\mathrm{i} \omega x_j},~~~j=1,\ldots,2n+1.
\end{equation}
Provided such $r$ exists, the solution is unique in $\mathcal{R}_n$.
However, this interpolation problem may have no solution.
To overcome this issue it is a common approach to consider a linearized interpolation problem instead, i.e., find polynomials $p$ and $q$ of degree $\leq n$ s.t.
\begin{equation}\label{eq:ratintpqlin}
p(\mathrm{i} x_j) = \mathrm{e}^{\mathrm{i} \omega x_j }q(\mathrm{i} x_j),~~~~j=1,\ldots,2n+1.
\end{equation}
\end{subequations}
This formulation originates from~\eqref{eq:ratintpq} when substituting $r=p/q$ therein and multiplying by $q$.
For a solution $r=p/q$ of~\eqref{eq:ratintpq}, the polynomials $p$ and $q$ also solve the linearized formulation~\eqref{eq:ratintpqlin}. 

Since the linearized system~\eqref{eq:ratintpqlin} has one more degree of freedom than equations, we find at least one non-trivial pair of polynomials $p$ and $q$ which solves the linearized interpolation problem. While $p$ and $q$ might not be a unique solution, for all solutions of~\eqref{eq:ratintpqlin} the quotient $p/q$ represents the same rational function. In particular, $p/q$ from~\eqref{eq:ratintpqlin} coincides with the solution of~\eqref{eq:ratintpq} if the latter exists, in which case these problems are equivalent. Otherwise, there exists at least one interpolation node for which the rational function $p/q$ from~\eqref{eq:ratintpqlin} does not satisfy~\eqref{eq:ratintpq}, i.e., $p(\mathrm{i} x_j)/q(\mathrm{i} x_j)\neq \mathrm{e}^{\mathrm{i}\omega x_j}$, and these nodes are referred to as {\em unattainable nodes}. Moreover, for all solutions $p$ and $q$ of~\eqref{eq:ratintpqlin}, the unattainable nodes are common zeros of $p$ and $q$.

\medskip
{\bf Osculatory rational interpolation.}
We briefly remark that rational interpolation with interpolation nodes of higher multiplicity may be treated in a similar way as the distinct nodes case.
Let $\mathrm{i} x_1,\ldots,\mathrm{i} x_{2n+1}$ denote interpolation nodes with $x_j\in\mathbb{R}$ for rational interpolation to $\mathrm{e}^{\omega z}$, allowing nodes of higher multiplicity.
In particular, assume this sequence of nodes consists of $m\leq 2n+1$ distinct nodes. To simplify the notation we introduce a mapping $\iota$ s.t.\ the index $\iota(j)$ corresponds to the $j$th distinct node which has multiplicity $n_j\geq 1$, i.e.,
\begin{equation*}
\{x_1,\ldots,x_{2n+1}\} = \{
\underbrace{x_{\iota(1)},\ldots,x_{\iota(1)}}_{\text{$n_1$ many}},
\underbrace{x_{\iota(2)},\ldots,x_{\iota(2)}}_{\text{$n_2$ many}},\ldots, \underbrace{x_{\iota(m)},\ldots,x_{\iota(m)}}_{\text{$n_m$ many}}\},
\end{equation*}
and $n_1+\ldots+n_m =2n+1$. Thus, this notation also covers the case of distinct interpolation nodes, i.e., $m=2n+1$ and $n_j=1$, and the following formulations for the interpolation problem have to be understood as a generalization of~\eqref{eq:ratintdistinctnodes}. The osculatory interpolation problem to $\mathrm{e}^{\omega z}$ consists of finding $r\in\mathcal{R}_n$ s.t.
\begin{subequations}\label{eq:ratintcondconfluentboth}
\begin{equation}\label{eq:ratintcondconfluent}
r^{(\ell)}\left(\mathrm{i} x_{\iota(j)}\right) = 
 \left.\frac{\mathrm{d}^{\ell} \mathrm{e}^{\omega z}}{\mathrm{d} z^{\ell}}\right|_{z=\mathrm{i} x_{\iota(j)}},
~~~\ell=0,\ldots,n_j-1,~~~j=1,\ldots,m.
\end{equation}
Moreover, the linearized interpolation problem consists of finding polynomials $p$ and $q$ of degree $\leq n$ s.t.
\begin{equation}\label{eq:ratintcondconfluentlin}
p^{(\ell)}\left(\mathrm{i} x_{\iota(j)}\right) = \left.\frac{\mathrm{d}^{\ell} (q(z) \mathrm{e}^{\omega z})}{\mathrm{d} z^{\ell}}\right|_{z=\mathrm{i} x_{\iota(j)}},~~~\ell=0,\ldots,n_j-1,~~~j=1,\ldots,m.
\end{equation}
Equivalently, we may also replace~\eqref{eq:ratintcondconfluentlin} by
\begin{equation}\label{eq:ratintcondconfluentlin2}
p(z) - q(z) \mathrm{e}^{\omega z} = \prod_{j=1}^{2n+1}(z-\mathrm{i} x_j) v(z),
\end{equation}
\end{subequations}
where $v$ denotes an analytic function. Following~\cite{Sa62,Cla78}, the interpolation problems in~\eqref{eq:ratintcondconfluentboth} are equivalent.
Similar to the case of distinct interpolation nodes, unattainable nodes (possibly of higher multiplicity) may occur and~\eqref{eq:ratintcondconfluent} has no solution in this case, however, a unique solution may be defined via the linearized interpolation problem.

\medskip
{\bf Unitarity.}
In the sequel, rational interpolation refers to the unique rational function $p/q$ which originates from solutions of the linearized problem~\eqref{eq:ratintcondconfluentlin2}, or equivalently~\eqref{eq:ratintcondconfluentlin} which simplifies to~\eqref{eq:ratintpqlin} in case of distinct nodes, possibly featuring unattainable points. As shown in the following proposition, the rational interpolant to $\mathrm{e}^{\omega z}$ is unitary, i.e., $r=p/q\in\mathcal{U}_n$. For the case of distinct interpolation nodes, a similar result also appeared in~\cite[Proposition~2.1]{JS23} earlier.
\begin{proposition}\label{prop:ratintconfluentunitary}
Provided the interpolation nodes $\mathrm{i} x_1,\ldots,\mathrm{i} x_{2n+1}$ are on the imaginary axis, the rational interpolant $r\in\mathcal{R}_n$ to $\mathrm{e}^{\omega z}$ in the sense of~\eqref{eq:ratintcondconfluentboth} is unitary, i.e., $r\in\mathcal{U}_n$.
\end{proposition}
\begin{proof}
The proof of this proposition is provided in Appendix~\ref{sec:appendixA}.
\end{proof}

\subsubsection{Rational interpolation at Chebyshev nodes -- algorithm (i)}
\label{subsec:ratintcheb}
We introduce $\mathring{r}\in\mathcal{U}_n$ as the rational interpolant to $\mathrm{e}^{\mathrm{i} \omega x}$ at $2n+1$ Chebyshev nodes, i.e.,
\begin{equation}\label{eq:ratinterpolateCheb}
\mathring{r}(\mathrm{i} \tau_j) = \mathrm{e}^{\mathrm{i} \omega \tau_j},~~~\text{where}~~\tau_j = \cos\left(\frac{(2j-1)\pi}{2(2n+1)}\right),~~~j=1,\ldots,2n+1.
\end{equation}
In practice, $\mathring{r}$ can be computed in barycentric rational form with a coefficient vector in the null-space of a $n\times (n+1)$-dimensional {\em Loewner matrix}~\cite{Be70,AA86,Kno08}. We propose a modified procedure which achieves unitarity to machine precision. Namely, the desired rational interpolant is found by dividing the $2n+1$ Chebyshev nodes into $n+1$ {\em support nodes} and $n$ {\em test nodes} and utilizing the rotated Loewner matrix approach outlined in~\cite{JS23}. For a Python implementation we refer to~\cite{Ja23}.
For an algorithm which computes rational interpolants at Chebyshev nodes utilizing Fourier transform see~\cite{PGD12}.

Rational interpolation at Chebyshev nodes has some relevance for our convergence results for unitary best approximations in Section~\ref{sec:convergence}, and potentially yields a simple but effective approach to approximating $\mathrm{e}^{\mathrm{i} \omega x}$.

\subsubsection{A modified BRASIL algorithm -- algorithm (iii)}\label{subsec:brasil}

The BRASIL (best rational approximation by successive interval length adjustment) algorithm~\cite{Ho21} aims to compute rational best approximations to real functions. Such approximations have equioscillating approximation errors and attain interpolation nodes. The original BRASIL algorithm starts with an initial guess for the interpolation nodes and computes an approximation via rational interpolation. Then, the nodes (respectively, the intervals spanned by neighboring interpolation nodes) are adjust s.t.\ the approximation error approaches an equioscillatory behavior, and this process is repeated iteratively.

We briefly sketch here a modified BRASIL algorithm to compute the unitary best approximation to $\mathrm{e}^{\mathrm{i} \omega x}$. This algorithm is introduced in full detail in~\cite{Ja23}. In the present work (in particular, Theorem~\ref{thm:bestapprox} and Corollary~\ref{cor:errattainsmax}) we show that the unitary best approximation to $\mathrm{e}^{\mathrm{i} \omega x}$ has an equioscillating {\em phase error} and is interpolatory. Similar to the BRASIL algorithm, we use an initial guess for the interpolation nodes, compute an approximation by solving the respective rational interpolation problem, and adjust the interpolation nodes s.t.\ the phase error approaches an equioscillatory behavior. This process is repeated iteratively until the phase error is nearly equioscillatory. In particular, for numerical tests in the present work we compute unitary best approximations which have an equioscillating phase error up to a relative deviation of $\leq 10^{-3}$. The generated approximations correspond to rational interpolants, and thus, are unitary (Subsection~\ref{subsec:ratint}). In particular, we use rotated Loewner matrices~\cite{JS23} to conserve unitarity to machine precision. Moreover, some modifications can be applied to conserve symmetry properties which are discussed in Section~\ref{sec:poles}.

Furthermore, in Proposition~\ref{prop:inodestoCheb} we show that interpolation nodes of the unitary best approximation converge to Chebyshev nodes in the limit $\omega\to0^+$. This motivates the use of Chebyshev nodes as initial guess for the interpolation nodes, and results in fast convergence for sufficiently small $\omega$. In a similar manner, uniformly spread nodes can be used as an initial guess near the other extreme $\omega\to(n+1)\pi^-$, see Proposition~\ref{prop:nodesconvwtonp1pi}. Moreover, when computing unitary best approximations for different values of $\omega$, the interpolation nodes of the best approximation for a nearby value of $\omega$ can be used as initial guesses since interpolation nodes depend continuously on $\omega$, as shown in Proposition~\ref{prop:inodescontinuous}.

\subsubsection{AAA and AAA--Lawson algorithms}

The AAA algorithm \cite{NST18} computes rational approximations in a barycentric form by minimizing a linearized error over a set of {\em test nodes} using a least squares procedure. It adaptively chooses {\em support nodes} in a greedy way to keep a uniform error small. The AAA--Lawson algorithm~\cite{NT20} first runs the AAA algorithm to choose support nodes. Then it utilizes a Lawson iteration by minimizing weighted least square errors in order to minimize an error over test nodes, which in the present setting can be understood as a discretized version of a uniform error over an interval.

When used in the context of approximating the exponential function on an interval on the imaginary axis \eqref{eq:rpqapproxexp}, AAA and AAA--Lawson methods produce unitary approximations~\cite{JS23}. In practice, due to finite precision arithmetic, the unitarity of these approximations deteriorates and we suggest using modified versions of AAA and AAA--Lawson outlined in~\cite{JS23} which achieve unitarity to machine precision.

\subsubsection{AAA--Lawson with Chebyshev support nodes -- algorithm (ii)}\label{subsec:LawsonCheb}

The AAA--Lawson method~\cite{NT20} aims to minimize a uniform error over an interval, and in practice comes reasonably close to best approximations, which are the topic of the present work. However, the adaptive choice of support nodes through the AAA iterations can violate certain symmetries of the poles of the best approximations that are shown later in Section~\ref{sec:poles}.  

We propose another variant of the AAA--Lawson method for~\eqref{eq:rpqapproxexp}. In addition to using the modified (unitary) algorithm from~\cite{JS23}, for an approximant of degree $(n,n)$ we fix the $n+1$ support nodes to be Chebyshev nodes. Thus, in contrast to the original AAA--Lawson method the support nodes are not computed one-by-one by initial AAA iterations. For a Python implementation see~\cite{Ja23}.

The approximants generated by this procedure seem to demonstrate an equi\-os\-cil\-la\-tion property in practice (superior to the AAA--Lawson method without pre-assigning support nodes), which characterizes best approximations, as shown in Theorem~\ref{thm:bestapprox}. 

\subsection{Unitary rational best approximations}\label{subsec:unitarybest}
In the present work we are concerned with unitary rational approximations to the exponential function on the imaginary axis, 
\begin{equation*} 
\tag{\ref{eq:rpqapproxexp}}
r(\mathrm{i} x) = \frac{p^\dag(\mathrm{i} x) }{p(\mathrm{i} x)} \approx \mathrm{e}^{\mathrm{i} \omega x},~~~x\in [-1,1]\subset\mathbb{R},~~ \omega>0,~~p\not\equiv 0.
\end{equation*}
To simplify our notation, we introduce the norm
\begin{equation*}
\|f\| := \max_{x\in[-1,1]} |f(\mathrm{i} x)|,
\end{equation*}
where $f$ is a complex function. In particular, for $r(\mathrm{i} x) \approx \mathrm{e}^{\mathrm{i} \omega x}$ as in~\eqref{eq:rpqapproxexp} the uniform error reads
\begin{equation}\label{eq:errnormnotation}
\|r-\exp(\omega \cdot) \| = \max_{x\in[-1,1]} |r(\mathrm{i} x)-\mathrm{e}^{\mathrm{i} \omega x}|.
\end{equation}
The main topic of the present work are unitary rational best approximations, i.e., rational functions $r\in\mathcal{U}_n$ that minimize the error 
\begin{equation}\label{eq:unitaryChebbestapprox}
\|r-\exp(\omega \cdot) \|
= \inf_{u\in \mathcal{U}_n} \|u-\exp(\omega \cdot) \|,
\end{equation}
where $n$ is a given degree and $\omega>0$.
Best approximations of the type~\eqref{eq:unitaryChebbestapprox} are also known as Chebyshev approximations in the literature.

For the exponential function this also provides best approximations~$r(\mathrm{i} y) \approx \mathrm{e}^{\mathrm{i} y}$ for~$y\in[a,b]$ where~$[a,b]\subset\mathbb{R}$ is an arbitrary interval. In particular, let~$\omega = (b-a)/2$, $\gamma=(a+b)/2$, and~$\widetilde{r}(\mathrm{i} y ) := \mathrm{e}^{\mathrm{i} \gamma}r(\mathrm{i} (y-\gamma)/\omega)$ for a given~$r\in\mathcal{U}_n$. Then
\begin{equation}\label{eq:fromm11toab}
\max_{y\in [a,b]} | \widetilde{r}(\mathrm{i} y ) - \mathrm{e}^{\mathrm{i} y}|
= \max_{y\in [a,b]} | \mathrm{e}^{\mathrm{i} \gamma} r(\mathrm{i} (y-\gamma)/\omega) - \mathrm{e}^{\mathrm{i} y}|
= \max_{x\in [-1,1]} | r(\mathrm{i} x) - \mathrm{e}^{\mathrm{i} \omega x}|,
\end{equation}
since $(y-\gamma)/\omega=x\in[-1,1]$ for $y\in[a,b]$. The approximation $\widetilde{r}(\mathrm{i} y) \approx \mathrm{e}^{\mathrm{i} y}$ is a unitary best approximation for $y\in[a,b]$ if and only if the underlying $r$ is a unitary best approximation~$r(\mathrm{i} x)\approx \mathrm{e}^{\mathrm{i} \omega x}$ for~$x\in[-1,1]$. 

\section{Existence}\label{sec:exist}

\begin{proposition}\label{prop:convsubsequence} 
For a fixed degree $n$, every sequence $\{r_j\in\mathcal{U}_n\}_{j\in\mathbb{N}}$ has a sub-sequence which converges to some $r\in\mathcal{U}_n$ in the sense of $|r_{j_k}(\mathrm{i} x) - r(\mathrm{i} x)|\to 0$ for $x\in[-1,1]$ point-wise except for at most $n$ points.
\end{proposition}
\begin{proof} The proof for this proposition follows along the lines of~\cite{Wa31,Tre13}. There exist polynomials $p_j$ of degree $\leq n$ with $\|p_j\|=1$ and $r_j=p_j^\dag/p_j$. Since the sequence $p_j$ is bounded in the set of polynomials, there exists a convergent sub-sequence $p_{j_k}\to \widetilde{p}$ where $\widetilde{p}$ is also a polynomial of degree $\leq n$ with $\|\widetilde{p}\|=1$. We define $\widetilde{r} := \widetilde{p}^\dag/\widetilde{p}\in\mathcal{U}_n$. For $x\in[-1,1]$ s.t.\ $\widetilde{p}(\mathrm{i} x)\neq 0$, we get
\begin{equation*}
|r_{j_k}(\mathrm{i} x) - \widetilde{r}(\mathrm{i} x)|
= \left| \frac{p_{j_k}^\dag(\mathrm{i} x)}{p_{j_k}(\mathrm{i} x)} - \frac{\widetilde{p}^\dag(\mathrm{i} x)}{\widetilde{p}(\mathrm{i} x)} \right|\to 0,
\end{equation*}
which shows point-wise convergence for $x\in[-1,1]$ except for the zeros of $\widetilde{p}$, i.e., with the exception of at most $n$ points.
While $r_{j_k}$ might not convergence to $\widetilde{r}$ at the zeros of $\widetilde{p}$, the zeros of $\widetilde{p}$ on the real axis correspond to removable singularities of $\widetilde{r}$. 
\end{proof}

\begin{proposition}[Existence]\label{prop:bestapproxinUnexists}
For given degree $n$ and $\omega>0$, there exists a unitary best approximation $r\in\mathcal{U}_n$, i.e.,
\begin{equation*}
\| r - \exp(\omega \cdot)\|
= \inf_{u \in\mathcal{U}_n} \| u - \exp(\omega \cdot)\|.
\end{equation*}
Due to the existence of the best approximation, we also refer to the infimum as minimum in the sequel.
\end{proposition}
\begin{proof}
Since $\|r-\exp(\omega\cdot)\|$ for $r\in\mathcal{U}_n$ is bounded from below, the following infimum exists,
\begin{equation*}
E := \inf_{u\in U_k} \| u - \exp(\omega \cdot)\|.
\end{equation*}
Consider a minimizing sequence $\{r_j\in\mathcal{U}_n\}_{j\in\mathbb{N}}$, i.e.,
\begin{equation*}
\lim_{j\to \infty} \| r_j - \exp(\omega \cdot)\|  = E.
\end{equation*}
Following Proposition~\ref{prop:convsubsequence}, the sequence $r_j$ has a sub-sequence~$r_{j_k}$ which converges to some~$r\in\mathcal{U}_n$ point-wise with the exception of at most $n$ points. For the points~$x$ where $r_{j_k}$ converges, we get
\begin{equation*}
|r(\mathrm{i} x) - \mathrm{e}^{\mathrm{i} \omega x}|
= \lim_{k\to \infty} |r_{j_k}(\mathrm{i} x) - \mathrm{e}^{\mathrm{i} \omega x}|
 \leq E. 
\end{equation*}
Due to continuity we conclude $\| r - \exp(\omega\, \cdot) \| = E$, and thus, $r$ attains the minimal error.
\end{proof}

\section{Approximation error and phase error}\label{sec:phaserror}

In the present section we first show that $r\in\mathcal{U}_n$ satisfies the representation $r(\mathrm{i} x) = \mathrm{e}^{\mathrm{i} g(x)}$ for the {\em phase function} $g$ which is specified below. This representation has also appeared in a previous work on unitary approximations~\cite{Hu10}.
\begin{definition}\label{def:errors}
For the unitary approximation, $r(\mathrm{i} x)= \mathrm{e}^{\mathrm{i} g(x)}\approx \mathrm{e}^{\mathrm{i} \omega x}$, we define the  approximation error as $r(\mathrm{i} x)-\mathrm{e}^{\mathrm{i} \omega x}$ and the phase error as $g(x) - \omega x$. We say that the approximation error is maximal if  $\|r -\exp(\omega \cdot)\|=2$.
\end{definition}

The connection between the {approximation error} and the {phase error}, which will be established in this section, is a crucial tool in our study of unitary best approximations. In particular, we show that the phase is uniquely defined and a smaller phase error corresponds to a smaller approximation error, provided the approximation error is not maximal, or equivalently, the phase error satisfies $\max_{x \in [-1,1]} |g(x) - \omega x|<\pi$ (cf. Proposition~\ref{prop:approxertophaseerrinequ}), which holds for a large range of frequencies, $\omega \in (0, (n+1)\pi)$ (cf. Proposition~\ref{prop:omega0}). Moreover, we prove that the error of the unitary best approximation is continuous in the frequency $\omega$ (cf. Proposition~\ref{prop:wtomincontinuous}).

We start by considering the relation between the phase function $g(x)$ and the unitary approximation $r(\mathrm{i} x)$ by studying the simplest unitary function -- the Cayley transform~\eqref{eq:cayley12}, which is a degree $(1,1)$ rational function. The arc tangent function satisfies the basic identity
\begin{equation}\label{eq:fractiontoarctan}
\log\left(\frac{1+\mathrm{i} x}{1-\mathrm{i} x}\right)  = 2\mathrm{i} \arctan x,~~~\text{and thus},~~~\frac{1+\mathrm{i} x}{1-\mathrm{i} x}
= \mathrm{e}^{2\mathrm{i} \arctan x},~~~\text{for $x\in\mathbb{R}$},
\end{equation}
which establishes that the Cayley transform ${(1+\mathrm{i} x)/(1-\mathrm{i} x)}$ has the phase function $g(x)={2 \arctan(x)}$.

Let $r\in\mathcal{U}_n$ with minimal degree $m\leq n$ and poles $s_1,\ldots,s_m$ given by $s_j=\xi_j + \mathrm{i} \mu_j$. In particular, $r$ can be understood as an irreducible rational function $r\in\mathcal{U}_m$ with~$\xi_j\neq 0$ for $j=1,\ldots,m$, due to Proposition~\ref{prop:fulldegree}.
For $z=\mathrm{i} x$ the representation~\eqref{eq:defrbyp} yields
\begin{equation}\label{eq:defr2}
r(\mathrm{i} x)
= \mathrm{e}^{\mathrm{i} \theta}  \prod_{j=1}^m \frac{1 + \mathrm{i} (x-\mu_j)/\xi_j}{1 - \mathrm{i} (x-\mu_j)/\xi_j},~~~x\in\mathbb{R}.
\end{equation}
Applying~\eqref{eq:fractiontoarctan} term-wise to~\eqref{eq:defr2}, we arrive at $r(\mathrm{i} x) = \mathrm{e}^{\mathrm{i} g(x)}$ with
\begin{equation}\label{eq:defg}
g(x) := \theta+2 \sum_{j=1}^m\arctan\left( \frac{x-\mu_j}{\xi_j}\right),~~~x\in\mathbb{R}.
\end{equation}
For the special case that~$r$ has a minimal degree~$m=0$, this representation remains valid with~$g=\theta$. However, for a given~$r\in\mathcal{U}_n$ the phase~$\theta$ in~\eqref{eq:defr2}, and respectively in~$g$~\eqref{eq:defg}, is not uniquely defined without further considerations, since~$\mathrm{e}^{\mathrm{i} \theta} = \mathrm{e}^{\mathrm{i} (\theta+2k\pi)}$ for~$k\in\mathbb{N}$.

\begin{proposition}\label{prop:approxertophaseerrinequ}
The following results hold true. 
\begin{enumerate}[label=(\roman*)]
\item\label{item:approxandphaseerrlesstwo} The approximation error of a unitary function $r\in\mathcal{U}_n$ is not maximal, i.e., $\|r - \exp(\omega\cdot)\|<2$, if and only if there exists $g$ of the form~\eqref{eq:defg} with $r(\mathrm{i} x) = \mathrm{e}^{\mathrm{i} g(x)}$ and $\max_{x\in[-1,1]}  | g(x) - \omega x | < \pi$.
\item\label{item:gunique} The phase function $g$ in~\ref{item:approxandphaseerrlesstwo} is unique.
\item\label{item:riqtogiq} Let $r_1\in\mathcal{U}_n$ with $\|r_1 - \exp(\omega\cdot)\|<2$ and $r_1(\mathrm{i} x)=\mathrm{e}^{\mathrm{i} g_1(x)}$. Then $r_2\in\mathcal{U}_n$ with $r_2(\mathrm{i} x)=\mathrm{e}^{\mathrm{i} g_2(x)}$ has a smaller approximation error than $r_1$,
\begin{subequations}\label{eq:r12inequtog12inequ}
\begin{equation}\label{eq:r12inequtog12inequ1}
\| r_2 - \exp(\omega \cdot)\| < \| r_1 - \exp(\omega \cdot)\|,
\end{equation}
if and only if it has a smaller phase error than $r_1$,
\begin{equation}\label{eq:r12inequtog12inequ2}
\max_{x\in[-1,1]} |g_2(x)-\omega x| < \max_{x\in[-1,1]} |g_1(x)-\omega x|.
\end{equation}
\end{subequations}
\item\label{item:xmaxerr} Let $r\in\mathcal{U}_n$ with $\| r - \exp(\omega \cdot)\| <2$ and $r(\mathrm{i} x)=\mathrm{e}^{\mathrm{i} g(x)}$, then a point $y\in[-1,1]$ satisfies $|r(\mathrm{i} y) - \mathrm{e}^{\mathrm{i} \omega y}| = \| r - \exp(\omega \cdot)\|$ if and only if $|g(y)-\omega y| = \max_{x\in[-1,1]} |g(x)-\omega x|$.
\end{enumerate}
\end{proposition}
\begin{proof}
Let $r\in\mathcal{U}_n$ be given and let $g$ correspond to~\eqref{eq:defg} s.t.\ $r(\mathrm{i} x) = \mathrm{e}^{\mathrm{i} g(x)}$. 
The point-wise deviation of $\mathrm{e}^{\mathrm{i} g(x)}\approx \mathrm{e}^{\mathrm{i} \omega x}$ satisfies the identity
\begin{equation}\label{eq:errortosinphasepointwise}
|\mathrm{e}^{\mathrm{i} g(x)} - \mathrm{e}^{\mathrm{i} \omega x}|  =  |\mathrm{e}^{\mathrm{i} (g(x) - \omega x)/2} - \mathrm{e}^{-\mathrm{i} (g(x) - \omega x)/2} |
= 2  | \sin((g(x) - \omega x)/2 ) |.
\end{equation}
Considering the notation~\eqref{eq:errnormnotation} and $r(\mathrm{i} x) = \mathrm{e}^{\mathrm{i} g(x)}$ we observe
\begin{equation}\label{eq:errortosinphase}
\|r - \exp(\omega\cdot)\|
= 2 \max_{x\in[-1,1]} | \sin((g(x) - \omega x)/2 ) |.
\end{equation}
Thus, the approximation error is not maximal,
\begin{subequations}\label{eq:approxerrorlesstwoequiv}
\begin{equation}\label{eq:approxerrorlesstwoinproof}
\|r - \exp(\omega\cdot)\|<2,
\end{equation}
if and only if
\begin{equation}\label{eq:non-maximal-sin-error}
\max_{x\in[-1,1]} |\sin((g(x) - \omega x)/2 )| < 1.
\end{equation}
Furthermore, since the phase error is continuous,~\eqref{eq:non-maximal-sin-error} holds true if and only if
\begin{equation}\label{eq:phaseerrorinterval}
g(x) - \omega x  \in (-\pi + 2\kappa\pi,\pi+2\kappa\pi),~~~\text{for $x\in[-1,1]$ and some $\kappa\in\mathbb{N}$}.
\end{equation}
\end{subequations}
The complex phase of $r$, i.e., $\mathrm{e}^{\mathrm{i} \theta}$ in~\eqref{eq:defr2}, is uniquely defined. However, the phase $\theta$, which also corresponds to a shift in $g$ as in~\eqref{eq:defg}, can be replaced by $\theta+2k\pi$ for $k\in\mathbb{N}$ without changing $\mathrm{e}^{\mathrm{i}  \theta}$.
Thus, if~\eqref{eq:approxerrorlesstwoinproof} holds true for a $\theta \in \mathbb{R}$, it remains true when shifting $\theta$ by multiples of $2\pi$ -- in particular, this allows us to shift $g$ by multiples of $2\pi$ to attain $\kappa=0$ in~\eqref{eq:phaseerrorinterval}, leading to $\max_{x\in[-1,1]}  | g(x) - \omega x | < \pi$ under the assumption that~\eqref{eq:approxerrorlesstwoinproof} holds true.
Conversely, $\max_{x\in[-1,1]}  | g(x) - \omega x | < \pi$ leads to~\eqref{eq:phaseerrorinterval} holding for $\kappa=0$, and consequently to the non-maximality of the approximation error~\eqref{eq:approxerrorlesstwoequiv}, completing the assertion~\ref{item:approxandphaseerrlesstwo}. Since this corresponds to a unique choice of $\theta$, the statement of~\ref{item:gunique} follows.

We proceed to show~\ref{item:riqtogiq}. If~\eqref{eq:r12inequtog12inequ1} holds, then the non-maximality of the approximation error for $r_1$ implies the same for $r_2$, i.e., $\|r_2-\exp(\omega \cdot)\| \leq \|r_1-\exp(\omega \cdot)\| < 2$. It follows from~\ref{item:approxandphaseerrlesstwo} and~\ref{item:gunique} that the phase functions $g_1$ and $g_2$ with $r_j(\mathrm{i} x)=\mathrm{e}^{\mathrm{i} g_j(x)}$ and $\max_{x\in[-1,1]}|g_j(x) -\omega x| < \pi$, for $j=1,2$, are uniquely defined.
Since the function $\sin(y)$ is strictly monotonic for $y\in(-\pi/2,\pi/2)$, the phase errors of $g_1$ and $g_2$ satisfy
\begin{equation}\label{eq:errortosinphase2}
\max_{x\in[-1,1]} | \sin((g_j(x) - \omega x)/2 ) | = \sin \Big( \max_{x\in[-1,1]} | g_j(x) - \omega x |/2 \Big),~~~j=1,2.
\end{equation}
Combining~\eqref{eq:errortosinphase} and~\eqref{eq:errortosinphase2}, the inequality~\eqref{eq:r12inequtog12inequ1} corresponds to 
\begin{equation}\label{eq:errortosinphase3}
\sin \Big( \max_{x\in[-1,1]} | g_2(x) - \omega x |/2 \Big) 
< \sin \Big( \max_{x\in[-1,1]} | g_1(x) - \omega x |/2 \Big).
\end{equation}
Again, since the arguments of the sinus functions on both sides of this inequality are in the interval $(-\pi/2,\pi/2)$, the sinus function is strictly monotonic in~\eqref{eq:errortosinphase3}. Thus, the inequality~\eqref{eq:r12inequtog12inequ1} implies~\eqref{eq:r12inequtog12inequ2}.

We proceed to show that~\eqref{eq:r12inequtog12inequ2} also implies~\eqref{eq:r12inequtog12inequ1}.
Assume~\eqref{eq:r12inequtog12inequ2} holds true, then the phase error for $r_2$ satisfies $\max_{x\in[-1,1]} |g_2(x)-\omega x|< \pi$.
Similar to above, since the sinus is strictly monotonic for arguments in $(-\pi/2,\pi/2)$, the inequality~\eqref{eq:r12inequtog12inequ2} implies~\eqref{eq:errortosinphase3}.
Furthermore,~\eqref{eq:errortosinphase2} holds true and together with~\eqref{eq:errortosinphase} this completes the proof of~\ref{item:riqtogiq}. 

We proceed to show~\ref{item:xmaxerr}. The identity~\eqref{eq:errortosinphasepointwise} implies that a point $y\in[-1,1]$ satisfies
\begin{equation*}
|r(\mathrm{i} y) - \mathrm{e}^{\mathrm{i} \omega y}| = \| r - \exp(\omega \cdot)\|
\end{equation*}
if and only if
\begin{equation}\label{eq:sinphaseerroyismax}
| \sin((g(y) - \omega y)/2 ) | = \max_{x\in[-1,1]} | \sin((g(x) - \omega x)/2 ) |.
\end{equation}
Since we assume $\| r - \exp(\omega \cdot)\| <2$, the phase error $g(x)-\omega x$ is in $[-\pi,\pi]$ for $x\in[-1,1]$ and the sinus functions in~\eqref{eq:sinphaseerroyismax} are strictly monotonic. Thus, a point $y$ attains the maximum in~\eqref{eq:sinphaseerroyismax} if and only if it attains the maximum of $|g(x)-\omega x|$ over $[-1,1]$, which completes our proof.
\end{proof}

In the sequel, for $r\in\mathcal{U}_n$ with non-maximal approximation error, i.e.,\ $\|r - \exp(\omega\cdot)\|<2$, the representation $r(\mathrm{i} x)=\mathrm{e}^{\mathrm{i} g(x)}$ refers to $g$ uniquely defined by Proposition~\ref{prop:approxertophaseerrinequ}, assertions~\ref{item:approxandphaseerrlesstwo} and~\ref{item:gunique}.
While this condition might seem restrictive for the definition of the phase function, in Proposition~\ref{prop:omega0} we show that it is satisfied by the unitary best  approximation for a large range of frequencies. Moreover, when it is not satisfied, every unitary function is a best approximation, i.e., we have a trivial case which is not of practical interest. We proceed by stating some auxiliary results first.
\begin{proposition}\label{prop:wtomincontinuous}
For a given degree $n$, the following regularity results hold true.
\begin{enumerate}[label=(\roman*)]
\item\label{item:erroratxiscont} The point-wise deviation $|r(x)-\mathrm{e}^{\mathrm{i} \omega x}|$ is Lipschitz continuous in $r$ and $\omega$ with Lipschitz constant $1$ for $x\in[-1,1]$. In particular, for~$r_1,r_2\in\mathcal{U}_n$ and~$\omega_1,\omega_2>0$ we have
\begin{equation}\label{eq:erratxcontinrw}
\big||r_2(x)-\mathrm{e}^{\mathrm{i} \omega_2 x}|-|r_1(x)-\mathrm{e}^{\mathrm{i} \omega_1 x}|\big|
\leq |\omega_2-\omega_1| + |r_{2}(x)-r_1(x)|,~~~x\in[-1,1].
\end{equation}
\item\label{item:minecontinw} The minimal error $\min_{r\in \mathcal{U}_n} \| r - \exp(\omega\cdot)\|$ is Lipschitz continuous in $\omega>0$ with Lipschitz constant $1$.
\end{enumerate}
\end{proposition}
\begin{proof}
We prove~\ref{item:erroratxiscont} and~\ref{item:minecontinw} sequentially.
\begin{enumerate}[label=(\roman*)]
\item For a fixed $r\in\mathcal{U}_n$ and $x\in[-1,1]$ the reverse triangular inequality yields
\begin{equation}\label{eq:krcontfirststep}
\begin{aligned}
\big||r(\mathrm{i} x) - \mathrm{e}^{\mathrm{i} \omega_2 x}|-|r(\mathrm{i} x) - \mathrm{e}^{\mathrm{i} \omega_1 x}|\big|
&\leq |(r(\mathrm{i} x) - \mathrm{e}^{\mathrm{i} \omega_2 x})-(r(\mathrm{i} x) - \mathrm{e}^{\mathrm{i} \omega_1 x})|\\
&\qquad= |\mathrm{e}^{\mathrm{i} \omega_2 x} - \mathrm{e}^{\mathrm{i} \omega_1 x}|  \leq |\omega_2-\omega_1|,
\end{aligned}
\end{equation}
where the last inequality follows from the upper bound~\eqref{eq:boundonexpiw}. Similarly, for a fixed $\omega>0$ we get
\begin{equation}\label{eq:krboundedbyr}
\big||r_2(\mathrm{i} x) - \mathrm{e}^{\mathrm{i} \omega x}|-|r_1(\mathrm{i} x) - \mathrm{e}^{\mathrm{i} \omega x}|\big|
\leq |r_2(\mathrm{i} x)-r_1(\mathrm{i} x)|.
\end{equation}
Furthermore, the left-hand side of~\eqref{eq:erratxcontinrw} expands to
\begin{equation*}
\begin{aligned}
\big||r_2(\mathrm{i} x)-\mathrm{e}^{\mathrm{i} \omega_2 x}|-|r_1(\mathrm{i} x)-\mathrm{e}^{\mathrm{i} \omega_1 x}|\big|
&\leq  \big||r_2(\mathrm{i} x)-\mathrm{e}^{\mathrm{i} \omega_2 x}|-|r_2(\mathrm{i} x)-\mathrm{e}^{\mathrm{i} \omega_1 x}|\big|\\
&~~~~ + \big||r_2(\mathrm{i} x)-\mathrm{e}^{\mathrm{i} \omega_1 x}|-|r_1(\mathrm{i} x)-\mathrm{e}^{\mathrm{i} \omega_1 x}|\big|.
\end{aligned}
\end{equation*}
Substituting~\eqref{eq:krcontfirststep} and~\eqref{eq:krboundedbyr} therein, we conclude~\eqref{eq:erratxcontinrw}.
\item
We recall the notation
\begin{equation*} \tag{\ref{eq:errnormnotation}}
\|r-\exp(\omega\cdot)\| = \max_{x\in[-1,1]} |r(\mathrm{i} x) - \mathrm{e}^{\mathrm{i} \omega x}|.
\end{equation*}
For given $\omega_1,\omega_2>0$ there exist points~$x_1$ and~$x_2\in[-1,1]$ with
\begin{equation}\label{eq:krx1x2max}
\|r-\exp(\omega_1\cdot)\| = |r(\mathrm{i} x_1)-\mathrm{e}^{\mathrm{i} \omega_1 x_1}|,~~~\text{and}~~
\|r-\exp(\omega_2\cdot)\| = |r(\mathrm{i} x_2)-\mathrm{e}^{\mathrm{i} \omega_2 x_2}|.
\end{equation}
In particular, $|r(\mathrm{i} x_1)-\mathrm{e}^{\mathrm{i} \omega_1 x_1}| \geq |r(\mathrm{i} x_2)-\mathrm{e}^{\mathrm{i} \omega_1 x_2}| $, and thus,
\begin{subequations}\label{eq:krboundedstepx}
\begin{equation}
|r(\mathrm{i} x_2)-\mathrm{e}^{\mathrm{i} \omega_2 x_2}|-|r(\mathrm{i} x_1)-\mathrm{e}^{\mathrm{i} \omega_1 x_1}|
\leq |r(\mathrm{i} x_2)-\mathrm{e}^{\mathrm{i} \omega_2 x_2}|-|r(\mathrm{i} x_2)-\mathrm{e}^{\mathrm{i} \omega_1 x_2}|.
\end{equation}
In a similar manner, $|r(\mathrm{i} x_2)-\mathrm{e}^{\mathrm{i} \omega_2 x_2}| \geq |r(\mathrm{i} x_1)-\mathrm{e}^{\mathrm{i} \omega_2 x_1}|$ implies
\begin{equation}
|r(\mathrm{i} x_2)-\mathrm{e}^{\mathrm{i} \omega_2 x_2}|-|r(\mathrm{i} x_1)-\mathrm{e}^{\mathrm{i} \omega_1 x_1}|
\geq |r(\mathrm{i} x_1)-\mathrm{e}^{\mathrm{i} \omega_2 x_1}|-|r(\mathrm{i} x_1)-\mathrm{e}^{\mathrm{i} \omega_1 x_1}|.
\end{equation}
\end{subequations}
Substituting $x_1$ and $x_2$ for $x$ and setting $r=r_1=r_2$ in~\eqref{eq:erratxcontinrw}, we observe that the right-hand sides of~\eqref{eq:krboundedstepx} are bounded in absolute value by $|\omega_2-\omega_1|$. Thus, combining~\eqref{eq:krboundedstepx} and~\eqref{eq:krx1x2max} we arrive at
\begin{equation}\label{eq:krtoeboundedbyomega}
\big| \|r-\exp(\omega_2\cdot)\| - \|r-\exp(\omega_1\cdot)\|\big|
\leq |\omega_2-\omega_1|.
\end{equation}

We proceed to show that the minimum over $r\in\mathcal{U}_n$ satisfies a similar upper bound.
Let $\omega_1$ and $\omega_2>0$ be given frequencies, then, as shown in Proposition~\ref{prop:bestapproxinUnexists}, there exist unitary best approximations $r_1$ and $r_2\in\mathcal{U}_n$ with
\begin{equation}\label{eq:minuerrattainedr1r2}
\begin{aligned}
&\min_{u\in\mathcal{U}_n} \|u-\exp(\omega_1\cdot)\|= \|r_1-\exp(\omega_1\cdot)\|,~~~
 \text{and}\\
&\min_{u\in\mathcal{U}_n} \|u-\exp(\omega_2\cdot)\|= \|r_2-\exp(\omega_2\cdot)\|.
\end{aligned}
\end{equation}
In particular, $\|r_2-\exp(\omega_2\cdot)\| \leq \|r_1-\exp(\omega_2\cdot)\|$, and thus,
\begin{subequations}\label{eq:krboundedstepx2}
\begin{equation}
\|r_2-\exp(\omega_2\cdot)\|-\|r_1-\exp(\omega_1\cdot)\|
\leq \|r_1-\exp(\omega_2\cdot)\|-\|r_1-\exp(\omega_1\cdot)\|.
\end{equation}
In a similar manner, $\|r_1-\exp(\omega_1\cdot)\| \leq \|r_2-\exp(\omega_1\cdot)\|$ implies
\begin{equation}
\|r_2-\exp(\omega_2\cdot)\|-\|r_1-\exp(\omega_1\cdot)\|
\geq \|r_2-\exp(\omega_2\cdot)\|-\|r_2-\exp(\omega_1\cdot)\|.
\end{equation}
\end{subequations}
Substituting $r_1$ and $r_2$ for $r$ in~\eqref{eq:krtoeboundedbyomega}, we observe that the right-hand sides of~\eqref{eq:krboundedstepx2} are bounded by $|\omega_2-\omega_1| $ is absolute value. Thus, combining~\eqref{eq:krtoeboundedbyomega},~\eqref{eq:minuerrattainedr1r2} and~\eqref{eq:krboundedstepx2} we observe
\begin{equation*}
\left| \min_{u\in\mathcal{U}_n} \|u-\exp(\omega_2\cdot)\|-  \min_{u\in\mathcal{U}_n} \|u-\exp(\omega_1\cdot)\|\right|
 \leq  |\omega_2-\omega_1|.
\end{equation*}
This implies that $\min_{u\in\mathcal{U}_n} \|u-\exp(\omega\cdot)\|$ is Lipschitz continuous in $\omega>0$ with Lipschitz constant $1$.
\end{enumerate}
\end{proof}

\begin{proposition}\label{prop:rforwtonp1pi}
For a sufficiently small $\varepsilon>0$ and $\omega = (1-n\varepsilon)(n+1)\pi$, the unitary function $\widetilde{r}\in\mathcal{U}_n$ defined via the poles
\begin{equation}\label{eq:ip.rtildewtonp1pipoles}
\widetilde{s}_j = \frac{\varepsilon}{1-n\varepsilon} + \mathrm{i}\,\frac{1}{1-n\varepsilon}\left(-1 + \frac{2j}{n+1}\right),~~~j=1,\ldots,n,
\end{equation}
and the complex phase $\mathrm{e}^{\mathrm{i} \phi}=1$, satisfies $\| \widetilde{r} - \exp(\omega\cdot)\|<2  $.
\end{proposition}
\begin{proof}
The proof of this proposition is provided in Appendix~\ref{sec:proofswtonp1pi}.
\end{proof}

\begin{proposition}\label{prop:omega0}
For a fixed degree $n$ and $\omega>0$,
\begin{equation}\label{eq:omega0}
\min_{r\in\mathcal{U}_n} \| r - \exp(\omega \cdot)\| < 2,~~~\text{if and only if $\omega\in(0,(n+1)\pi )$.}
\end{equation}
Moreover, for $\omega\geq (n+1)\pi$ all $r\in\mathcal{U}_n$ attain the error $\| r - \exp(\omega \cdot)\| = 2$.
\end{proposition}
\begin{proof}
We first show that there exists a $\widehat{\omega}_n>0$ s.t.\ for $\omega>0$,
\begin{equation}\label{eq:ip.widehatomega}
\min_{r\in\mathcal{U}_n} \| r - \exp(\omega \cdot)\| < 2,~~~\text{if and only if $\omega\in(0, \widehat{\omega}_n )$.}
\end{equation}
To this end, we first show that the minimum in~\eqref{eq:omega0} is monotonically increasing in $\omega$. Let $r_1\in\mathcal{U}_n$ denote the unitary best approximation for a given $\omega_1$, and let $\omega_2 \leq \omega_1$ be fixed. For~$\omega_2$, the unitary function~$r(\mathrm{i} x) = r_1(\mathrm{i} \tfrac{\omega_2}{\omega_1} x)$ attains the approximation error
\begin{equation*}
\max_{x\in[-1,1]} |r(\mathrm{i} x)-\mathrm{e}^{\mathrm{i} \omega_2 x}|
= \max_{x\in[-1,1]} |r_1(\mathrm{i} \tfrac{\omega_2}{\omega_1} x)-\mathrm{e}^{\mathrm{i} \omega_2 x}|
= \max_{y\in\tfrac{\omega_2}{\omega_1}[-1,1]\subset[-1,1]} |r_1(\mathrm{i} y  )-\mathrm{e}^{\mathrm{i} \omega_1 y}|,
\end{equation*}
which implies
\begin{equation*}
\min_{u\in\mathcal{U}_n} \| u - \exp(\omega_2 \cdot)\|
\leq \| r_1 - \exp(\omega_1 \cdot)\|
= \min_{u\in\mathcal{U}_n} \| u - \exp(\omega_1 \cdot)\|,~~~\text{for $\omega_2\leq \omega_1$}.
\end{equation*}

Since $\mathrm{e}^{\mathrm{i} \omega x} \approx 1+\mathrm{i}\omega x$ for $\omega\to 0$, and since the constant function $r\equiv 1$ is included in~$\mathcal{U}_n$, the approximation error~$\min_{r\in\mathcal{U}_n}\|r-\exp(\omega\cdot)\|$ converges to zero for~$\omega\to0$. In particular,~$\min_{r\in\mathcal{U}_n}\|r-\exp(\omega\cdot)\|<2$ for sufficiently small $\omega$. More precise error bounds are provided in Proposition~\ref{prop:errorbound}.
Following Proposition~\ref{prop:wtomincontinuous}~\ref{item:minecontinw}, the mapping from $\omega>0$ to the minimal approximation error, i.e.,  $ \omega \mapsto \min_{r\in\mathcal{U}_n} \| r - \exp(\omega \cdot)\|$ is continuous, and thus, the set of $\omega>0$ with $\min_{r\in\mathcal{U}_n} \| r - \exp(\omega \cdot)\| \in (0,2)$ is an open set. Convergence, monotonicity and continuity imply that this open set is also an interval $(0,\widehat{\omega}_n)$, as required in~\eqref{eq:ip.widehatomega}.

In case of $\omega\geq \widehat{\omega}_n$ this implies $\min_{u\in\mathcal{U}_n}\| u-\exp(\omega \cdot) \| = 2$, and all $r\in\mathcal{U}_n$ attain an error of {two} since unitarity implies $\| r-\exp(\omega \cdot) \|\leq 2$.

\medskip
We proceed to show that $\widehat{\omega}_n$ in~\eqref{eq:ip.widehatomega} is bounded from above by
\begin{equation}\label{eq:ip.widehatomegaupperbound}
\widehat{\omega}_n \leq (n+1)\pi.
\end{equation}
Following Proposition~\ref{prop:approxertophaseerrinequ},
the approximation error attained by $r\in\mathcal{U}_n$ is non-maximal, i.e., $<2$, if and only if there exists a phase function $g$ with $r(\mathrm{i} x) = \mathrm{e}^{\mathrm{i} g(x)}$ and phase error bounded by $\pi$, i.e., $\max_{x\in[-1,1]}| g(x) - \omega x | < \pi$.
The latter implies $ g(-1) < \pi - \omega $ and $ g(1) > \omega - \pi $, so that~{$g(1)-g(-1)>2(\omega-\pi)$}.
Considering the maximum and minimum values of $\arctan$ in $g$,~\eqref{eq:defg}, we also have $\max_{x\in[-1,1]} g(x) - \min_{x\in[-1,1]} g(x) \leq 2n \pi $. For $\omega = (n+1) \pi$, where $2 (\omega -\pi) = 2  n \pi$, these lead to a contradiction, and thus,~$\max_{x\in[-1,1]}| g(x) - \omega x | < \pi$ is not feasible for~$\omega = (n+1)\pi$.
Consequently, $\min_{r\in\mathcal{U}_n}\|r - \exp(\omega\cdot)\| = 2$ for~$\omega = (n+1)\pi$ and we conclude~\eqref{eq:ip.widehatomegaupperbound}.

\medskip
It remains to show that~\eqref{eq:ip.widehatomegaupperbound} is an equality.
Following Proposition~\ref{prop:rforwtonp1pi}, for sufficiently small $\varepsilon>0$ and $\omega = (1-n\varepsilon)(n+1)\pi$ we can construct $r\in\mathcal{U}_n$ with $\|r-\exp(\omega \cdot)\|<2$. This especially implies that the unitary best approximation attains an approximation error $<2$ for $\omega$ arbitrarily close to $(n+1)\pi$, approaching from below. Together with the upper bound $\widehat{\omega}_n \leq (n+1)\pi$~\eqref{eq:ip.widehatomegaupperbound}, this implies that $\widehat{\omega}_n$ in~\eqref{eq:ip.widehatomega} satisfies $\widehat{\omega}_n=(n+1)\pi$, i.e.,~\eqref{eq:omega0} holds true.
\end{proof}

We remark that, while we consider $n$ to be fixed in Proposition~\ref{prop:omega0}, we may also consider a fixed $\omega>0$ instead. Namely, for a sufficiently large degree $n$ the condition $\omega<(n+1)\pi$ holds true and the best approximation $r$ in $\mathcal{U}_n$ attains an approximation error $\|r-\exp(\omega\cdot)\|<2$. 

Note that $\mathrm{e}^{\mathrm{i} \omega x}$ has $\frac{\omega}{\pi}$ full oscillations or {\em wavelengths} in the interval $[-1,1]$. The case $\omega=(n+1)\pi$ corresponds to $n+1$ wavelengths. Thus $\omega \in (0, (n+1) \pi)$, would correspond to a restriction that degree $(n,n)$ unitary rational approximants should not be used for approximating $n+1$ or more wavelengths.

\section{Equioscillation and uniqueness}\label{sec:main}

For the approximation $r(\mathrm{i} x)\approx\mathrm{e}^{\mathrm{i} \omega x}$ for $\omega\in(0,(n+1)\pi)$, $x\in[-1,1]$ and $r\in\mathcal{U}_n$, we have the representation $r(\mathrm{i} x) = \mathrm{e}^{\mathrm{i} g(x)}$, where $g$ is to be understood as the unique representation from Proposition~\ref{prop:approxertophaseerrinequ}.
While the approximation error $r(\mathrm{i} x)-\mathrm{e}^{\mathrm{i} \omega x}$ (cf. Definition~\ref{def:errors}) lives in the complex plane, the phase error $g(x) - \omega x$ is real-valued. Equioscillating properties are typically not viable in the complex plane.
However, following Proposition~\ref{prop:approxertophaseerrinequ} and Proposition~\ref{prop:omega0}, a unitary function $r\in\mathcal{U}_n$ minimizes $\|r-\exp(\omega \cdot)\|$ for the non-trivial case $\omega\in(0,(n+1)\pi)$ if and only if the respective phase function $g$ minimizes the phase error $\max_{x\in[-1,1]} |g(x)-\omega x|$ in the class of functions of the form~\eqref{eq:defg}.
In this case, the best approximant $r$ can be uniquely characterized via an equioscillation property of the real phase error, which is introduced in the following.

We say that the phase error of a unitary approximation $r(\mathrm{i} x) = \mathrm{e}^{\mathrm{i} g(x)} \approx \mathrm{e}^{\mathrm{i} \omega x}$
equioscillates between $m$ points
\begin{equation*}
-1\leq \eta_1 < \ldots < \eta_m \leq 1,
\end{equation*}
if
\begin{equation}\label{eq:defeo}
g(\eta_j) - \omega \eta_j = (-1)^{j+\iota} \max_{x\in[-1,1]}|g(x) - \omega x |,~~~j=1,\ldots,m,~~~\iota\in\{0,1\}.
\end{equation}

\begin{theorem}[Equioscillation characterization of best approximants]
\label{thm:bestapprox}
Provided $n$ is a fixed degree and $\omega\in(0,(n+1)\pi)$, then
$r\in\mathcal{U}_n$ is a unitary rational best approximation to $\mathrm{e}^{\mathrm{i} \omega x}$, i.e.,
\begin{equation*}
\| r - \exp(\omega \cdot)\|
= \min_{u \in\mathcal{U}_n} \| u - \exp(\omega \cdot)\| <2,
\end{equation*}
if and only if the phase error
of $r(\mathrm{i} x)\approx \mathrm{e}^{\mathrm{i} \omega x}$ equioscillates between $2n+2$ points in $[-1,1]$. Moreover, then
\begin{enumerate}[label=(\roman*)]
\item\label{item:rwisunique} the unitary best approximation $r$ is unique,
\item\label{item:rwmindegreen} $r$ has minimal degree $n$ and distinct poles, 
\item\label{item:rwequi} there are exactly $2n+2$ equioscillation points, which include the points $-1$ and $1$, and
\item\label{item:rwphaseerrmonotonic} the $2n$ equioscillation points in $(-1,1)$ are exactly the zeros of the derivative of the phase error and these zeros are simple.
\end{enumerate}
\end{theorem}
\begin{proof}
The structure of our proof is motivated by the proofs of Theorem~10.1 and Theorem~24.1 in~\cite{Tre13}.
The proof of Theorem~\ref{thm:bestapprox} consists of
\begin{enumerate}[leftmargin=3cm]
\item[Proposition~\ref{prop:p11equitoopti}] (equioscillation implies optimality),
\item[Proposition~\ref{prop:p2optitoequi}] (optimality implies equioscillation, and~\ref{item:rwmindegreen}--\ref{item:rwphaseerrmonotonic}), and
\item[Proposition~\ref{prop:p3uniqueness}] (uniqueness~\ref{item:rwisunique}).
\end{enumerate}
\end{proof}

Throughout the present work we refer to the equioscillation points in Theorem~\ref{thm:bestapprox} as equioscillation points of the phase error and equioscillation points of the unitary best approximation in an equivalent manner.

\begin{figure}
\centering
\includegraphics{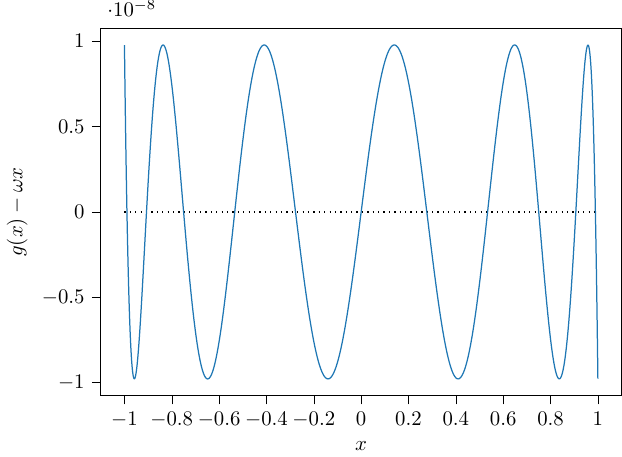}
\caption{
The phase error $ g(x)-\omega x$ of the unitary rational best approximation $r(\mathrm{i} x) = \mathrm{e}^{\mathrm{i} g(x)} \approx \mathrm{e}^{\mathrm{i} \omega x}$ for $\omega=2.85$ and degree $n=5$.}
\label{fig:phaseerror}
\end{figure}

\begin{corollary}[to Theorem~\ref{thm:bestapprox}]\label{cor:errattainsmax}
For a fixed degree $n$ and $\omega\in(0,(n+1)\pi)$,
let $r\in \mathcal{U}_n$ be the unitary best approximant $ r(\mathrm{i} x)=\mathrm{e}^{\mathrm{i} g(x)} \approx \mathrm{e}^{\mathrm{i} \omega x} $ as in Theorem~\ref{thm:bestapprox}.
Let $-1=\eta_1<\ldots<\eta_{2n+2}=1$ denote the equioscillation points of the phase error of $r$.
Then, following Proposition~\ref{prop:approxertophaseerrinequ}.\ref{item:xmaxerr} the point-wise approximation error satisfies
\begin{equation*}
| r(\mathrm{i} \eta_j) - \mathrm{e}^{\mathrm{i} \omega \eta_j}| = \| r - \exp(\omega \cdot)\| ,~~~j = 1,\ldots, 2n+2.
\end{equation*}
Due to Proposition~\ref{prop:approxertophaseerrinequ}.\ref{item:xmaxerr} and Theorem~\ref{thm:bestapprox}.\ref{item:rwphaseerrmonotonic}, the equioscillation points are exactly the points where the approximation and phase error attain their maxima in absolute value on $[-1,1]$, and the phase error is strictly monotonic between the equioscillation points.
Moreover, the phase error has exactly one simple zero between each pair of neighboring equioscillation points, i.e., we have nodes $x_1,\ldots,x_{2n+1}\in\mathbb{R}$ with~$x_j\in(\eta_j,\eta_{j+1})$, and
\begin{equation*}
g(x_j)=\omega x_j,~~~j = 1,\ldots, 2n+1,
\end{equation*}
and these nodes also provide zeros of the approximation error. Namely, the unitary best approximation interpolates $\mathrm{e}^{\omega z}$ with interpolation nodes $\mathrm{i} x_1,\ldots,\mathrm{i} x_{2n+1}\in\mathrm{i}\mathbb{R}$ for $x_j\in(\eta_j,\eta_{j+1})$ as above, i.e.,
\begin{equation*}
r(\mathrm{i} x_j) = \mathrm{e}^{\mathrm{i} \omega x_j},~~~ j = 1,\ldots, 2n+1.
\end{equation*}
\end{corollary}

The phase error of the unitary best approximation is illustrated for a numerical example in Fig.~\ref{fig:phaseerror}.
Results concerning the approximation error, cf.\ Corollary~\ref{cor:errattainsmax}, are illustrated in figures~\ref{fig:bestapproxinterpolate} and~\ref{fig:errorpath}. The equioscillation property leads to the path of the approximation error in the complex plane being {\em floral} in nature. The error attains extrema at the ends of the path, which correspond to the equioscillation points $\eta_1=-1$ and $\eta_{2n+2}=1$, respectively. The remaining extrema, which correspond to the $2n$ equioscillation points in $(-1,1)$, appear at the end of $2n$ {\em petals}. Unlike rose curves~\cite{La95}, the petals are neither uniform in their spacing nor width, and can be nested inside others.

In Fig.~\ref{fig:phaseerror} we also observe that the phase error attains its maximum at the first equioscillation point, i.e., $\eta_1=-1$. As shown in Proposition~\ref{prop:eo1max} further below, this holds true in general for the phase error of the unitary best approximation in $\mathcal{U}_n$ and $\omega\in(0,(n+1)\pi)$. Thus, the equioscillation points $\eta_1,\ldots,\eta_{2n+2}$ satisfy~\eqref{eq:defeo} with $\iota=1$.

 \begin{figure}
\centering
\includegraphics{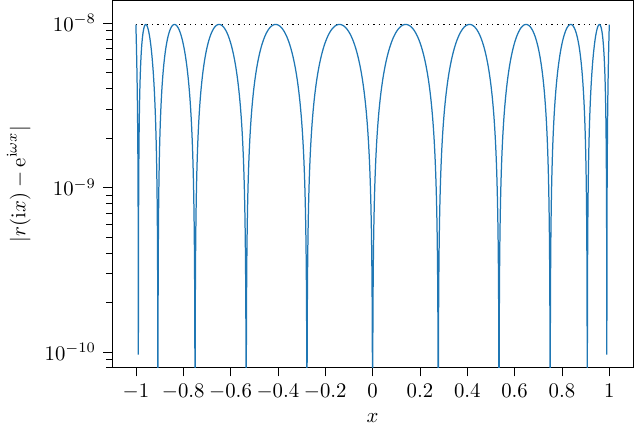}
\caption{The absolute value of the approximation error $|r(\mathrm{i} x) - \mathrm{e}^{\mathrm{i} \omega x}|$ of the unitary best approximation $r(\mathrm{i} x) \approx \mathrm{e}^{\mathrm{i} \omega x}$ for $\omega=2.85$ and degree $n=5$.}
\label{fig:bestapproxinterpolate}
\end{figure}

\begin{figure}
\centering
\includegraphics{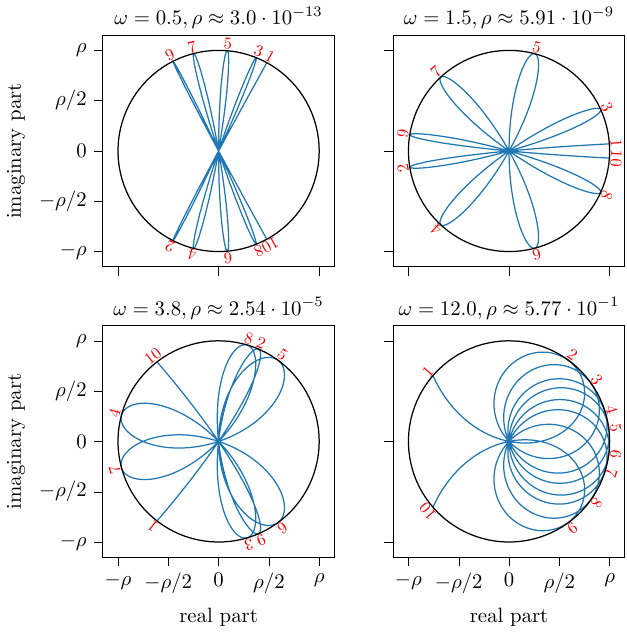}
\caption{Path of the approximation error $x\mapsto r(\mathrm{i} x)-\mathrm{e}^{\mathrm{i} \omega x}$ for $x\in[-1,1]$ in the complex plane, where $r\in\mathcal{U}_n$ refers to the best approximation to $\mathrm{e}^{\mathrm{i} \omega x}$ with $n=4$, and $\omega = 0.5, 1.5, 3.8, 12.0$. The different sub-plots show results for the different values of $\omega$ as noted in the title of each sub-plot together with the radius $\rho$ of the circle which corresponds to the respective maximal magnitude of the approximation error, namely, $\rho\approx 3\cdot 10^{-13}, 5.91\cdot10^{-9}, 2.54\cdot10^{-5},5.77\cdot10^{-1}$. Note that there are $2n+2=10$ points where the approximation error attains its maximum in magnitude, and these points correspond to the equioscillation points of the phase error. Two of these points appear at the 
ends of the error path ($x=-1$ and $x=1$) and $2n$ appear at tips of {\em petals}. Moreover, these points are numbered in ascending order, i.e., `$1$' refers the first equioscillation point at $x=-1$ and `$10$' refers to the last equioscillation point at $x=1$.}
\label{fig:errorpath}
\end{figure}

\subsection{Proof of the equioscillation theorem}\label{subsec:mainproof}

Before stating propositions~\ref{prop:p11equitoopti},~\ref{prop:p2optitoequi} and~\ref{prop:p3uniqueness}, which together comprise the proof of the equioscillation theorem, Theorem~\ref{thm:bestapprox}, we state an auxiliary result in Proposition~\ref{prop:gidentitytor}.

\begin{proposition}\label{prop:gidentitytor}
Let $r_1,r_2\in\mathcal{U}_n$ with $r_1(\mathrm{i} x) = \mathrm{e}^{\mathrm{i} g_1(x)}$ and $r_2(\mathrm{i} x) = \mathrm{e}^{\mathrm{i} g_2(x)}$ where $g_1$ and $g_2$ correspond to functions of the form~\eqref{eq:defg}.
Let $n_1$ and $n_2\leq n$ denote the minimal degrees of $r_1$ and $r_2$, respectively.
If $g_1-g_2$ has at least $n_1+n_2+1$ zeros counting multiplicity, then~$r_1=r_2$.
\end{proposition}
\begin{proof}
We first remark that $g_j$ in $r_j(\mathrm{i} x) = \mathrm{e}^{\mathrm{i} g_j(x)}$ for $j=1,2$ does not necessarily need to be unique for the present proposition, especially, we do not require $r_j$ to satisfy $\|r_j-\exp(\omega\cdot)\|<2$ for some $\omega>0$.
A zero of $g_2-g_1$ of multiplicity $k$ refers to a point $x$ where
\begin{equation*}
\frac{\mathrm{d}^\ell}{\mathrm{d} x^\ell} (g_2(x)-g_1(x))=0~~~\text{for}~~\ell = 0,\ldots, k-1.
\end{equation*}
This implies
\begin{equation*}
\frac{\mathrm{d}^\ell}{\mathrm{d} x^\ell} r_1(\mathrm{i} x) = \frac{\mathrm{d}^\ell}{\mathrm{d} x^\ell} r_2(\mathrm{i} x)~~~\text{for}~~\ell = 0,\ldots,k-1,
\end{equation*}
since~$\tfrac{\mathrm{d}^\ell}{\mathrm{d} x^\ell} r_j(\mathrm{i} x)$ is a function of~$g_j, \tfrac{\mathrm{d} g_j}{\mathrm{d} x} ,\ldots, \tfrac{\mathrm{d}^\ell g_j}{\mathrm{d}^\ell x}$ for $j=1,2$ and $\ell=0,\ldots,k-1$.
Thus, the zeros of~$g_2(x)-g_1(x)$ are also zeros $r_2(\mathrm{i} x)- r_1(\mathrm{i} x)$ with the respective multiplicities.

For $ r_1 = p_1^\dag/p_1$ and $ r_2=p_2^\dag/p_2$
the difference $r_2-r_1$ corresponds to a rational function,
\begin{equation*}
r_2-r_1 = \frac{p_2^\dag p_1 - p_2 p_1^\dag}{p_2 p_1}
=: \frac{\widehat{p}}{\widehat{q}},
\end{equation*}
where $\widehat{p}$ and $\widehat{q}$ are polynomials of degree $\leq n_1+n_2$. Since $r_1\in\mathcal{U}_{n_1}$ and $r_2\in\mathcal{U}_{n_2}$ may be assumed to be given in an irreducible form, the polynomials $p_1$ and $p_2$ have no zeros on the imaginary axis and this carries over to $\widehat{q}$. Thus, any zero of $r_2-r_1$ on the imaginary axis is also a zero of $\widehat{p}$ with the respective multiplicity.
Consequently, the $n_1+n_2+1$ zeros of $r_2-r_1$, counting multiplicity, are also zeros of $\widehat{p}$ which implies $\widehat{p}\equiv 0$ and $r_1=r_2$.
\end{proof}

\begin{proposition}[Equioscillation implies optimality]\label{prop:p11equitoopti}
Let $n$ denote a fixed degree, $r\in \mathcal{U}_n$, and $\omega\in(0,(n+1)\pi)$.
If the phase error of~$r(\mathrm{i} x)\approx \mathrm{e}^{\mathrm{i} \omega x}$ equioscillates between~$2n+2$ points,
then $r$ minimizes the approximation error
\begin{equation}\label{eq:eqiotooptimalminaerr}
\| r-\exp(\omega \cdot) \|
= \min_{u\in\mathcal{U}_n} \| u-\exp(\omega \cdot) \| <2.
\end{equation}
\end{proposition}
\begin{proof}
Let $\widetilde{r}\in\mathcal{U}_n$ denote the minimizer of the approximation error, which exists due to Proposition~\ref{prop:bestapproxinUnexists}. Due to $\omega\in(0,(n+1)\pi)$ and Proposition~\ref{prop:omega0}, the minimizer satisfies
\begin{equation}
\| \widetilde{r}-\exp(\omega \cdot) \|
= \min_{u\in\mathcal{U}_n} \| u-\exp(\omega \cdot) \| <2.
\end{equation}
Moreover, following Proposition~\ref{prop:approxertophaseerrinequ} there exists a unique representation $\widetilde{r}(\mathrm{i} x)=\mathrm{e}^{\mathrm{i} \widetilde{g}(x)}$ with
\begin{equation*}
\max_{x\in[-1,1]}  |  \widetilde{g}(x) - \omega x | < \pi.
\end{equation*}
We proceed to show that some $r\in\mathcal{U}_n$ with $r(\mathrm{i} x) = \mathrm{e}^{\mathrm{i} g(x)}$ for $g$ of the form~\eqref{eq:defg}, and for which the phase error equioscillates $2n+2$ times, also minimizes the approximation error, i.e.,~\eqref{eq:eqiotooptimalminaerr}. Assume the opposite, i.e.,
\begin{equation}\label{eq:eqiotooptimal0a}
\| r-\exp(\omega \cdot) \|>\| \widetilde{r}-\exp(\omega \cdot) \|.
\end{equation}
Following Proposition~\ref{prop:approxertophaseerrinequ}, this implies
\begin{equation}\label{eq:eqiotooptimal1a}
\max_{x\in[-1,1]}  |  g(x) - \omega x |
> \max_{x\in[-1,1]}  |  \widetilde{g}(x) - \omega x |.
\end{equation}
The phase error of $r(\mathrm{i} x)\approx \mathrm{e}^{\mathrm{i} \omega x}$ equioscillates $2n+2$ times, i.e.,\ there exist equi\-oscillation points $\eta_1<\ldots<\eta_{2n+2}$ with
\begin{equation*}
g(\eta_j) - \omega \eta_j = (-1)^{j+\iota} \max_{x\in[-1,1]}|g(x) - \omega x |,~~~j=1,\ldots,m,~~~\iota\in\{0,1\}.
\end{equation*}
The inequality~\eqref{eq:eqiotooptimal1a} implies that $g$ and $\widetilde{g}$ intersect at least once between each pair of neighboring equioscillation points. More precisely, there exist points $x_j\in(\eta_j,\eta_{j+1})$, $j=1,\ldots,2n+1$, with
\begin{equation*}
g(x_j) = \widetilde{g}(x_j),~~~ j = 1,\ldots 2n+1.
\end{equation*}
Following Proposition~\ref{prop:gidentitytor}, this implies $r=\widetilde{r}$ (independently of the minimal degrees of $r$ and $\widetilde{r}$). This identity is contradictory to~\eqref{eq:eqiotooptimal0a}, and thus,
\begin{equation*}
\| r-\exp(\omega \cdot) \|
\leq \| \widetilde{r}-\exp(\omega \cdot) \| < 2.
\end{equation*}
Since $\widetilde{r}$ minimizes the approximation error, the assertion~\eqref{eq:eqiotooptimalminaerr} holds true.
\end{proof}

\begin{proposition}[Optimality implies equioscillation]
\label{prop:p2optitoequi}
For a fixed degree $n$ and $\omega\in(0,(n+1)\pi)$,
let $r\in\mathcal{U}_n$ be a best approximation to $\mathrm{e}^{\mathrm{i} \omega x}$, i.e.,
\begin{equation*}
\|r-\exp(\omega \cdot)\| = \min_{u \in \mathcal{U}_n}\|u-\exp(\omega \cdot)\|<2.
\end{equation*}
Then the phase error of $r(\mathrm{i} x) = \mathrm{e}^{\mathrm{i} g(x)} \approx \mathrm{e}^{\mathrm{i} \omega x}$
equioscillates between $2n+2$ points. Moreover, the following statements of Theorem~\ref{thm:bestapprox} hold true,
\begin{enumerate}
\item[\ref*{item:rwmindegreen}] $r$ has minimal degree $n$ and distinct poles, 
\item[\ref*{item:rwequi}] there are exactly $2n+2$ equioscillation points, which include the points $-1$ and $1$, and
\item[\ref*{item:rwphaseerrmonotonic}] the $2n$ equioscillation points in $(-1,1)$ are exactly the zeros of the derivative of the phase error and these zeros are simple.
\end{enumerate}
\end{proposition}
\begin{proof}
Assume $r_1\in\mathcal{U}_n$ is a best approximant with 
\begin{equation}\label{eq:opttoequir1conds2}
\|r_1-\exp(\omega \cdot)\|<2.
\end{equation}
Let $m\leq n$ denote the minimal degree of $r_1$, thus, $r_1\in\mathcal{U}_m$.
Let $g_1$ be uniquely defined by $ r_1(\mathrm{i} x) = \mathrm{e}^{\mathrm{i} g_1(x)}$ as in Proposition~\ref{prop:approxertophaseerrinequ}. We define
\begin{equation}\label{eq:opttoequidefalph}
\alpha := \max_{x\in[-1,1]} |g_1(x)-\omega x|.
\end{equation}
Our proof consists of the following steps which we elaborate in detail further below.
\begin{enumerate}[label=(\alph*)]
\item\label{item:equiproof1} We first assume the phase error $g_1(x)-\omega x$ has $\ell < m+n+2 $ equioscillating points $\eta_1<\ldots<\eta_\ell$, and to simplify our notation, we assume $g_1(\eta_1)-\omega \eta_1=-\alpha$, thus,
\begin{equation}\label{eq:phaseerrequiinproof}
\eta_1<\ldots<\eta_\ell,~~~\text{with}~~g_1(\eta_j)-\omega \eta_j = (-1)^j \alpha,~~~j=1,\ldots,\ell.
\end{equation}
\item\label{item:equiproof2} This implies that for a sufficiently small $\varepsilon>0$ there exist points $y_1<\ldots<y_{\ell-1}$ with $y_j \in (\eta_j,\eta_{j+1})$, $j=1,\ldots,\ell-1$, s.t.,
\begin{equation}\label{eq:proofoedefyjs}
\begin{aligned}
& g_1(x) - \omega x < \alpha - \varepsilon~~~\text{for}~~
x \in [-1, y_1 + \varepsilon] \cup [y_2-\varepsilon, y_3 + \varepsilon]\\
&~~~~~~~~~~~~~~~~~~~~~~~~~~~~~~~~~~~~~~~  \cup [y_4-\varepsilon, y_5 + \varepsilon] \cup \ldots,~~~\text{and} \\
& g_1(x) - \omega x > -\alpha + \varepsilon~~~\text{for}~~
x \in [y_1-\varepsilon, y_2 + \varepsilon] \cup [y_3-\varepsilon, y_4+ \varepsilon]\\
&~~~~~~~~~~~~~~~~~~~~~~~~~~~~~~~~~~~~~~~~~~   \cup [y_5-\varepsilon, y_6 + \varepsilon] \cup \ldots.
\end{aligned}
\end{equation}
\item\label{item:equiproof3} We proceed to construct $r_2\in\mathcal{U}_n$ with $\|r_2-\exp(\omega\cdot)\|<2$ and $r_2(\mathrm{i} x) = \mathrm{e}^{\mathrm{i} g_2(x)}$, s.t., for a sufficiently small~$\varepsilon>0$,
\begin{subequations}\label{eq:dummygamma2inequalities}
\begin{equation}\label{eq:dummygamma2epsdist}
\max_{x\in[-1,1]} | g_2(x) - g_1(x)| < \varepsilon,
\end{equation}
and
\begin{equation}\label{eq:dummygamma2smallerlarger}
g_2(x) - g_1(x) \left\{\begin{array}{ll}
> 0 & x\in [-1,y_1-\varepsilon]\cup[y_2+\varepsilon,y_3-\varepsilon]\cup\cdots,~~~\text{and}\\
< 0 & x\in [y_1+\varepsilon,y_2-\varepsilon]\cup[y_3+\varepsilon,y_4-\varepsilon]\cup\cdots,
\end{array}\right.
\end{equation}
where the points $y_1,\ldots,y_{\ell-1}$ in~\eqref{eq:dummygamma2smallerlarger} refer to the points in~\eqref{eq:proofoedefyjs}.
\end{subequations}
\item\label{item:equiproof4} Based on~\eqref{eq:proofoedefyjs} and~\eqref{eq:dummygamma2inequalities}, we show that~$r_2$ has an approximation error strictly smaller than~$r_1$.
This  yields a contradiction since~$r_1$ is a best approximation. As a consequence, we conclude that~$r_1$ has at least~$m+n+2$ equioscillating points.
\item\label{item:equiproof5} The derivative of the phase error reveals that the number of equioscillating points is at most $2k+2$ where $k\leq m$ denotes the number of distinct poles of $r_1$. As a consequence we conclude $k=n$, i.e., $m=n$ and all poles of $r_1$ are distinct, which completes our proof.
\end{enumerate}
We proceed to discuss~\ref{item:equiproof1}-\ref{item:equiproof5} in detail. Some of the arguments in~\ref{item:equiproof1}-\ref{item:equiproof4} also appear in the proof~\cite[Theorem~24.1]{Tre13} in a slightly different setting.

Considering our assumption to have $\ell < m+n+2 $ equioscillating points~\eqref{eq:phaseerrequiinproof}, we remark that~$\ell\geq 1$ since the extreme value~\eqref{eq:opttoequidefalph} is attained at least once for $x\in[-1,1]$. Moreover, the assumption that $g_1(\eta_1)-\omega \eta_1=-\alpha$ is not critical since the present proof easily adapts to the case $g_1(\eta_1)-\omega \eta_1=\alpha$.

Since $g_1(x)-\omega x$ is continuous and has the equioscillating points $x_1,\ldots,x_\ell$, we find a sufficiently small $\varepsilon>0$ s.t.~\eqref{eq:proofoedefyjs} holds true. The definition of the points $y_1,\ldots,y_{\ell-1}$ requires $\ell>1$. However, for the case $\ell=1$ the inequality~\eqref{eq:proofoedefyjs} simplifies to~$g_1(x)-\omega x< \alpha-\varepsilon$ for $x\in[-1,1]$ and the following steps of the proof apply with minimal adaptions to the notation.

We proceed to construction $r_2\in\mathcal{U}_n$ as in~\ref{item:equiproof3}. Since $r_1$ has minimal degree $m\leq n$, we have $r_1=p^\dag/p$ for a polynomial $p$ of degree exactly $m$.
Let $\delta p$ be a polynomial of degree~$\leq n$ which will be specified later.
We define
\begin{equation}\label{eq:r2frompdp}
r_2 = \frac{(p+\delta p)^\dag}{p+\delta p} 
= \frac{p^\dag + \delta p^\dag }{p + \delta p}\in\mathcal{U}_n.
\end{equation}
The difference between $r_2$ and $r_1$ corresponds to
\begin{equation*}
\begin{aligned}
r_2 - r_1 &= \frac{p^\dag + \delta p^\dag }{p + \delta p} - \frac{p^\dag}{p}
=  \frac{(p^\dag + \delta p^\dag )p-p^\dag(p + \delta p)}{p(p + \delta p)}\\
&=  \frac{\delta p^\dag p-p^\dag  \delta p}{p^2 + p\delta p} 
=  \frac{\delta p^\dag p-p^\dag  \delta p}{p^2} + \mathcal{O}\big(|\delta p|^2\big),~~~\text{for $|\delta p|\to0$},
\end{aligned}
\end{equation*}
where $|\delta p|$ has to be understood as the maximum over the considered argument set.
On the imaginary axis, where $p^\dag(\mathrm{i} x)=\overline{p(\mathrm{i} x)}$ and $\delta p^\dag(\mathrm{i} x)=\overline{\delta p(\mathrm{i} x)}$, this simplifies to
\begin{equation}\label{eq:r2minusr}
r_2(\mathrm{i} x) - r_1(\mathrm{i} x) = \frac{2 \mathrm{i} \operatorname{Im}(\overline{\delta p(\mathrm{i} x)} p(\mathrm{i} x))}{p(\mathrm{i} x)^2} + \mathcal{O}\big(\|\delta p\|^2\big),~~~\text{for $\|\delta p\|\to0$ and $x\in[-1,1]$},
\end{equation}
where
\begin{equation}\label{eq:r2minusremasym}
\frac{\operatorname{Im}(\overline{\delta p(\mathrm{i} x)} p(\mathrm{i} x))}{p(\mathrm{i} x)^2} = \mathcal{O}(\|\delta p\|).
\end{equation}
Furthermore, expanding the denominator of this quotient we observe
\begin{equation*}
\frac{1}{p(\mathrm{i} x)^2}
 =\frac{1}{\overline{p(\mathrm{i} x)} p(\mathrm{i} x)} \cdot \frac{\overline{p(\mathrm{i} x)}}{p(\mathrm{i} x)}
=\frac{r_1(\mathrm{i} x)}{|p(\mathrm{i} x)|^2}.
\end{equation*}
Substituting this identity in~\eqref{eq:r2minusr}, we arrive at
\begin{equation}\label{eq:optimaltoequir2tor1x}
r_2(\mathrm{i} x) = r_1(\mathrm{i} x) \cdot\left(1 + 2\mathrm{i} \frac{\operatorname{Im}(\overline{\delta p(\mathrm{i} x)} p(\mathrm{i} x))}{|p(\mathrm{i} x)|^2}\right) + \mathcal{O}\big(\|\delta p\|^2\big).
\end{equation}
Considering the asymptotic behavior~\eqref{eq:r2minusremasym}, we note
\begin{equation*}
\mathrm{e}^{2\mathrm{i} \operatorname{Im}(\overline{\delta p(\mathrm{i} x)} p(\mathrm{i} x))/|p(\mathrm{i} x)|^2} =
1 + 2\mathrm{i} \frac{\operatorname{Im}(\overline{\delta p(\mathrm{i} x)} p(\mathrm{i} x))}{|p(\mathrm{i} x)|^2} +\mathcal{O}(\|\delta p\|^2).
\end{equation*}
Thus, substituting this and $r_1(\mathrm{i} x)=\mathrm{e}^{\mathrm{i} g_1(x)}$ in~\eqref{eq:optimaltoequir2tor1x} we arrive at
\begin{equation}\label{eq:optoeqr2exp0}
r_2(\mathrm{i} x) = \mathrm{e}^{\mathrm{i} (g_1(x) + 2 \operatorname{Im}(\overline{\delta p(\mathrm{i} x)} p(\mathrm{i} x))/|p(\mathrm{i} x)|^2)} + \mathcal{O}(\|\delta p\|^2).
\end{equation}
On the other hand, combining~\eqref{eq:r2minusr} and~\eqref{eq:r2minusremasym} we deduce  that~\eqref{eq:opttoequir1conds2} implies~$ \|r_2 - \exp(\omega \cdot)\|<2$ if $\|\delta p\|$ is sufficiently small. Thus, under this assumption we find a unique $g_2$ with $r_2(\mathrm{i} x) = \mathrm{e}^{\mathrm{i} g_2(x)}$ as in Proposition~\ref{prop:approxertophaseerrinequ}, and~\eqref{eq:optoeqr2exp0} corresponds to
\begin{equation}\label{eq:optoeqr2exp0k1}
g_2(x) = g_1(x) + 2 \frac{\operatorname{Im}(\overline{\delta p(\mathrm{i} x)} p(\mathrm{i} x))}{|p(\mathrm{i} x)|^2} +2k\pi +\mathcal{O}(\|\delta p\|^2),~~~k\in\mathbb{N}.
\end{equation}
Moreover, $g_1$ and $g_2$ both have a phase error $<\pi$, i.e., $|g_1(x)-\omega x|<\pi$ and $|g_2(x)-\omega x|<\pi$, and thus, $|g_1(x)-g_2(x)|<2\pi$ for $x\in[-1,1]$. As a consequence, for a sufficiently small $\|\delta p\|$ we observe $k=0$, and~\eqref{eq:optoeqr2exp0k1} simplifies to
\begin{equation}\label{eq:g2togviadeltap}
g_2(x) = g_1(x) + 2 \frac{\operatorname{Im}(\overline{\delta p(\mathrm{i} x)} p(\mathrm{i} x))}{|p(\mathrm{i} x)|^2} +\mathcal{O}(\|\delta p\|^2).
\end{equation}
Combining~\eqref{eq:r2minusremasym} and~\eqref{eq:g2togviadeltap}, we conclude
\begin{equation*}
|g_2(x) - g_1(x)| =\mathcal{O}(\|\delta p \|),~~~x\in[-1,1],
\end{equation*}
which concludes~\eqref{eq:dummygamma2epsdist} for a given $\varepsilon>0$ and $\|\delta p\|$ sufficiently small.

We proceed to specify $\delta p$ in consideration of~\eqref{eq:dummygamma2smallerlarger}. In particular, this imposes some conditions on $\operatorname{Im}(\overline{\delta p(\mathrm{i} x)} p(\mathrm{i} x))$ in~\eqref{eq:g2togviadeltap}.
To this end, we define
\begin{equation*}
p(\mathrm{i} x) =: a(x) + \mathrm{i} b(x),
\end{equation*}
where $a$ and $b$ are real polynomials, i.e., mapping $\mathbb{R}$ to $\mathbb{R}$, of degree exactly~$m$.
The polynomial $p$ has no zeros on the imaginary axis since such points would correspond to common zeros of $p$ and $p^\dag$ which is contrary to $r_1$ being irreducible, cf.\ Proposition~\ref{prop:fulldegree}.
This further implies that the polynomials~$a$ and~$b$ have no common zeros on the real axis.

Our next step is to show existence of real polynomials $\delta a$ and $\delta b$ of degree $\leq n$ s.t.
\begin{equation}\label{eq:dummygamma2smallerlarger0}
\delta a(x) b(x) - a(x)\delta b(x) \left\{\begin{array}{ll}
> 0 & x\in [-1,y_1)\cup(y_2,y_3)\cup\cdots,~~~\text{and}\\
< 0 & x\in (y_1,y_2)\cup(y_3,y_4)\cup\cdots,
\end{array}\right.
\end{equation}
where the real polynomials $a$ and $b$ correspond to $p=a+\mathrm{i} b$ and the points $y_1,\ldots,y_{\ell-1}$ correspond to~\eqref{eq:dummygamma2smallerlarger}.
We recall that $\delta a b - a\delta b$ is a polynomial of degree $\leq n+m$ and $\ell-1 \leq m+n$, and the inequalities~\eqref{eq:dummygamma2smallerlarger0} hold true if this polynomial corresponds to
\begin{equation}\label{eq:dummygamma2smallerlarger01}
\delta a(x) b(x) - a(x)\delta b(x) = c(-1)^{\ell-1}(x-y_1)(x-y_2)\cdots(x-y_{\ell-1}),
\end{equation}
for a pre-factor $c>0$ which is specified by other conditions further below. We apply the Fredholm alternative of linear algebra (similar to the proof of Theorem~24.1 in~\cite{Tre13}) to show that the mapping from $\delta a$ and $\delta b$ to the polynomial $\delta a b - a\delta b$ is surjective. Certainly, this mapping is linear, the choice of~$\delta a$ and~$\delta b$ corresponds to a $(2n+2)$-dimensional space and $\delta a b - a \delta b$ corresponds to the $(m+n+1)$-dimensional space of real polynomials of degree $\leq m+n$. To show that this mapping is surjective, it is enough to show that its kernel has at most dimension $n-m+1$. Assuming $\delta a b - a\delta b =0 $ then $\delta a b = a \delta b$. Since $a$ and $b$ have no common zeros, all the roots of $a$ must be roots of $\delta a$ and all the roots of $b$ must be roots of $\delta b$. Thus, $a=\delta a g$ and $b=\delta b g $ for some polynomial $g$ of degree $\leq n-m$. The set of polynomials of degree $n-m$ has dimension $n-m+1$ which completes this argument.
Thus, there exists $\delta a$ and $\delta b$ s.t.~\eqref{eq:dummygamma2smallerlarger01} holds true, and this implies~\eqref{eq:dummygamma2smallerlarger0}. 
With these polynomials $\delta a$ and $\delta b$ we define $\delta p$ as 
\begin{equation*}
\delta p(\mathrm{i} x) := \delta a(x) + \mathrm{i} \delta b(x).
\end{equation*}
This definition extends to complex arguments, $\delta p(z)$ for $z\in\mathbb{C}$, in a direct manner. The pre-factor $c>0$ in~\eqref{eq:dummygamma2smallerlarger01} corresponds to a common scaling factor of $\delta a$ and $\delta b$ which we may use to re-scaled $\delta p$ s.t.\ $\|\delta p\|$ is arbitrary small. Moreover,
\begin{equation}\label{eq:Imdpptodabdba}
\operatorname{Im}(\overline{\delta p(\mathrm{i} x)} p(\mathrm{i} x))
= \operatorname{Im}\big((\delta a(x) - \mathrm{i} \delta b(x))(a(x) + \mathrm{i} b(x)) \big)
= \delta a(x) b(x) - a(x)\delta b(x).
\end{equation}
Combining~\eqref{eq:g2togviadeltap},~\eqref{eq:dummygamma2smallerlarger0} and~\eqref{eq:Imdpptodabdba}, we conclude that for a sufficiently small $\varepsilon>0$ and sufficiently small scaling of $\delta p$, the inequalities~\eqref{eq:dummygamma2inequalities} hold true.

We proceed with~\ref{item:equiproof4}. Namely, from the conditions~\eqref{eq:proofoedefyjs} and~\eqref{eq:dummygamma2inequalities} we deduce
\begin{equation}\label{eq:g2isbetterdummy}
\max_{x\in[-1,1]} | g_2(x) - \omega x| < \max_{x\in[-1,1]} | g_1(x) - \omega x| = \alpha.
\end{equation}
We first show
\begin{equation}\label{eq:g2isbetterdummya}
g_2(x) - \omega x < \alpha,~~~x\in[-1,1].
\end{equation}
For $x \in [-1, y_1 + \varepsilon] \cup [y_2-\varepsilon, y_3 + \varepsilon] \cup \ldots$, the upper bound in~\eqref{eq:proofoedefyjs} and $g_2(x)< g_1(x) + \varepsilon$ from~\eqref{eq:dummygamma2epsdist} show
\begin{equation*}
g_2(x) - \omega x <  g_1(x) - \omega x + \varepsilon < \alpha,
~~~~x \in [-1, y_1 + \varepsilon] \cup [y_2-\varepsilon, y_3 + \varepsilon] \cup \ldots.
\end{equation*}
For $x \in [y_1 + \varepsilon, y_2-\varepsilon] \cup [y_3 + \varepsilon, y_4-\varepsilon] \cup \ldots$ we have $ g_2(x) < g_1(x) $~\eqref{eq:dummygamma2smallerlarger}, and together with $ |g_1(x) - \omega x| \leq  \alpha$ this shows
\begin{equation*}
g_2(x) - \omega x < g_1(x) - \omega x \leq \alpha,
~~~~ x \in [y_1 + \varepsilon, y_2-\varepsilon] \cup [y_3 + \varepsilon, y_4-\varepsilon] \cup \ldots.
\end{equation*}
We proceed to show
\begin{equation}\label{eq:g2isbetterdummyb}
g_2(x) - \omega x > -\alpha,~~~x\in[-1,1].
\end{equation}
For $x \in [y_1-\varepsilon, y_2 + \varepsilon] \cup [y_3-\varepsilon, y_4+ \varepsilon] \cup\ldots$ the lower bound in~\eqref{eq:proofoedefyjs}
and $g_2(x) > g_1(x) - \varepsilon$ from~\eqref{eq:dummygamma2epsdist} entail
\begin{equation*}
g_2(x) - \omega x > g_1(x) - \omega x - \varepsilon > -\alpha,
~~~~ x \in [y_1-\varepsilon, y_2 + \varepsilon] \cup [y_3-\varepsilon, y_4+ \varepsilon] \cup\ldots.
\end{equation*}
For $x \in [-1,y_1 -\varepsilon]\cup[y_2 + \varepsilon, y_3-\varepsilon] \cup\ldots$ we have $ g_2(x) > g_1(x) $~\eqref{eq:dummygamma2smallerlarger}, and $ |g_1(x) - \omega x| \leq \alpha$, which implies
\begin{equation*}
g_2(x) - \omega x > g_1(x) - \omega x \geq -\alpha,
~~~~ x \in [-1,y_1 -\varepsilon]\cup[y_2 + \varepsilon, y_3-\varepsilon] \cup\ldots.
\end{equation*}
Combining~\eqref{eq:g2isbetterdummya} and~\eqref{eq:g2isbetterdummyb} we conclude~\eqref{eq:g2isbetterdummy}.

Following Proposition~\ref{prop:approxertophaseerrinequ}, the inequality~\eqref{eq:g2isbetterdummy} entails
\begin{equation*}
\| r_2 - \exp(\omega\cdot)\| 
< \| r_1 - \exp(\omega\cdot)\|,
\end{equation*}
which is contrary to $r_1$ attaining the minimal approximation error.
We conclude that the phase error has at least $m+n+2$ equioscillating points.

We finalize our proof by showing~\ref{item:equiproof5}. Provided $s_1,\ldots,s_m$ denote the poles of $r_1$ with $s_j=\xi_j+\mathrm{i} \mu_j$, the phase error $g_1(x)-\omega x$ has the derivative
\begin{equation}\label{eq:ddxphaseerror}
\begin{aligned}
\frac{\mathrm{d}}{\mathrm{d} x}(g_1(x)-\omega x) &= 2 \sum_{j=1}^m \frac{1}{1+\big(\frac{x-\mu_j}{\xi_j}\big)^2} \cdot \frac{1}{\xi_j} - \omega
=  2 \sum_{j=1}^m \frac{\xi_j}{\xi_j^2+(x-\mu_j)^2} - \omega\\
&=  2 \sum_{j=1}^m \frac{\xi_j}{|\mathrm{i} x-s_j|^2} - \omega.
\end{aligned}
\end{equation}
This corresponds to a partial fraction.
Due to $m$ being the minimal degree of $r$ and Proposition~\ref{prop:fulldegree}, we have $\overline{s}_j\neq -s_\iota$ for $j,\iota\in\{1,\ldots,m\}$.
Thus, the denominators $|\mathrm{i} x-s_j|^2$, $j=1,\ldots,m$, of the partial fraction above are distinct if and only if the poles $s_1,\ldots,s_m$ are distinct.
As a consequence, the derivative~\eqref{eq:ddxphaseerror} of the phase error corresponds to a rational function $\widetilde{p}/\widetilde{q}$ where $\widetilde{p}$ and $\widetilde{q}$ denote polynomials of degree $\leq 2k$ for $k\leq m$ denoting the number of distinct poles of $r_1$.
Thus, the derivative of the phase error has at most $2k$ zeros counting multiplicity. Equioscillation points in $(-1,1)$ are necessarily extreme values of the phase error, and thus, zeros of its derivative. Consequently,  $ g_1(x) - \omega x $  has at most $2k$ equioscillation points in $(-1,1)$, and $2k+2$ many in $[-1,1]$.

Since we show that the phase error of $r_1$ has at least $m+n+2$ equioscillation points in~\ref{item:equiproof4}, and we show that the phase error has at most $2k+2$ equioscillation points in~\ref{item:equiproof5} where $k\leq m\leq n$, we conclude that the phase error of $r_1$ has exactly $2n+2$ equioscillating points, i.e.,\ $k=m=n$.
Since $k$ and $m$ denote the number of distinct poles and the minimal degree of $r_1$, respectively, this shows~\ref{item:rwmindegreen}.

We recall that equioscillation points in $(-1,1)$ are zeros of the derivative of the phase error and there are at most $2n$ located in $(-1,1)$ counting multiplicity. As a consequence, we conclude that exactly $2n$ equioscillation points are located in $(-1,1)$ and these correspond to the simple zeros of the derivative of the phase error. Moreover, the remaining two equioscillation points are located at $-1$ and $1$. This shows~\ref{item:rwequi} and~\ref{item:rwphaseerrmonotonic}.
\end{proof}

\begin{proposition}[Uniqueness of unitary best approximants]\label{prop:p3uniqueness}
Provided $n$ is a fixed degree and $\omega\in(0,(n+1)\pi)$,
the best approximation $r\in\mathcal{U}_n$ to $\mathrm{e}^{\mathrm{i} \omega x}$ is unique.
\end{proposition}
\begin{proof} 
In contrast to Proposition~\ref{prop:p11equitoopti} we now to consider the case that two unitary functions $r_1,r_2\in\mathcal{U}_n$ both minimize the approximation error. 
Some of the arguments in this proof also appear in the proof~\cite[Theorem~10.1]{Tre13} in a slightly different setting.

Let $r_1\in\mathcal{U}_n$ denote a best approximation which, since $\omega\in(0,(n+1)\pi)$, attains an approximation error $\| r_1 - \exp(\omega\cdot)\| < 2$.
Assume there exists $r_2\in\mathcal{U}_n$ which attains the same approximation error as $r_1$,
\begin{equation*}
\|r_2 - \exp(\omega \cdot)\| = \|r_1 - \exp(\omega \cdot)\| < 2.
\end{equation*}
As a consequence of  Proposition~\ref{prop:approxertophaseerrinequ},
for $g_1$ with $r_1(\mathrm{i} x) = \mathrm{e}^{\mathrm{i} g_1(x) }$ and $g_2$ with $r_2(\mathrm{i} x) = \mathrm{e}^{\mathrm{i} g_2(x) }$
this implies 
\begin{equation*}
\max_{x\in[-1,1]} | g_2(x) - \omega x|
= \max_{x\in[-1,1]} | g_1(x) - \omega x|=:\alpha.
\end{equation*}
Due to optimality of $r_1$, Proposition~\ref{prop:p2optitoequi} implies that $g_1(x)-\omega x$ equioscillates between $2n+2$ points.
Let $\eta_1<\ldots<\eta_{2n+2}$ denote the equioscillation points of $g_1(x)-\omega x$ and assume $g_1(\eta_1) - \omega \eta_1 = -\alpha$ which implies $g_2(\eta_1)-\omega \eta_1 \geq g(\eta_1)-\omega \eta_1$.
Due to equioscillation of $g_1(x)-\omega x$ we have 
\begin{equation}\label{eq:g2minusgatxj}
g_2(\eta_j) - g_1(\eta_j) =
\left\{\begin{array}{ll}
\geq 0,&~\text{for $j=1,3,\ldots,2n+1$, and,}\\
\leq 0,&~\text{for $j=2,4,\ldots,2n+2$.}
\end{array}\right.
\end{equation}
Thus, the function
\begin{equation*}
f(x):= g_2(x) - g_1(x)
\end{equation*}
is zero at least once in $[\eta_j,\eta_{j+1}]$ for $j=1,\ldots, 2n+1$.
We proceed to show that 
\begin{equation}\label{eq:fjm1zerosisubeta1etaj}
\text{$f$ has at least $j-1$ zeros counting multiplicity in $[\eta_1,\eta_j]$,}
\end{equation}
for $j=2,\ldots, 2n+2$, by induction.
\begin{itemize}
\item The case $j=2$. From~\eqref{eq:g2minusgatxj}, since $ g_2(\eta_1) - g_1(\eta_1) \geq 0$ and $ g_2(\eta_2) - g_1(\eta_2) \leq 0$, the function $f$ has at least one zero in $[\eta_1,\eta_2]$. Thus, the statement~\eqref{eq:fjm1zerosisubeta1etaj} holds true for $j=2$.

\item We prove the induction step from $k$ to $k+1$ by contradiction. Assume the statement~\eqref{eq:fjm1zerosisubeta1etaj} holds true for $j=1,\ldots,k$, in particular, $f$ has at least $k-1$ zeros in the interval $[\eta_1,\eta_k]$, and assume~\eqref{eq:fjm1zerosisubeta1etaj} is false for $j=k+1$, i.e., $f$ has $<k$ zeros in $[\eta_1,\eta_{k+1}]$. This implies that $f$ has exactly $k-1$ zeros in $[\eta_1,\eta_k]$ counting multiplicity, and $f$ has no zero in $(\eta_k,\eta_{k+1}]$.

Since $f$ has at least one zero in each interval $[\eta_j,\eta_{j+1}]$, this implies $f(\eta_k)=0$. Moreover, since $f$ has exactly $k-1$ zeros in $[\eta_1,\eta_k]$ and $k-2$ zeros in $[\eta_1,\eta_{k-1}]$ due to our induction assumption, the zero of $f$ at $\eta_k$ has multiplicity one, and $f$ has no other zeros in the interval $(\eta_{k-1},\eta_{k+1}]$. The inequalities~\eqref{eq:g2minusgatxj} imply that $f$ has alternating signs or is zero at the points $\eta_{k-1},\eta_k,\eta_{k+1}$. Namely, $f(\eta_{k+1})\geq 0$ (or $\leq 0$), $f(\eta_{k})\leq 0$ (or $\geq 0$) and $f(\eta_{k-1})\geq 0$ (or $\leq 0$), and since the only zero of $f$ in $(\eta_{k-1},\eta_{k+1}]$ is a simple zero located at $\eta_k$, we conclude $f(\eta_{k-1})=0$. Repeatedly applying these arguments, we observe that $f$ has simple zeros at $\eta_2,\ldots,\eta_k$.

Due to the location of the zeros of $f$, the sign of $f(\eta_1)$ is equal (opposite) to the sign of $f(\eta_{k+1})$ for $k$ even (odd). However, this is in contradiction to the sign of $f$ from~\eqref{eq:g2minusgatxj}. Thus, the assumption that $[\eta_1,\eta_{k+1}]$ has $<k$ zeros counting multiplicity leads to a contradiction, and our induction step holds true. 
\end{itemize}
The induction above shows~\eqref{eq:fjm1zerosisubeta1etaj}, and thus, $f=g_2-g_1$ has at least $2n+1$ zeros in $[\eta_1,\eta_{2n+2}]$ counting multiplicity. As a consequence, Proposition~\ref{prop:gidentitytor} entails $r_1=r_2$ which proves uniqueness of the unitary best approximant.
\end{proof}

\section{Symmetry}\label{sec:symmetry}
In the present work, we refer to a function $r$ as symmetric if
\begin{equation*}\tag{\ref{eq:defsym}}
r(-z) = r(z)^{-1},~~~ z\in \mathbb{C}.
\end{equation*}
A property which certainly holds true for the exponential function, namely, $\mathrm{e}^{-z}=(\mathrm{e}^z)^{-1}$. We use the term symmetry for an approximation $r$ to the exponential function in reference to the respective property of time integrators. Before showing that unitary best approximations are symmetric in Proposition~\ref{prop:sym} below, we state some equivalent definitions for symmetry.
\begin{proposition}\label{prop:symequivstatements}
Let $r\in\mathcal{U}_n$ with minimal degree $n$, and let $\mathrm{e}^{\mathrm{i} \theta}$ denote the complex phase of $r$ as in~\eqref{eq:defrbyp}. The following properties are equivalent.
\begin{enumerate}[label=(\roman*)]
\item\label{item:s2.1symmetric} $r$ is symmetric~\eqref{eq:defsym},
\item\label{item:s2.1xsymUn} $r$ satisfies
\begin{equation}\label{eq:rissymid0b}
r(-\mathrm{i} x) = \overline{r(\mathrm{i} x)},~~~x\in\mathbb{R},
\end{equation}
\item\label{item:s2.2evenodd} the real and imaginary parts of $r(\mathrm{i} x)$ are even and odd, respectively,
\item\label{item:s2.3poles} the poles of~$r$ are either real or come in complex conjugate pairs, and $r$ has the complex phase $\mathrm{e}^{\mathrm{i}\theta}= r(0) \in\{-1,1\}$, and
\item\label{item:s2.4realp} $r(z)=\sigma \rho(-z)/\rho(z)$ where $\rho$ is a real polynomial of degree exactly $n$, and $\sigma = r(0)\in\{-1,1\}$.
\end{enumerate}
\end{proposition}
\begin{proof}
Unitarity $|r(\mathrm{i} x)|=1$ implies $\overline{r(\mathrm{i} x)} = r(\mathrm{i} x)^{-1}$ for $x\in\mathbb{R}$. Substituting $\mathrm{i} x$ for $z$ in~\eqref{eq:defsym} and applying this identity, we conclude that~\ref{item:s2.1symmetric} entails~\ref{item:s2.1xsymUn}. On the other hand, for $r\in\mathcal{U}_n$ the identity~\eqref{eq:rissymid0b} corresponds to
\begin{equation}\label{eq:defsymix}
r(-\mathrm{i} x) = r(\mathrm{i} x)^{-1},~~~x\in\mathbb{R}.
\end{equation}
The functions on the left and right-hand side, i.e., $r(-z)$ and $r(z)^{-1}$, are in $\mathcal{U}_n$ since $r\in\mathcal{U}_n$. Since rational functions of degree $(n,n)$ which coincide at $2n+1$ or more points are identical, the identity~\eqref{eq:defsymix} certainly implies~\eqref{eq:defsym}. We conclude that~\ref{item:s2.1symmetric} and~\ref{item:s2.1xsymUn} are equivalent. 

\medskip
We show that~\ref{item:s2.1xsymUn} and~\ref{item:s2.2evenodd} are equivalent.
Considering the real and imaginary parts of $x\mapsto r(\mathrm{i} x)$ for $x\in\mathbb{R}$, we have
\begin{equation*}
r(-\mathrm{i} x) = \operatorname{Re} r(-\mathrm{i} x)+\mathrm{i} \operatorname{Im} r(-\mathrm{i} x),~~~\text{and}~~\overline{r(\mathrm{i} x)} = \operatorname{Re} r(\mathrm{i} x)-\mathrm{i} \operatorname{Im} r(\mathrm{i} x) .
\end{equation*}
Thus,~\eqref{eq:rissymid0b} holds true if and only if
\begin{equation*}
\operatorname{Re} r(-\mathrm{i} x) = \operatorname{Re} r(\mathrm{i} x),~~~\text{and}~~\operatorname{Im} r(-\mathrm{i} x) = -\operatorname{Im} r(\mathrm{i} x),
\end{equation*}
which shows that~\ref{item:s2.1xsymUn} and~\ref{item:s2.2evenodd} are equivalent.

\medskip
We proceed to show that~\ref{item:s2.1xsymUn} implies~\ref{item:s2.4realp}. We recall the representation~\eqref{eq:defrbyp} for $r\in\mathcal{U}_n$ with minimal degree $n$. Let $s_1,\ldots,s_n\in\mathbb{C}$ and $\mathrm{e}^{\mathrm{i} \theta}$ denote the poles and the complex phase of $r$, respectively, where $|\mathrm{e}^{\mathrm{i} \theta}|=1$. Then $r$ is of the form
\begin{equation}\label{eq:defrbypagain}\tag{\ref*{eq:defrbyp}$\star$}
r(z) = (-1)^n \mathrm{e}^{\mathrm{i} \theta}  \prod_{j=1}^{n} \frac{z+\overline{s}_j}{z-s_j},~~~z\in\mathbb{C}.
\end{equation}
Using the representation~\eqref{eq:defrbypagain} for $r(-z)$ and $r(z)^{-1}$, the identity~\eqref{eq:defsym} reads
\begin{equation}\label{eq:defrbyprminusrinv}
\mathrm{e}^{\mathrm{i} \theta}  \prod_{j=1}^{n} \frac{z-\overline{s}_j}{z+s_j}
= \mathrm{e}^{-\mathrm{i} \theta}  \prod_{j=1}^{n} \frac{z-s_j}{z+\overline{s}_j},~~~z\in\mathbb{C}.
\end{equation}
Since we assume that $r\in\mathcal{U}_n$ has minimal degree $n$, the representation~\eqref{eq:defrbypagain} is irreducible and so are the rational functions in~\eqref{eq:defrbyprminusrinv}. As a consequence,~\eqref{eq:defrbyprminusrinv} implies
\begin{equation}\label{eq:rhoids}
\mathrm{e}^{\mathrm{i} \theta}=\mathrm{e}^{-\mathrm{i} \theta},~~~\text{and}~~
\prod_{j=1}^{n} (z-\overline{s}_j) = \prod_{j=1}^{n} (z-s_j) ~~ =:\rho(z),~~~z\in\mathbb{C}.
\end{equation}
While the first identity therein implies $\mathrm{e}^{\mathrm{i}\theta}\in\{-1,1\}$, the latter entails $\overline{\rho(z)} = \rho(\overline{z})$, i.e., $\rho$ is a polynomial with real coefficients. We remark that the numerator in~\eqref{eq:defrbypagain} corresponds to
\begin{equation*}
(-1)^n \prod_{j=1}^{n} (z+\overline{s}_j) = \prod_{j=1}^{n} (-z-\overline{s}_j) = \rho(-z).
\end{equation*}
Substituting $\rho$ in~\eqref{eq:defrbypagain}, we observe
\begin{equation}\label{eq:defrbypgamma}
r(z) = \sigma \frac{\rho(-z)}{\rho(z)},~~~\sigma:=\mathrm{e}^{\mathrm{i} \theta}\in\{-1,1\},~~~z\in\mathbb{C}.
\end{equation}
Substituting $z=0$ therein, we observe $r(0) = \sigma = \mathrm{e}^{\mathrm{i} \theta}$. Thus,~\ref{item:s2.1xsymUn} implies~\ref{item:s2.4realp}.

Moreover, we show that~\ref{item:s2.3poles} is equivalent to~\ref{item:s2.4realp}. Considering~\ref{item:s2.3poles}, we have $r$ of the form~\eqref{eq:defrbypagain} with real or complex conjugate pairs of poles and $\mathrm{e}^{\mathrm{i}\theta}\in\{-1,1\}$. Thus,~\eqref{eq:rhoids} holds true and this shows~\eqref{eq:defrbypgamma} as above, i.e., we arrive at~\ref{item:s2.4realp}. On the other hand, assume~\ref{item:s2.4realp} holds true, then $r(z)=\sigma \rho(-z)/\rho(z)$ corresponds to $r=p^\dag/p$ with $p=\rho$ for $\sigma=1$ and $p=\mathrm{i}\rho$ for $\sigma=-1$. Thus, $\sigma$ corresponds to the complex phase $\mathrm{e}^{\mathrm{i} \theta}$ of $r$ and since $\rho$ has real coefficients, its zeros are real or come in complex conjugate pairs which carries over to the poles of $r$. Hence,~\ref{item:s2.3poles} and~\ref{item:s2.4realp} are equivalent.

\medskip
We proceed to show that~\ref{item:s2.4realp} implies~\ref{item:s2.1xsymUn}. Consider $r(\mathrm{i} x) = \sigma \rho(-\mathrm{i} x)/\rho(\mathrm{i} x)$ where $\sigma\in\{-1,1\}$ and $\rho$ is a polynomial with real coefficients, i.e., $\overline{\rho(z)} = \rho(\overline{z})$. This identity entails
\begin{equation*}
\overline{r(\mathrm{i} x)}
= \sigma \overline{\left(\frac{\rho(-\mathrm{i} x)}{\rho(\mathrm{i} x)}\right)}
= \sigma \frac{\rho(\mathrm{i} x)}{\rho(-\mathrm{i} x)}
= r(-\mathrm{i} x),~~~x\in\mathbb{R},
\end{equation*}
and thus, the identity~\eqref{eq:rissymid0b} holds true, which shows~\ref{item:s2.1xsymUn}.
\end{proof}

\begin{proposition}[Symmetry]\label{prop:sym}
Provided $n$ is a fixed degree and $\omega\in(0,(n+1)\pi)$, the best approximation to $\mathrm{e}^{\mathrm{i} \omega x}$ in $\mathcal{U}_n$ is symmetric.
\end{proposition}
\begin{proof}
Let $r\in\mathcal{U}_n$ denote the best approximation to $\mathrm{e}^{\mathrm{i} \omega x}$.
For $r\in\mathcal{U}_n$ the rational function $\overline{r}$, which satisfies $ \overline{r}(\mathrm{i} x) = \overline{r(-\mathrm{i} x)} $ on the imaginary axis, is also unitary, i.e., $\overline{r}\in\mathcal{U}_n$. Moreover, the approximation error of $ \overline{r(-\mathrm{i} x)}\approx \mathrm{e}^{\mathrm{i}\omega x} $ corresponds to
\begin{equation*}
\max_{x\in[-1,1]}|\overline{r(-\mathrm{i} x)} - \mathrm{e}^{\mathrm{i} \omega x}|
= \max_{x\in[-1,1]}|r(-\mathrm{i} x) - \mathrm{e}^{-\mathrm{i} \omega x}|
= \max_{x\in[-1,1]}|r(\mathrm{i} x) - \mathrm{e}^{\mathrm{i} \omega x}|.
\end{equation*}
This shows that $r$ and $\overline{r(-\mathrm{i} x)}$ are both unitary best approximants, and since we show uniqueness of the unitary best approximation in Theorem~\ref{thm:bestapprox}, we conclude $\overline{r(-\mathrm{i} x)} = r(\mathrm{i} x)$.
This entails symmetry~\eqref{eq:rissymid0b}.
\end{proof}

\begin{corollary}\label{cor:gsym}
As shown in Proposition~\ref{prop:sym}, the best approximation $r\in\mathcal{U}_n$ is symmetric for $\omega\in(0,(n+1)\pi)$. We note some consequences of symmetry.

Proposition~\ref{prop:symequivstatements}.\ref{item:s2.2evenodd} shows that $r(\mathrm{i} x)$ has an odd imaginary part which especially implies $\operatorname{Im} r(0)=0$. Since $|r(\mathrm{i} x)|=1$, this entails $r(0)=-1$ or $r(0)=1$. Since $\|r-\exp(\omega \cdot)\|<2$, we have $|r(0)-1|<2$ and we conclude $r(0)=1$. Due to Proposition~\ref{prop:symequivstatements}.\ref{item:s2.3poles}, this carries over to the complex phase of $r$, namely, $\mathrm{e}^{\mathrm{i}\theta}=1$. Moreover, according to Proposition~\ref{prop:symequivstatements}.\ref{item:s2.4realp} the unitary best approximation is of the form
\begin{equation*}
r(z) = \frac{p(-z)}{p(z)},~~~\text{where $p$ is a real polynomial of degree exactly $n$}.
\end{equation*}

Consider the representation $r(\mathrm{i} x)=\mathrm{e}^{\mathrm{i} g(x)}$. We recall the definition of $g$ from~\eqref{eq:defg} for the case that poles of $r$ are subject to Proposition~\ref{prop:symequivstatements}.\ref{item:s2.3poles}, i.e., poles are real or come in complex conjugate pairs. Let $\xi_1,\ldots,\xi_{n_1}\in\mathbb{R}$ denote real poles, and let $ s_{n_1+1},\ldots,s_{n_1+n_2}\in\mathbb{C}$ with $s_j=\xi_j+\mathrm{i}\mu_j$ denote complex poles which occur in pairs $s_j$ and $\overline{s}_j$ with $n_1+2n_2=n$. Then $g$ corresponds to
\begin{equation}\label{eq:defgagain}\tag{\ref*{eq:defg}$\star$}
g(x) = \theta + 2 \sum_{j=1}^{n_1}\arctan\left( \frac{x}{\xi_j}\right) + 2 \sum_{j=n_1+1}^{n_1+n_2}\left(\arctan\left( \frac{x+\mu_j}{\xi_j}\right)+\arctan\left( \frac{x-\mu_j}{\xi_j}\right)\right).
\end{equation}
Since the arc tangents is an odd function, terms in the first sum therein are also odd, namely,
\begin{equation*}
\arctan\left( \frac{-x}{\xi_j}\right) = -\arctan\left( \frac{x}{\xi_j}\right),
\end{equation*}
for $j=1,\ldots,n_1$. Inserting $-x$ for terms of the second sum in~\eqref{eq:defgagain} and making use of arc tangents being odd, we observe
\begin{equation*}
\arctan\left( \frac{-x+\mu_j}{\xi_j}\right)+\arctan\left( \frac{-x-\mu_j}{\xi_j} \right)
= -\arctan\left( \frac{x-\mu_j}{\xi_j}\right)-\arctan\left( \frac{x+\mu_j}{\xi_j} \right),
\end{equation*}
for $j=n_1+1,\ldots,n_1+n_2$.
Thus, both sums in~\eqref{eq:defgagain} are odd and vanish for $x=0$, i.e., $g(0)=\theta$ and $r(0)=\textrm{e}^{\textrm{i} \theta}$. It remains to show that $\theta=0$.
Following Proposition~\ref{prop:approxertophaseerrinequ}, the case $\|r-\exp(\omega \cdot)\|<2$ entails $\max_{x\in[-1,1]}  | g(x) - \omega x | < \pi$, in particular, $|g(0)|<\pi$. Since $r(0)=1$ this implies $g(0)=0$. We conclude $\theta=0$ in~\eqref{eq:defgagain}. Thus, for the unitary best approximation $r(\mathrm{i} x)=\mathrm{e}^{\mathrm{i} g(x)}$ the function $g$ is odd, i.e., $g(-x)=-g(x)$.
\end{corollary}

Moreover, symmetry also affects equioscillation points and interpolation nodes.
\begin{proposition}\label{prop:inodessym}
The interpolation nodes and equioscillation points of the unitary best approximation are mirrored around the origin.
\end{proposition}
\begin{proof}
Let $r$ denote the unitary best approximation to $\mathrm{e}^{\mathrm{i} \omega x}$ for a fixed degree $n$ and $\omega\in(0,(n+1)\pi)$, and let $r(\mathrm{i} x) =\mathrm{e}^{\mathrm{i} g(x)}$, where the representation $g$ is unique as in Proposition~\ref{prop:approxertophaseerrinequ}. Following Proposition~\ref{prop:sym}, $r$ is symmetric, and we show in Corollary~\ref{cor:gsym} that, consequently, the function $g$ is odd which entails that the phase error $g(x)-\omega x$ is odd. Thus, its zeros $x_1,\ldots,x_{2n+1}$ are mirrored around the origin. This carries over to the interpolation nodes $\mathrm{i}x_1,\ldots,\mathrm{i}x_{2n+1}$ (see Corollary~\ref{cor:errattainsmax}). In a similar manner the equioscillation points, which correspond exactly to the points of extreme value of the (odd) phase error, are also mirrored around the origin.
\end{proof}

\section{Continuity in \texorpdfstring{$\omega$}{w}}\label{sec:cont}
In Proposition~\ref{prop:wtomincontinuous} further above we show that the approximation error of the unitary best approximation is continuous in the frequency $\omega$. In the present section we clarify that the unitary best approximation to $\mathrm{e}^{\mathrm{i}\omega x}$ also depends continuously on $\omega$. This carries over to the poles, the phase function $g$, the phase error, the equioscillation points and the interpolation nodes of the unitary best approximation. 
\begin{proposition}\label{prop:minimizercontinuousonw}
For a fixed degree $n$ and $\omega\in(0,(n+1)\pi)$, the unique minimizer
\begin{equation*}
r_\omega = \operatorname*{arg\,min}_{u \in \mathcal{U}_n} \| u - \exp(\omega \cdot)\|,
\end{equation*}
depends continuously on $\omega$. In particular,
$\|r_{\omega}-r_{\omega_a}\|\to 0$ for $\omega\to\omega_a\in(0,(n+1)\pi)$.
\end{proposition}
\begin{proof}
From Theorem~\ref{thm:bestapprox} we recall that the minimizer $r_\omega$ is unique for $\omega\in(0,(n+1)\pi)$. Let $\{\omega_j\in(0,(n+1)\pi)\}_{j\in\mathbb{N}}$ denote a sequence with $\omega_j \to \omega_a\in(0,(n+1)\pi)$ and let $r_j\in\mathcal{U}_n$ be defined as
\begin{equation*}
r_j = \operatorname*{arg\,min}_{r\in\mathcal{U}_n}\|r - \exp(\omega_j\cdot )\|,~~~\text{thus,}~~
\min_{r\in\mathcal{U}_n} \|r - \exp(\omega_j\cdot )\| = \|r_j - \exp(\omega_j\cdot )\|,~~~j\in\mathbb{N}.
\end{equation*}
We further define $r_a\in\mathcal{U}_n$ as
\begin{equation*}
r_a = \operatorname*{arg\,min}_{r\in\mathcal{U}_n} \|r - \exp(\omega_a\cdot )\|,~~~\text{thus,}~~
\min_{r\in\mathcal{U}_n} \|r - \exp(\omega_a\cdot )\| = \|r_a - \exp(\omega_a\cdot )\|.
\end{equation*}
Assume $\lim_j r_j \neq r_a$ in the $\|.\|$ norm, then, for a some $\varepsilon>0$ we find a sub-sequence $r_{j_k}$ with $\|r_{j_k}-r_a\|>\varepsilon$. Due to Proposition~\ref{prop:convsubsequence}, this sub-sequence has a converging sub-sub-sequence ($r_{j_\ell}$ with $\{r_{j_\ell}\}\subset\{r_{j_k}\}$). We let $r_b\in\mathcal{U}_n$ denote the limit of $r_{j_\ell}$. Namely, $|r_{j_\ell}(\mathrm{i} x) - r_b(\mathrm{i} x)|\to 0$ for $x\in[-1,1]$ up to at most $n$ points. Due to continuity and by construction of the sub-sequence $\{r_{j_k}\}$ this implies $r_b\neq r_a$.
From Proposition~\ref{prop:wtomincontinuous}.\ref{item:erroratxiscont} we deduce continuity of the point-wise deviation, i.e.,
\begin{equation}\label{eq:erjwjpointwisetoerw}
|r_{j_\ell}(\mathrm{i} x)-\mathrm{e}^{\mathrm{i} \omega_{j_\ell} x}| \to |r_b(\mathrm{i} x)-\mathrm{e}^{\mathrm{i} \omega_a x}|,
\end{equation}
for points $x\in[-1,1]$ for which $r_{j_\ell}(\mathrm{i} x) \to r_b(\mathrm{i} x)$.

Due to continuity of $\min_{r\in\mathcal{U}_n}\|r - \exp(\omega\cdot )\|$ in $\omega$ as given in Proposition~\ref{prop:wtomincontinuous}.\ref{item:minecontinw}, we get
\begin{equation*}
 \|r_{j_\ell} - \exp(\omega_{j_\ell}\cdot )\|
= \min_{r\in\mathcal{U}_n} \|r - \exp(\omega_{j_\ell}\cdot )\|
\to \min_{r\in\mathcal{U}_n} \|r - \exp(\omega_a\cdot )\|
= \|r_a - \exp(\omega_a\cdot )\|,
\end{equation*}
for $\ell\to\infty$ where $\omega_{j_\ell} \to \omega_a$, i.e.,
\begin{equation}\label{eq:erjwjmaxtoerw}
\max_{x\in[-1,1]}|r_{j_\ell}(\mathrm{i} x)-\mathrm{e}^{\mathrm{i} \omega_{j_\ell} x}|
\to \max_{x\in[-1,1]}|r_a(\mathrm{i} x)-\mathrm{e}^{\mathrm{i} \omega_a x}|.
\end{equation}
Combining~\eqref{eq:erjwjpointwisetoerw} and~\eqref{eq:erjwjmaxtoerw}, for points $y\in[-1,1]$ for which $r_{j_\ell}(\mathrm{i} y) \to r_b(\mathrm{i} y)$, we get
\begin{equation}\label{eq:erjwjpointwisetoerw2}
\begin{aligned}
|r_b(\mathrm{i} y)-\mathrm{e}^{\mathrm{i} \omega_a y}|
& = \lim_{\ell\to\infty}|r_{j_\ell}(\mathrm{i} y)-\mathrm{e}^{\mathrm{i} \omega_{j_\ell} y}|\\
& \leq \lim_{\ell\to\infty} \max_{x\in[-1,1]}  |r_{j_\ell}(\mathrm{i} x)-\mathrm{e}^{\mathrm{i} \omega_{j_\ell} x}|
= \max_{x\in[-1,1]}|r_a(\mathrm{i} x)-\mathrm{e}^{\mathrm{i} \omega_a x}|,
\end{aligned}
\end{equation}
and due to continuity of the error in $x$, we get
\begin{equation}\label{eq:erjwjpointwisetoerw3}
\|r_b-\exp(\omega_a \cdot)\| \leq \|r_a -\exp(\omega_a\cdot)\|.
\end{equation}
This yields a contradiction since $r_a$ is the unique minimizer for $\omega=\omega_a$ and $r_b\neq r_a$ by construction.
\end{proof}

We recall that from symmetry arguments, Corollary~\ref{cor:gsym}, the complex phase of $r_\omega$ is one for $\omega \in(0,(n+1)\pi)$, i.e., $\mathrm{e}^{\mathrm{i} \theta}=1$ in the representation~\eqref{eq:defrbyp}. In the following proposition, we show that the poles of $r_\omega$ are continuous in $\omega$.
\begin{proposition}\label{prop:polescont}
For a given degree $n$ and $\omega \in(0,(n+1)\pi)$ the poles of the unitary best approximation to $\mathrm{e}^{\mathrm{i} \omega x}$ depend continuously on $\omega$.
\end{proposition}
\begin{proof}
We consider a sequence $\{\omega_j\in(0,(n+1)\pi)\}_{j\in\mathbb{N}}$ which converges $\omega_j \to \omega_a\in(0,(n+1)\pi)$. Following Proposition~\ref{prop:minimizercontinuousonw}, the corresponding sequence of unitary best approximations $r_{\omega_j}$ converges to $r_{\omega_a}$ in the $\|\cdot\|$ norm. Let $p_{\omega_j}$ and $p_{\omega_a}$ denote the denominators of $r_{\omega_j}$ and $r_{\omega_a}$, respectively.
We remark that $r_{\omega_j}$ and $r_{\omega_a}$ are non-degenerate rational functions, and respectively, $p_{\omega_j}$ and $p_{\omega_a}$ have degree exactly $n$, see Theorem~\ref{thm:bestapprox}. In the non-degenerate case, the convergence of $r_{\omega_j}$ to $r_{\omega_a}$ in the $\|\cdot\|$ norm also implies that the coefficients of $p_{\omega_j}$ converge to the coefficients of $p_{\omega_a}$, see~\cite[Theorem~1]{Gu90} for instance. Convergence of the coefficients of $p_{\omega_j}$ implies that its zeros converge to the zeros of $p_{\omega_a}$ (up to ordering, cf.~\cite[Proposition~5.2.1]{Ar11}). We recall that the zeros of $p_{\omega_j}$ correspond to the poles of $r_{\omega_j}$. Especially, we have shown that if $\omega_j\to\omega_a$, then the poles of $r_{\omega_j}$ converge to $r_{\omega_a}$ which shows our assertion on continuity.
\end{proof}
\begin{proposition}\label{prop:gcont}
Let $n$ denote a fixed degree and $\omega \in(0,(n+1)\pi)$. Let $r_\omega(\mathrm{i} x)\approx\mathrm{e}^{\mathrm{i} \omega x}$ denote the best approximation in $\mathcal{U}_n$, and let $g_\omega$ correspond to the unique representation $r_\omega(\mathrm{i} x)=\mathrm{e}^{\mathrm{i} g_\omega(x)} $ as in Proposition~\ref{prop:approxertophaseerrinequ}. Then $g_\omega$ depends continuously on $\omega$. Namely, $\max_{x\in[-1,1]} |g_{\omega_j}(x)-g_{\omega_a}(x)|\to 0 $ for a sequence $\omega_j\to\omega_a\in(0,(n+1)\pi)$.
\end{proposition}
\begin{proof}
Consider $g_\omega$ from Proposition~\ref{prop:approxertophaseerrinequ} as defined in~\eqref{eq:defg}. Following Corollary~\ref{cor:gsym}, the constant term $\theta$ in $g_\omega$ is zero for all $\omega\in(0,(n+1)\pi)$. The poles of $r_\omega$ are continuous in $\omega$ as shown in Proposition~\ref{prop:polescont}, and poles have a non-zero real part since the best approximation has minimal degree $n$. Together with compactness of $[-1,1]$ and continuity of the $\arctan$ function, this entails that the function $g_\omega$~\eqref{eq:defg} changes continuously in $\omega$ in the $\infty$-norm over $x\in[-1,1]$.
\end{proof}

\begin{corollary}[to Proposition~\ref{prop:gcont}]\label{cor:phaseerrcont}
The phase error $g_\omega (x) - \omega x$ is continuous in $\omega$ for $\omega\in(0,(n+1)\pi)$. Moreover, this carries over to its maximum, i.e.,
\begin{equation*}
\max_{x\in[-1,1]}|g_\omega (x) - \omega x|,
\end{equation*}
due to compactness of $[-1,1]$.
\end{corollary}
\begin{proposition}\label{prop:eo1max}
Let $n$ be a given degree and $\omega\in(0,(n+1)\pi)$. Then the phase error of the unitary best approximation to $\mathrm{e}^{\mathrm{i}\omega x}$ attains a maximum at the first equioscillation point $\eta_1=-1$. Thus, the equioscillation points $\eta_1<\ldots<\eta_{2n+1}$ of the unitary best approximation satisfy
\begin{equation}\label{eq:defeo2}
g(\eta_j)-\omega \eta_j = (-1)^{j+1} \max_{x\in[-1,1]} |g(x) - \omega x| ,~~~\text{for $j=1,\ldots,2n+2$}.
\end{equation}
\end{proposition}
\begin{proof}
Due to symmetry the phase function $g$~\eqref{eq:defg} of the unitary best approximation is an odd function with $\theta=0$ (see Corollary~\ref{cor:gsym}). Thus, $g$ attains values between $-n\pi$ and $n\pi$. As a consequence, for $\omega \in (n\pi,(n+1)\pi)$ the phase error $g(x)-\omega x$ is strictly positive at the first equioscillation point $\eta_1=-1$. Following Theorem~\ref{thm:bestapprox}, the phase error is non-trivial, i.e., $\max_{x\in[-1,1]} |g(x) - \omega x|>0$ for $\omega\in(0,(n+1)\pi)$, and in particular, $g(\eta_1) - \omega \eta_1 \neq 0$ for $\omega\in(0,(n+1)\pi)$ since the phase error attains its maximum in absolute value at the equioscillation points. 

We summarize that, at $\eta_1$ the phase error is non-zero for $\omega\in(0,(n+1)\pi)$ and strictly positive for $\omega \in (n\pi,(n+1)\pi)$. Due to continuity in $\omega$, the phase error is strictly positive at $\eta_1$ for $\omega\in(0,(n+1)\pi)$, and thus, it attains a maximum at $\eta_1$. As a consequence, the definition of equioscillation points~\eqref{eq:defeo} specifies to~\eqref{eq:defeo2} for the unitary best approximation.
\end{proof}

\begin{proposition}\label{prop:inodescontinuous}
For a given degree $n$ and $\omega \in(0,(n+1)\pi)$ the equioscillation points and interpolation nodes of the best approximation in $\mathcal{U}_n$ depend continuously on $\omega$.
\end{proposition}
\begin{proof}
Let $r_\omega(\mathrm{i} x)\approx \mathrm{e}^{\mathrm{i} \omega x}$ denote the best approximation in $\mathcal{U}_n$ with $r_\omega(\mathrm{i} x)=\mathrm{e}^{\mathrm{i} g_\omega(x)} $ where $g_\omega$ is defined as in Proposition~\ref{prop:approxertophaseerrinequ}.
Following Theorem~\ref{thm:bestapprox}, there exist exactly $2n+2$ equioscillation points $\eta_1,\ldots,\eta_{2n+2}$ where the phase error $g_\omega(x)-\omega x$ attains its extreme value with alternating sign.
Following Corollary~\ref{cor:errattainsmax}, the zeros of the phase error $x_1,\ldots,x_{2n+1}$ and the equioscillation points are interlaced, i.e., $x_j\in(\eta_j,\eta_{j+1})$ for $j=1,\ldots,2n+1$, and these zeros correspond to simple zeros. Since the equioscillation points and zeros of the phase error depend on $\omega$, we may also write $\eta_j=\eta_j(\omega)$ and $x_j=x_j(\omega)$. We further recall that the interpolation nodes of $r_\omega$ correspond to $\mathrm{i}x_1(\omega),\ldots,\mathrm{i}x_{2n+1}(\omega)$, and equivalently to our assertion on interpolation nodes, we proceed to show that the zeros $x_j(\omega)$ depend continuously on $\omega$.

Consider $\omega_a\in(0,(n+1)\pi)$ and $j\in\{1,\ldots,2n+1\}$ fixed. To simplify our notation we let $\gamma_\omega$ denote the phase error, i.e., $\gamma_\omega(x) := g_\omega(x) - \omega x $. Furthermore, we first assume that the phase error for $\omega=\omega_a$ is negative for $x$ left of the zero $x_j(\omega_a)$. Especially, for a sufficiently small $\varepsilon>0$, and $t_1 := x_j(\omega_a)-\varepsilon$ and $t_2 := x_j(\omega_a)+\varepsilon$, we have $\gamma_{\omega_a}(t_1)<0$ and $\gamma_{\omega_a}(t_2)>0$ since $x_j(\omega_a)$ is a simple zeros.
Since the phase error $\gamma_\omega$ depends continuously on $\omega$ (see Corollary~\ref{cor:phaseerrcont}), we find a sufficiently small $\delta\omega>0$ s.t.\ for $\omega_b\in(\omega_a-\delta\omega, \omega_a+\delta\omega)$ the inequalities $\gamma_{\omega_b}(t_1)<0$ and $\gamma_{\omega_b}(t_2)>0$ remain true. The phase error is continuous in $x$, and thus, $\gamma_{\omega_b}(x)$ has a zero in the interval $[t_1,t_2]$. Especially, such $\delta\omega>0$ exists for an arbitrary small $\varepsilon>0$, i.e., for an arbitrary small neighborhood $[t_1,t_2]$ of $x_j(\omega_a)$, which shows that the zeros $x_j(\omega)$ depends continuously $\omega$. Similar arguments hold true for the case that the phase error is positive left of the zero $x_j(\omega_a)$. Thus, the zeros of the phase error, and consequently, the interpolation nodes $\mathrm{i}x_1(\omega),\ldots,\mathrm{i}x_{2n+1}(\omega)$, depend continuously on $\omega$.

We proceed to consider equioscillation points in a similar manner. Since $\eta_1(\omega)=-1$ and $\eta_{2n+2}(\omega)=1$ we only need to consider $\eta_j(\omega)$ for $j\in\{2,\ldots,2n+1\}$. Let $\omega_a\in(0,(n+1)\pi)$ and $j\in\{2,\ldots,2n+1\}$ be fixed. We first consider indices $j$ s.t.\ the phase error attains a maximum at $\eta_j(\omega_a)$, namely, $\alpha := \gamma_{\omega_a}(\eta_j(\omega_a)) > 0$.
For a sufficiently small $\beta>0$ and $\varepsilon>0$, the points $t_1:=\eta_j(\omega_a)-\varepsilon$ and $t_2:=\eta_j(\omega_a)+\varepsilon$ satisfy $\gamma_{\omega_a}(t_\ell)< \alpha-2\beta$ for $\ell=1,2$ and the phase error $ \gamma_{\omega_a}$ is strictly positive on the interval $[t_1,t_2]$. Since the phase error is continuous in $\omega$ (see Corollary~\ref{cor:phaseerrcont}), we find a sufficiently small $\delta\omega>0$ s.t.\ for $\omega_b\in(\omega_a-\delta\omega,\omega_a+\delta\omega)$ the inequalities $\gamma_{\omega_b}(\eta_j(\omega_a))>\alpha-\beta$ and $\gamma_{\omega_b}(t_\ell)<\alpha-\beta$ for $\ell=1,2$ hold true, and in addition, the phase error remains strictly positive on the interval $[t_1,t_2]$.
Thus, for $\omega_b\in(\omega_a-\delta\omega,\omega_a+\delta\omega)$ the phase error $\gamma_{\omega_b}$ has at least one maximum in the interval $[t_1,t_2]$. Since the phase error does not change its sign in $[t_1,t_2]$, this interval remains enclosed by the zeros $x_{j-1}(\omega_b)$ and $x_j(\omega_b)$. The phase error has exactly one extreme value, i.e., the equioscillation point $\eta_j(\omega_b)$, in between $x_{j-1}(\omega_b)$ and $x_j(\omega_b)$, which shows that the maximum in $[t_1,t_2]$ corresponds to $\eta_j(\omega_b)$ for $\omega_b\in(\omega_a-\delta\omega,\omega_a+\delta\omega)$. For an arbitrary small $\varepsilon>0$, i.e., for an arbitrary small interval $[t_1,t_2]$ around $x_j(\omega_a)$, we find $\beta,\delta\omega>0$ s.t.\ this argument holds true. Thus, the equioscillation point $\eta_j(\omega_a)$ changes continuously with $\omega$. Similar arguments apply for the equioscillation points $\eta_j(\omega)$ where the phase error attains a minimum, which completes our proof.
\end{proof}

\section{Error behavior and convergence}\label{sec:convergence}
In the present section we provide an upper bound and an asymptotic expansion for the approximation error of the unitary best approximation. These results are based on an upper error bound for the diagonal Pad\'e approximation and an asymptotic error expansion for the rational interpolation at Chebyshev nodes.

\subsection{Error bound}\label{subsec:errbound}
We first remark some results for the diagonal Pad\'e approximation, continuing from Subsection~\ref{subsec:intropade}.
The $(n,n)$-Pad\'e approximation~\eqref{eq:defPade} satisfies, cf.~\cite[eq.~(10.24)]{Hi08},
\begin{equation}\label{eq:leadingordererrorpade}
\mathrm{e}^{z} - \widehat{r}(z) = (-1)^n \frac{(n!)^2}{(2n)!(2n+1)!} z^{2n+1} +\mathcal{O}(|z|^{2n+2}),~~~~|z|\to 0.
\end{equation}
In the present work we apply the $(n,n)$-Pad\'e approximation $\widehat{r}(z)\approx\mathrm{e}^z$ as $\widehat{r}(\mathrm{i} \omega x)\approx\mathrm{e}^{\mathrm{i} \omega x}$ where $x\in[-1,1]$.
For the respective error we recall the notation from~\eqref{eq:errnormnotation}, i.e.,
\begin{equation*}
\| \widehat{r}(\omega\cdot) - \exp(\omega\cdot)\| = \max_{x\in[-1,1]} |\widehat{r}(\mathrm{i} \omega x) - \mathrm{e}^{\mathrm{i} \omega x}|.
\end{equation*}

\begin{proposition}[A direct consequence of Theorem~5 in~\cite{LC92}]
\label{prop:padeerrorbound}
The $(n,n)$-Pad\'e approximation to $\textrm{e}^{z}$ satisfies
\begin{equation}\label{eq:errboundpade}
\| \widehat{r}(\omega\cdot) - \exp(\omega\cdot)\|
\leq \frac{(n!)^2 \omega^{2n+1}}{(2n)!(2n+1)!}.
\end{equation}
\end{proposition}
\begin{proof}
This error bound directly follows from the point-wise error bound in~\cite[Theorem~5]{LC92}, i.e.,
\begin{equation*}
|\widehat{r}(\mathrm{i} \omega x) - \mathrm{e}^{\mathrm{i} \omega x}|
\leq \frac{(n!)^2 (\omega|x|)^{2n+1}}{(2n)!(2n+1)!},~~~x\in\mathbb{R}.
\end{equation*}
\end{proof}
Comparing with~\eqref{eq:leadingordererrorpade} for $z=\mathrm{i}\omega x$ and $x\in[-1,1]$, we remark that the error bound~\eqref{eq:errboundpade} is asymptotically correct for $\omega\to 0^+$. Thus, this error bound is tight for sufficiently small frequencies $\omega>0$, and potentially, even for frequencies outside of an asymptotic regime.

In the following proposition we show that the error bound~\eqref{eq:errboundpade} for the $(n,n)$-Pad\'e approximation also yields an error bound for the unitary best approximation. 
\begin{proposition}[Error bound]\label{prop:errorbound}
The error of the unitary best approximation for $\omega>0$ is bounded by
\begin{equation}\label{prop:errorboundconverg}
\min_{u\in\mathcal{U}_n}\| u - \exp(\omega\cdot)\| \leq \frac{(n!)^2 \omega^{2n+1}}{(2n)!(2n+1)!}.
\end{equation}
\end{proposition}
\begin{proof}
The $(n,n)$-Pad\'e approximation $\widehat{r}$ is a unitary rational approximation and this carries over to the re-scaled Pad\'e approximation, i.e.,\ $\widehat{r}(\omega\cdot)\in\mathcal{U}_n$. Especially,
\begin{equation*}
\min_{u\in\mathcal{U}_n} \|u-\exp(\omega\cdot)\| \leq \| \widehat{r}(\omega\cdot) - \exp(\omega \cdot)\|.
\end{equation*}
Together with~\eqref{eq:errboundpade} this shows
the error bound~\eqref{prop:errorboundconverg}.
\end{proof}
We remark that with Stirling's approximation the upper bound in~\eqref{prop:errorboundconverg} approximately corresponds to
\begin{equation}\label{eq:errupperboundapprox}
\frac{(n!)^2 \omega^{2n+1}}{(2n)!(2n+1)!}
\approx \frac{\sqrt{2n}}{\sqrt{2n+1}} \left(\frac{\mathrm{e}\, \omega}{2(2n+1)}\right)^{2n+1}.
\end{equation}
This formula indicates that super-linear convergence occurs  for $\mathrm{e}\, \omega/(2(2n+1))<1$, i.e.,
\begin{equation}\label{eq:superlinearconvwhen}
\omega \leq  4\mathrm{e}^{-1} (n+1/2) \approx 1.47 (n+1/2).
\end{equation}
In particular, this bound accurately represents the error bound~\eqref{prop:errorboundconverg} for larger $n$, when Stirling's approximation is accurate and $\sqrt{2n}/\sqrt{2n+1}\approx 1$. However, since the error bound~\eqref{prop:errorboundconverg} is not necessarily tight, super-linear convergence can already be expected for larger $\omega$ (or smaller $n$).

The error bound~\eqref{prop:errorboundconverg} implies asymptotic convergence for $\omega\to 0$, and super-linear convergence in the degree $n$. While this error bound is certainly relevant from a theoretical point of view, it's not necessarily practical. In the following subsection, we study the asymptotic error for $\omega\to0^+$ which also leads to a more practical error estimate.

\subsection{Asymptotic error and error estimate}\label{subsec:asymerr}

We proceed to derive an asymptotic expansion of the approximation error of the unitary best approximation to $\textrm{e}^{\textrm{i} \omega x}$ for $\omega\to0^+$. This result is closely related a similar expansion for the rational interpolant to $\textrm{e}^{\omega z}$ at Chebyshev nodes on the imaginary axis. Let $T_{2n+1}$ denote the Chebyshev polynomial of degree $2n+1$. Provided $\tau_1,\ldots,\tau_{2n+1}$ refer to the $2n+1$ Chebyshev nodes as in~\eqref{eq:ratinterpolateCheb}, we have the representation
\begin{equation*}
T_{2n+1}(x) = 2^{2n} \prod_{j=1}^{2n+1} (x-\tau_j).
\end{equation*}
Certainly, the monic Chebyshev polynomial is $2^{-2n}T_{2n+1}(x)$. Following classical results, for instance \cite[Corollary~8.1]{SM03}, the monic Chebyshev polynomial uniquely minimizes the $\infty$-norm on $[-1,1]$ in the set of monic polynomials, i.e., for any sequence $x_1,\ldots,x_{2n+1}\in[-1,1]$ distinct to the sequence of Chebyshev nodes,
\begin{equation}\label{eq:Chebmin}
\max_{x\in[-1,1]}\left| \prod_{j=1}^{2n+1} (x-\tau_j)\right| = 2^{-2n} < \max_{x\in[-1,1]}\left| \prod_{j=1}^{2n+1} (x-x_j)\right|.
\end{equation}

\begin{proposition}\label{prop:asymerrcheb}
Let $\mathring{r}_\omega$ denote the rational interpolant to $\mathrm{e}^{\omega z}$ at Chebyshev nodes on the imaginary axis, i.e.,~\eqref{eq:ratinterpolateCheb}. Then
\begin{equation}\label{eq:asymerrratintCheb}
\|\mathring{r}_\omega - \exp(\omega\cdot)\|
= \frac{2(n!)^2 }{(2n)!(2n+1)!}\left(\frac{\omega}{2} \right)^{2n+1}+\mathcal{O}(\omega^{2n+2}),~~~\omega\to 0^+.
\end{equation}
\end{proposition}
\begin{proof}
This result is based on the asymptotic error expansion for rational interpolants in Proposition~\ref{prop:ratintasymerr}. Substituting $\tau_j$ for the nodes $x_j$ therein, more precisely, $2^{-2n}$ for the product $|\prod(x-\tau_j)|$ as in~\eqref{eq:Chebmin}, we arrive at~\eqref{eq:asymerrratintCheb}.
\end{proof}

In the limit $\omega\to0^+$, the unitary best approximation attains the same asymptotic error as the rational interpolant at Chebyshev nodes.

\begin{proposition}[Asymptotic error] \label{prop:asymerrorbest}
The asymptotic error of the unitary best approximation corresponds to
\begin{equation}\label{eq:asymerrbest}
\min_{u\in\mathcal{U}_n}\| u - \exp(\omega\cdot)\| = \frac{2(n!)^2}{(2n)!(2n+1)!} \left( \frac{\omega }{2} \right)^{2n+1} + \mathcal{O}(\omega^{2n+2}),~~~\omega\to 0^+.
\end{equation}
\end{proposition}
\begin{proof}
The rational interpolant at Chebyshev nodes is unitary, i.e., $\mathring{r}_\omega\in\mathcal{U}_n$, and thus,
\begin{equation*}
\min_{u\in\mathcal{U}_n}\| u - \exp(\omega\cdot)\| \leq \|\mathring{r}_\omega - \exp(\omega\cdot)\|.
\end{equation*}
Thus, the asymptotic error~\eqref{eq:asymerrratintCheb} of $\mathring{r}_\omega$ implies
\begin{equation}\label{eq:asymerrbesteq}
\min_{u\in\mathcal{U}_n}\| u - \exp(\omega\cdot)\|
\leq \frac{2(n!)^2}{(2n)!(2n+1)!} \left( \frac{\omega }{2} \right)^{2n+1} + \mathcal{O}(\omega^{2n+2}),~~~\omega\to 0^+.
\end{equation}
Moreover, since $r_\omega$ interpolates $\mathrm{e}^{\omega z}$ at $2n+1$ nodes on the imaginary axis for $\omega\in(0,(n+1)\pi)$, see Corollary~\ref{cor:errattainsmax}, its error satisfies the lower bound from Proposition~\ref{prop:ratintasymlowererrbound}, i.e.,
\begin{equation*}
\|r_\omega-\exp(\omega\cdot)\|
\geq \frac{2(n!)^2}{(2n)!(2n+1)!} \left( \frac{\omega }{2} \right)^{2n+1} + \mathcal{O}(\omega^{2n+2}),~~~\omega\to 0^+.
\end{equation*}
Together with the upper bound~\eqref{eq:asymerrbesteq}, this shows~\eqref{eq:asymerrbest}.
\end{proof}

Based on numerical experiments we also suggest using the leading-order term of the asymptotic expansion~\eqref{eq:asymerrbest} as an error estimate in practice, i.e.,
\begin{equation}\label{eq:asymerrest}
\min_{u\in\mathcal{U}_n}\| u - \exp(\omega \cdot)\|\approx \frac{2(n!)^2}{(2n)!(2n+1)!} \left( \frac{\omega }{2} \right)^{2n+1}.
\end{equation}
Numerical tests provided in the following subsection indicate that this error estimate is practical even for $\omega$ outside of an asymptotic regime.

\begin{remark}\label{remark:bestapproxvsPade}
The asymptotic error~\eqref{eq:asymerrbest} of the unitary best approximation is by a factor $(1/2)^{2n}$ smaller than the asymptotic error of the Pad\'e approximation~\eqref{eq:errboundpade}. Furthermore, considering~\eqref{eq:errboundpade} and~\eqref{eq:asymerrbest} the error of the unitary best approximation for a given $\omega$ can be compared to the error of the Pad\'e approximation for $\omega/2$.
\end{remark}

Using Stirling's approximation to simplify the error estimate~\eqref{eq:asymerrest}, along similar lines to~\eqref{eq:errupperboundapprox}, we also observe
\begin{equation}\label{eq:asymerrestStirlings}
\min_{u\in\mathcal{U}_n}\| u - \exp(\omega \cdot)\|\approx \frac{2\sqrt{2n}}{\sqrt{2n+1}} \left(\frac{\mathrm{e}\, \omega}{4(2n+1)}\right)^{2n+1}.
\end{equation}
Similar to~\eqref{eq:superlinearconvwhen}, based on the error estimate~\eqref{eq:asymerrest} the estimate~\eqref{eq:asymerrestStirlings} indicates that super-linear convergence sets in for
\begin{equation}\label{eq:superlinearconvwhenest}
\omega \leq  8\mathrm{e}^{-1} (n+1/2) \approx 2.94 (n+1/2).
\end{equation}

\subsection{Numerical illustrations}\label{subsec:plots}

In Fig.~\ref{fig:convergence} we plot the absolute value of the approximation error for various approximations over $n$. Different subplots show results for different values of $\omega$. In Fig.~\ref{fig:erroroverw} we plot the absolute value of approximation errors over $\omega$ for different values of $n$.
We observe that~\eqref{eq:errboundpade} (and respectively~\eqref{prop:errorboundconverg}) indeed yields an upper error bound for the Pad\'e approximation and the unitary best approximation.
Moreover, the asymptotic error~\eqref{eq:asymerrbest} of the unitary best approximation is verified in Fig.~\ref{fig:erroroverw}. For $n=4$ and $n=8$ we also observe that the approximation error of the unitary best approximation and rational interpolation at Chebyshev nodes have the same asymptotic behavior for $\omega\to0^+$, as expected from~\eqref{eq:asymerrratintCheb} and~\eqref{eq:asymerrbest}. For larger $n$ the approximation error of the Chebyshev approximant is already beyond machine precision before $\omega$ reaches an asymptotic regime. For $\omega$ outside of an asymptotic regime the unitary best approximation clearly outperforms rational interpolation at Chebyshev nodes in terms of accuracy. Moreover, the unitary best approximation performs better than the Pad\'e approximation. Especially, the error of these two approximations seems to verify Remark~\ref{remark:bestapproxvsPade}.
Furthermore, Fig.~\ref{fig:convergence} also shows the approximation error of a rational approximation based on polynomial Chebyshev approximations (Subsection~\ref{subsec:ratapproxfrompolcheb}), and the polynomial Chebyshev approximation $p(\mathrm{i} x)\approx\mathrm{e}^{\mathrm{i} \omega x}$ of degree $n$ (for instance~\cite[Subsection~III.2.1]{Lu08}). However, these approximations show to be the least accurate among those compared. 

For the unitary best approximation, the error estimate~\eqref{eq:asymerrest}, which corresponds to the leading-order term of the asymptotic expansion~\eqref{eq:asymerrbest}, shows to be practical and reliable in these examples, in particular, it seems to constitute an upper error bound.

\begin{figure}
\centering
\includegraphics{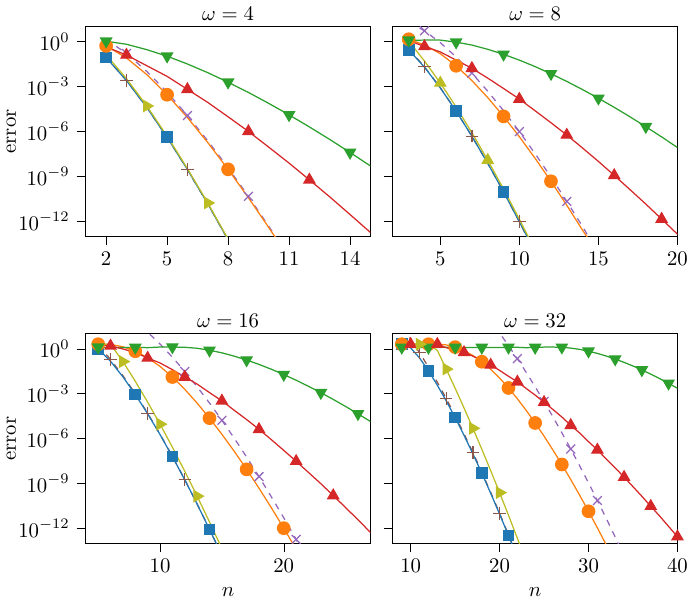}
\caption{This figure shows approximation errors and error estimates over $n$ for different rational approximations $r(\mathrm{i} x) \approx \mathrm{e}^{\mathrm{i} \omega x}$. The four different plots show results for different values of $\omega$, namely, $\omega = 4,8,16,32$ as marked in the plots.
Each of the plots shows the error of the unitary best approximation ($\square$), the error estimate~\eqref{eq:asymerrest} which corresponds to the leading-order term of the asymptotic error expansion~\eqref{eq:asymerrbest} (dashed, $+$), the error of the Pad\'e approximation ($\circ$), and the error bound~\eqref{prop:errorboundconverg} (dashed, $\times$). Furthermore, these plots show the error of the polynomial Chebyshev approximation~($\nabla$) of degree $n$ to $\mathrm{e}^{\mathrm{i} \omega x}$, the error of $\check{r}$~($\triangle$) from~\eqref{eq:rfrompcheb}, and the error of the rational interpolant at Chebyshev nodes $\mathring{r}$ ($\triangleright$) from~\eqref{eq:ratinterpolateCheb}.}\label{fig:convergence}
\end{figure}

\begin{figure}
\centering
\includegraphics{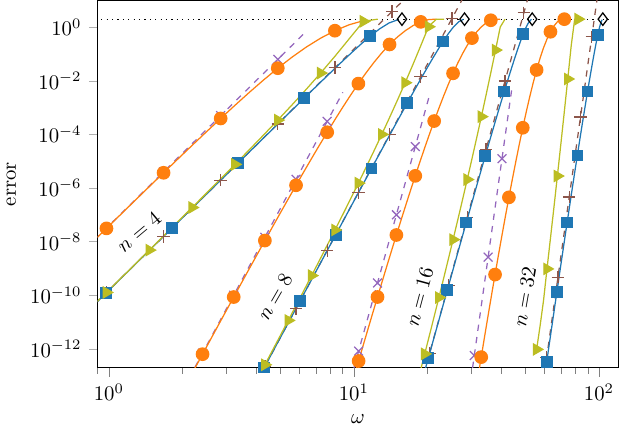}
\caption{This figure shows the approximation error for various approximations and different degrees $n$ over $\omega$. Namely, $n=4,8,16,32$ as marked in the plot. For each choice of $n$ this plot shows the error of the unitary best approximation ($\square$), the error estimate~\eqref{eq:asymerrest} which corresponds to the leading-order term of the asymptotic error expansion~\eqref{eq:asymerrbest} (dashed, $+$), the error of the Pad\'e approximation ($\circ$) and the error bound~\eqref{prop:errorboundconverg} (dashed, $\times$), and the error of the rational interpolant at Chebyshev nodes $\mathring{r}$ ($\triangleright$) from~\eqref{eq:ratinterpolateCheb}. The error estimate and bound seem to be asymptotically correct for $\omega\to0^+$. The error of the unitary best approximation is strictly smaller than two for $\omega< (n+1)\pi$, whereas $(n+1)\pi$ is marked by ($\diamond$) symbols for each $n$.
}\label{fig:erroroverw}
\end{figure}

\section{Further asymptotic considerations}\label{sec:asymprop}

In the present section we show results for the poles, interpolation nodes and equioscillation points in the limits $\omega\to0^+$ and $\omega\to(n+1)\pi^-$. These results are illustrated by numerical experiments at the end of the present section.

\subsection{The limit \texorpdfstring{$\omega\to 0^+$}{w to 0+}}\label{subsec:limitwtozero}

The unitary best approximation is uniquely characterized by an equioscillating phase error for $\omega\in(0,(n+1)\pi)$. In the following proposition, we show that it can be continuously extended to $\omega=0$.
\begin{proposition}\label{prop:rwto1omega0}
The unitary best approximation $r_\omega(\mathrm{i} x)\approx \mathrm{e}^{\mathrm{i} \omega x}$ in $\mathcal{U}_n$, which has minimal degree $n$ for $\omega\in(0,(n+1)\pi)$,
continuously extends to $\omega=0$ with $r_0\equiv 1$, i.e.,
\begin{equation}\label{eq:rwtopade}
\|r_\omega - 1\|\to 0,~~~\omega\to 0^+.
\end{equation}
For $\omega=0$ the approximation $r_\omega\equiv1$ is exact but degenerate.
\end{proposition}
\begin{proof}
Since $r\equiv1\in\mathcal{U}_n$ for any degree $n$, we have
\begin{equation*}
\| r_\omega -\exp(\omega\cdot)\| \leq \| 1 - \exp(\omega\cdot) \| = \mathcal{O}(\omega),~~~\omega\to 0^+.
\end{equation*}
Thus, using triangular inequality we observe
\begin{equation*}
\|r_\omega - 1\| \leq \| r_\omega -\exp(\omega\cdot)\| + \| \exp(\omega\cdot) - 1 \| = \mathcal{O}(\omega),~~~\omega\to 0^+,
\end{equation*}
which shows~\eqref{eq:rwtopade}.
\end{proof}

We proceed to show that the poles of the unitary best approximation $r_\omega$ scaled by $\omega$ converge to poles of the Pad\'e approximation in the limit $\omega\to0^+$.

\begin{proposition}\label{prop:poleslimit}
Let $\widehat{s}_1,\ldots,\widehat{s}_n\in\mathbb{C}$ denote the poles of the $(n,n)$-Pad\'{e} approximation to~$\mathrm{e}^{z}$ (cf.\ Subsection~\ref{subsec:intropade}).
Let $s_1,\ldots,s_n\in\mathbb{C}$ denote the poles of $r_\omega\in \mathcal{U}_n$ which refers to the unitary best approximation $r_\omega(\mathrm{i} x)\approx \mathrm{e}^{\mathrm{i} \omega x}$.
In particular, we write~$s_j=s_j(\omega)$.
Then the poles can be numbered s.t.
\begin{equation}\label{eq:convofpoles}
\omega\, s_j(\omega)  \to \widehat{s}_j,~~~\text{for $\omega \to 0^+$},~~~j=1,\ldots,n.
\end{equation}
\end{proposition}
\begin{proof}
We recall that $r_\omega$ interpolates $\mathrm{e}^{\omega z}$ at $2n+1$ nodes $\mathrm{i}x_1(\omega),\ldots,\mathrm{i}x_{2n+1}(\omega)$ with $x_j(\omega)\in[-1,1]$. Using Proposition~\ref{prop:ratintasympoles}, particularly,~\eqref{eq:sinttopade} which holds true for poles of rational interpolants including the case that interpolation nodes in $[-1,1]$ depend on $\omega$, we arrive at~\eqref{eq:convofpoles}.
\end{proof}

We remark that the convergence result of Proposition~\ref{prop:rwto1omega0} holds true for rational interpolants to $\textrm{e}^{\textrm{i} \omega x}$ and $\omega\to0^+$ in general, cf.\ Proposition~\ref{prop:ratintasymdenom}. In a similar manner Proposition~\ref{prop:poleslimit} directly follows from Proposition~\ref{prop:ratintasympoles}, and thus, the same convergence result holds true for the poles of a rational interpolant to $\mathrm{e}^{\mathrm{i} \omega}$ in general. In the following proposition, we show that the interpolation nodes of the unitary best approximation $r_\omega$ converge to Chebyshev nodes in the limit $\omega\to 0^+$, taking into account that $r_\omega$ minimizes the approximation error.
\begin{proposition}\label{prop:inodestoCheb}
Let $\mathrm{i}x_1(\omega),\ldots,\mathrm{i}x_{2n+1}(\omega)$ denote the interpolation nodes of the best approximation to $\mathrm{e}^{\omega z}$ in $\mathcal{U}_n$.
For $\omega\to0^+$ the nodes $x_1(\omega),\ldots,x_{2n+1}(\omega)$ converge to the $2n+1$ Chebyshev nodes, i.e., $\tau_1,\ldots,\tau_{2n+1}$ in~\eqref{eq:ratinterpolateCheb}.
\end{proposition}
\begin{proof}
As remarked in~\eqref{eq:Chebmin}, the $\infty$-norm on $[-1,1]$ of monic polynomials is uniquely minimized by the Chebyshev nodes $\tau_1,\ldots,\tau_{2n+1}$.
We first show that the interpolation nodes $x_1(\omega),\ldots,x_{2n+1}(\omega)$ of $r_\omega$ attain a minimizing sequence for this $\infty$-norm, namely, 
\begin{equation}\label{eq:minpolylikecheb}
\max_{x\in[-1,1]}\prod_{j=1}^{2n+1} |x-x_j(\omega)|
\to  \max_{x\in[-1,1]}\prod_{j=1}^{2n+1} |x-\tau_j|
= 2^{-2n},~~~\omega\to0^+.
\end{equation}
We prove this claim by contradiction, assuming $x_j(\omega)$ above yields no minimizing sequence. This implies that for some $\varepsilon>0$ there exists a sequence $\{\omega_\ell\}_{\ell\in\mathbb{N}}$ with $\omega_\ell\to0^+$, s.t.
\begin{equation}\label{eq:minpolylikechebnot}
\max_{x\in[-1,1]}\prod_{j=1}^{2n+1} |x-x_j(\omega_\ell)|
> 2^{-2n}+\varepsilon,~~~\ell\in\mathbb{N}.
\end{equation}

We proceed with the representation~\eqref{eq:interr} for the asymptotic error of rational interpolants. The representation~\eqref{eq:interr} has a remainder $\mathcal{O}(\omega^{2n+2})$ for $\omega\to0^+$. For a given $\beta>0$ and a sufficiently small $\omega>0$ this remainder term is bounded by $\beta\omega^{2n+1}$. In particular, consider $\widetilde{r}_\omega$ to be the rational interpolant to $\mathrm{e}^{\mathrm{i} \omega x}$ with fixed interpolation nodes $x_1,\ldots,x_{2n+1}$, then for a sufficiently small $\omega>0$,
\begin{equation}\label{eq:interrorremboundeddelta}
\left| \|\widetilde{r}_\omega-\exp(\omega\cdot)\| - 
\max_{x\in[-1,1]}\prod_{j=1}^{2n+1}|x-x_j| \frac{ (n!)^2}{(2n)!(2n+1)!} \omega^{2n+1}\right|
< \beta\omega^{2n+1}.
\end{equation}
Since this holds true for any choice of interpolation nodes in $[-1,1]$ (including the cases of interpolation nodes of higher multiplicity, which makes the set of interpolation nodes $[-1,1]^{2n+1}$ compact), we may define $\omega_\beta$ s.t.~\eqref{eq:interrorremboundeddelta} holds true for $\omega<\omega_\beta$, and for all choices of $x_j\in[-1,1]$ in a uniform sense.

Since~\eqref{eq:interrorremboundeddelta} holds true uniformly in the choice of the nodes for $\omega<\omega_\beta$, it also applies to the unitary best approximation $r_\omega$ which interpolates $\mathrm{e}^{\omega z}$ at nodes depending on $\omega$, i.e., $\mathrm{i} x_j(\omega)$. Especially, for a sufficiently small $\omega>0$ the upper bound~\eqref{eq:interrorremboundeddelta} implies
\begin{equation}\label{eq:interrorremboundedbetabelow}
\|r_\omega-\exp(\omega\cdot)\| > \max_{x\in[-1,1]}\prod_{j=1}^{2n+1}|x-x_j(\omega)| \frac{ (n!)^2}{(2n)!(2n+1)!} \omega^{2n+1} - \beta\omega^{2n+1}.
\end{equation}
Moreover, the bound from~\eqref{eq:interrorremboundeddelta} holds true for the rational interpolant $\mathring{r}_\omega$ at Chebyshev nodes, i.e., $x_j=\tau_j$, in which case we can substitute $2^{-2n}$ for the product term~\eqref{eq:minpolylikecheb} is this bound. In particular, for a sufficiently small $\omega>0$,
\begin{equation}\label{eq:chebrorremboundedbeta}
\|\mathring{r}_\omega-\exp(\omega\cdot)\| < \frac{2^{-2n}(n!)^2}{(2n)!(2n+1)!} \omega^{2n+1}+ \beta\omega^{2n+1}.
\end{equation}
Considering $\varepsilon$ as in~\eqref{eq:minpolylikechebnot} above, and
\begin{equation}\label{eq:minsequencethisbeta}
\beta = \frac{\varepsilon(n!)^2}{2(2n)!(2n+1)!},
\end{equation}
we find an index $\ell_0$ s.t.\  the inequalities~\eqref{eq:interrorremboundedbetabelow}  and~\eqref{eq:chebrorremboundedbeta} hold true for $\omega_\ell$ with $\ell>\ell_0$. Substituting~\eqref{eq:minsequencethisbeta} for $\beta$ in~\eqref{eq:interrorremboundedbetabelow} we observe
\begin{equation*}
\|r_{\omega_\ell}-\exp(\omega_\ell\cdot)\| > \left( \max_{x\in[-1,1]} \prod_{j=1}^{2n+1}|x-x_j(\omega_\ell)| - \frac{\varepsilon}{2}\right) \frac{ (n!)^2}{(2n)!(2n+1)!} \omega_\ell^{2n+1},~~~\ell>\ell_0.
\end{equation*}
Combining this with~\eqref{eq:minpolylikechebnot}, we arrive at
\begin{equation}\label{eq:ip.besterrboundbelow}
\|r_{\omega_\ell}-\exp(\omega_\ell\cdot)\| > \frac{(2^{-2n}+\varepsilon/2) (n!)^2}{(2n)!(2n+1)!} \omega_\ell^{2n+1},~~~\ell>\ell_0.
\end{equation}

On the other hand, substituting~\eqref{eq:minsequencethisbeta} for $\beta$ in~\eqref{eq:chebrorremboundedbeta} shows that for $\omega_\ell$ with $\ell>\ell_0$, the approximation error of the rational interpolant at Chebyshev nodes is bounded by 
\begin{equation}\label{eq:chebrorremboundedbeta2}
\|\mathring{r}_{\omega_\ell}-\exp(\omega_\ell\cdot)\| < \frac{(2^{-2n}+\varepsilon/2)(n!)^2}{(2n)!(2n+1)!} \omega_\ell^{2n+1},~~~\ell>\ell_0.
\end{equation}
The inequalities~\eqref{eq:ip.besterrboundbelow} and~\eqref{eq:chebrorremboundedbeta2} imply that the rational interpolant at Chebyshev nodes $\mathring{r}_\omega$ attains a smaller approximation error than the unitary best approximation $r_\omega$ for some $\omega_\ell$ with $\ell>\ell_0$, where $\omega_\ell\to0^+$. Since the unitary best approximation minimizes the approximation error in $\mathcal{U}_n$, and $\mathring{r}_\omega\in\mathcal{U}_n$, this yields a contradiction. Thus, we conclude that~\eqref{eq:minpolylikecheb} holds true, i.e., the nodes $x_j(\omega)$ correspond to a minimizing sequence for the $\infty$-norm of all monic polynomials and $\omega\to0^+$. We recall that this $\infty$-norm is uniquely minimized by Chebyshev nodes $\tau_1,\ldots\tau_{2n+1}$. 

It remains to show that interpolation nodes $x_j(\omega)$ converge to Chebyshev nodes $\tau_j$. Assume otherwise, i.e., there exists a sequence $\{\omega_k>0\}_{k\in\mathbb{N}}$ with $\omega_k\to0$ s.t.\ $|x_\iota(\omega_k)-\tau_\iota|>\widetilde{\varepsilon}$ for some $\widetilde{\varepsilon}>0$, $k\in\mathbb{N}$ and at least one index $\iota\in\{1,\ldots,2n+1\}$. Due to compactness arguments, we find a sub-sequence of $\{\omega_{k_\ell}>0\}_{\ell \in \mathbb{N}}$ and nodes $\widetilde{x}_1,\ldots,\widetilde{x}_{2n+1}\in[-1,1]$ s.t.\ $x_j(\omega_{k_\ell})\to \widetilde{x}_j$ for $\ell\to\infty$, and since $x_j(\omega)$ is a minimizing sequence for the $\infty$-norm~\eqref{eq:minpolylikecheb} for $\omega\to0^+$, the points $\widetilde{x}_j$ also minimize the $\infty$-norm. Due to the assumption $|x_\iota(\omega_k)-\tau_\iota|>\widetilde{\varepsilon}$ above, the set of nodes $\widetilde{x}_1,\ldots,\widetilde{x}_{2n+1}$ is distinct to the set of Chebyshev nodes $\tau_1,\ldots,\tau_{2n+1}$, which yields a contradiction since $\tau_j$ uniquely minimize this norm. We conclude that the interpolation nodes $x_j(\omega)$ converge to the Chebyshev nodes $\tau_j$ for $\omega\to0^+$ and up to ordering.
\end{proof}

\subsection{The limit \texorpdfstring{$\omega\to (n+1)\pi^-$}{w to (n+1)pi-}}\label{subsec:limitwnpipi}

In the present section we consider poles, interpolation nodes and equioscillation points of the unitary best approximation to $\mathrm{e}^{\mathrm{i} \omega x}$ in the limit $\omega\to (n+1)\pi^-$. 

\begin{proposition}\label{prop:polesconvwtonp1pi}
Let $n$ denote a fixed degree and let $r_\omega\in\mathcal{U}_n$ denote the best approximation to $\mathrm{e}^{\mathrm{i} \omega x}$ for $\omega\in(0,(n+1)\pi)$.
Then the poles $s_1(\omega),\ldots,s_n(\omega)\in\mathbb{C}$ of $r_\omega$ can be numbered s.t.
\begin{equation}\label{eq:poleslimitwtonp1pi}
s_j(\omega) \to \mathrm{i} \left(-1 +\frac{2j}{n+1}\right),~~~\text{for $j=1,\ldots,n$ and $\omega\to (n+1)\pi^-$.}
\end{equation}
In particular, $\operatorname{Re} s_j(\omega) \to 0^+$ in this limit.
\end{proposition}
\begin{proof}
The proof of this proposition is provided in Appendix~\ref{sec:proofswtonp1pi}.
\end{proof}

\begin{proposition}\label{prop:nodesconvwtonp1pi}
Let $n$ denote a fixed degree and let $r_\omega\in\mathcal{U}_n$ denote the best approximation to $\mathrm{e}^{\mathrm{i} \omega x}$ for $\omega\in(0,(n+1)\pi)$.
Let $\mathrm{i} x_1(\omega),\ldots,\mathrm{i} x_{2n+1}(\omega)$ denote the interpolation nodes of $r_\omega$ in ascending order, then
\begin{equation*}
x_k(\omega) \to -1 +\frac{k}{n+1},~~~k=1,\ldots,2n+1,~~~\omega\to(n+1)\pi^-.
\end{equation*}
Moreover, let $\eta_1(\omega),\ldots,\eta_{2n+2}(\omega)$ denote the equioscillation points of $r_\omega$ in ascending order. The left-most and right-most of these points are fixed, namely, $\eta_1(\omega)=-1$ and $\eta_{2n+2}(\omega)=1$, and the remaining equioscillation points satisfy
\begin{equation*}
\begin{aligned}
&\eta_{2j}(\omega),\eta_{2j+1}(\omega) \to -1+\frac{2j}{n+1},~~~\text{for $j=1,\ldots,n$, and $\omega\to(n+1)\pi^-$},\\
&\text{where} ~~ \eta_{2j}(\omega)< x_{2j}(\omega) < \eta_{2j+1}(\omega),~~~\text{for $\omega\in(0,(n+1)\pi)$}.
\end{aligned}
\end{equation*}
\end{proposition}
\begin{proof}
The proof of this proposition is provided in Appendix~\ref{sec:proofswtonp1pi}.
\end{proof}

We remark that it is already part of the statement of  Theorem~\ref{thm:bestapprox} that two of the equioscillation points are fixed at the boundary, i.e., $\eta_1=-1$ and $\eta_{2n+2}=1$.

\subsection{Numerical illustrations}\label{subsec:plotslimit}

In figures~\ref{fig:interpolationnodes} and~\ref{fig:polesconv} we verify propositions of the present section considering poles, interpolation nodes and equioscillation points of the unitary best approximation $r_\omega(\mathrm{i} x)\approx\mathrm{e}^{\mathrm{i} \omega x}$ in the limits $\omega\to0^+$ and $\omega\to(n+1)\pi^-$. In particular, for $r_\omega\in\mathcal{U}_n$ with $n=4$, and respectively, $(n+1)\pi\approx 15.71$.

We recall the notation $\mathrm{i}x_1,\ldots,\mathrm{i}x_{2n+1}$ for the interpolation nodes and $\eta_1,\ldots,\eta_{2n+2}$ for the equioscillation points of the unitary best approximation where the nodes $x_1,\ldots,x_{2n+1}$ refer to the zeros of the respective phase error. The $2n+1$ zeros $x_j$ and $2n+2$ equioscillation points $\eta_j$ are mirrored around the origin (Proposition~\ref{prop:inodessym}). In Fig.~\ref{fig:interpolationnodes} we plot the non-negative zeros $x_j$ and equioscillation points $\eta_j$ of $r_\omega$ over $\omega$. While the sets of zeros and equioscillation points are uniquely defined for $\omega\in(0,(n+1)\pi)$ due to Theorem~\ref{thm:bestapprox} and Corollary~\ref{cor:errattainsmax}, we only show results for $\omega \in[0.5,15.5]$ in this figure due to numerical difficulties. Namely, for smaller $\omega$ computing the best approximation requires higher precision arithmetic since the phase error is already below machine precision. On the other hand, for $\omega \to (n+1)\pi^-$, in particular, for $\omega>15.5$ in the present example, numerical difficulties occur since equioscillation points approach the zeros of the phase error.

The results shown in Fig.~\ref{fig:interpolationnodes} indicate that propositions~\ref{prop:inodestoCheb} and~\ref{prop:nodesconvwtonp1pi} hold true. Moreover, this figure illustrates that the equioscillation points and interpolation nodes depend continuously on $\omega$ as shown in Proposition~\ref{prop:inodescontinuous}.
We also refer to Fig.~\ref{fig:ip.gnp1pi} for a plot of the phase function and phase error for the current example, i.e., $n=4$, and $\omega=15.5$. We recall that equioscillation points correspond to extrema of the phase error. Similar to Fig.~\ref{fig:interpolationnodes}, this figure shows that pairs of equioscillation points enclose specific zeros $x_j$ of the phase error, as described in Proposition~\ref{prop:nodesconvwtonp1pi} for the limit $\omega\to(n+1)\pi^-$. For this example, the phase error can be described as a sawtooth function. When considering the limit $\omega\to(n+1)\pi^-$, singularities of the phase error occur around the points $-1+2j/(n+1)$, for $j=1,\ldots,n$, which also explains numerical difficulties when computing the unitary best approximation for $\omega\approx (n+1)\pi$. 

In Fig.~\ref{fig:polesconv} we verify the results of proposition~\ref{prop:poleslimit} and~\ref{prop:polesconvwtonp1pi}. The left plot in this figure shows the poles $s_1(\omega),\ldots,s_n(\omega)$ of $r_\omega$ in the complex plane for $\omega\to(n+1)\pi^-$. The poles converge to the points $\mathrm{i}(-1+2j/(n+1))$ on the imaginary axis, as expected from Proposition~\ref{prop:polesconvwtonp1pi}. Especially, poles converge these points while remaining in the right complex plane. In the right plot we show the re-scaled poles $\omega s_j(\omega)$ for $\omega\to0^+$ to verify Proposition~\ref{prop:poleslimit}, i.e., re-scaled poles $\omega s_j(\omega)$ converge to the poles of the $(n,n)$-Pad\'e approximation.

\begin{figure}
\centering
\includegraphics{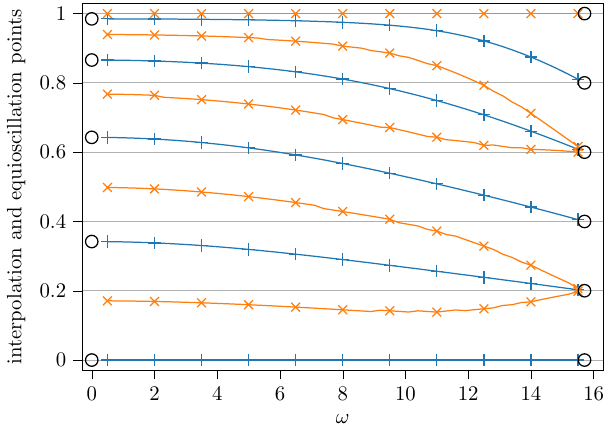}
\caption{
This figure illustrates the non-negative zeros $x_{n+1},\ldots,x_{2n+1}$ ($+$) and equioscillation points $\eta_{n+2},\ldots,\eta_{2n+2}$ ($\times$) of the phase error of the unitary best approximation $r_\omega$ to $\mathrm{e}^{\mathrm{i} \omega x}$ of degree $n=4$ over $\omega$. The zeros of the phase error are directly related to the interpolation nodes $\mathrm{i}x_{n+1},\ldots,\mathrm{i}x_{2n+1}$ of $r_\omega$. We show results for $\omega \in[0.5,15.5]\subset(0,(n+1)\pi)$, avoiding values of $\omega$ too close to zero or $(n+1)\pi\approx 15.71$ due to numerical difficulties. 
At $\omega=0$ we plot the $n+1$ non-negative Chebyshev nodes $\tau_{n+1},\ldots,\tau_{2n+1}$ ($\circ$). At $\omega=0.5$ the nodes $x_{n+1},\ldots,x_{2n+1}$ seem to coincide with the respective set of Chebyshev nodes as expected for sufficiently small $\omega$ due to Proposition~\ref{prop:inodestoCheb}, particularly, in the limit $\omega\to0^+$. At $\omega=(n+1)\pi \approx 15.71$ the symbols ($\circ$) mark the points $-1+j/(n+1)$ for $j=n+1,\ldots,2n+2$. We observe that for $\omega\to (n+1)\pi^-$, the nodes $x_j$ approach the points $-1+j/(n+1)$ for $j=1,\ldots,2n+1$. Moreover, while the equioscillation point $\eta_{2n+2}=1$ is fixed, the other equioscillation points seem to enclose nodes $x_j$ which approach $-1+2j/(n+1)$ for $j=1,\ldots,n$ and $\omega\to (n+1)\pi^-$. Observations for $\omega$ approaching $(n+1)\pi$ comply with Proposition~\ref{prop:nodesconvwtonp1pi}. 
}\label{fig:interpolationnodes}
\end{figure}

\begin{figure}
\centering
\includegraphics{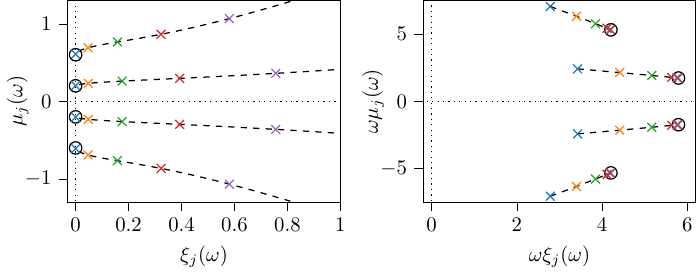}
\caption{
We let $s_j(\omega)=\xi_j(\omega)+\mathrm{i} \mu_j(\omega)$ for $j=1,\ldots,n$ denote the poles of the best approximation $r_\omega(\mathrm{i} x)\approx \mathrm{e}^{\mathrm{i} \omega x}$ where $r_\omega\in\mathcal{U}_n$ and $n=4$.
{\bf left:} The case $\omega \to (n+1)\pi^-\approx 15.71$. The poles $s_j(\omega)$ for $\omega \in  \{15.5, 13.1, 10.7, 8.3, 5.9\}$ are marked by~($\times$) symbols, and in addition, sequences $\{ s_j(\omega):\omega \leq 15.5\}$ are connected by dashed lines. With increasing $\omega$, the poles approach the imaginary axis. In particular, poles converge to points $\mathrm{i}(-1+2j/(n+1))$ for $j=1,\ldots,n$ which are marked by~($\circ$).
{\bf right:} The case $\omega \to 0^+$. The re-scaled poles $\omega s_j(\omega)$ for $\omega \in \{  8, 6, 4,2,1,0.5\} $ are marked by~($\times$) symbols,
and in addition, sequences $\{ \omega s_j(\omega):\omega\in[0.5,8]\}$  are connected by dashed lines. Poles of the $(n,n)$-Pad\'{e} approximation are marked by~($\circ$) symbols. In this plot, the re-scaled poles $\omega s_j(\omega)$ for $\omega=10$ correspond to the left-most poles and re-scaled poles for $\omega = 1$ and $\omega = 0.5$ already overlap with poles of the $(n,n)$-Pad\'{e} approximation.}
\label{fig:polesconv}
\end{figure}

\section{Poles and stability}\label{sec:poles}

In the present section we summarize some previously stated results for the poles of the unitary best approximation, and we show that poles are located in the right complex plane which implies stability~\eqref{eq:rzless1leftplane}.

Similar to symmetry properties discussed in Section~\ref{sec:symmetry}, stability has some relevance for numerical time integration. We remark on this property very briefly, referring the reader to~\cite{HW02} and others for further details. Stability of a time integrator is typically defined via a stability function, which results from applying the time integrator to a scalar differential equation. Especially, for a time integrator based on a scalar rational approximation $r(z)\approx \mathrm{e}^{\omega z}$, $\omega>0$, the stability function corresponds to $r$ itself, and provided $r$ satisfies~\eqref{eq:rzless1leftplane}, such a method is A-stable~\cite[Section~IV.3]{HW02}.

We proceed to recapitulate some results regarding poles of the unitary best approximation $r_\omega(\textrm{i} x)\approx\textrm{e}^{\textrm{i} \omega x}$ from previous sections. Provided $\omega\in(0,(n+1)\pi)$, the unitary best approximation $r_\omega\in\mathcal{U}_n$ has minimal degree $n$ (Theorem~\ref{thm:bestapprox}), and thus, its poles are not located on the imaginary axis. Moreover, the poles are distinct. Since $r_\omega$ is symmetric (Propositions~\ref{prop:sym}), its poles are real or come in complex conjugate pairs (Corollary~\ref{cor:gsym}). Furthermore, poles change continuously with $\omega\in(0,(n+1)\pi)$ (Proposition~\ref{prop:polescont}), in the limit $\omega\to0^+$ the poles scaled by $\omega$ converge to the poles of the Pad\'e approximation (Proposition~\ref{prop:poleslimit}), and in the limit $\omega\to(n+1)\pi^-$ the poles converge to the points $\mathrm{i}(-1+2j/(n+1))$ for $j=1,\ldots,n$ (Proposition~\ref{prop:polesconvwtonp1pi}).

\begin{corollary}[to Proposition~\ref{prop:poleslimit}]\label{cor:polestoinf}
Following~\cite[Section~5.7]{BG96} and others, the poles $\widehat{s}_1,\ldots,\widehat{s}_n$ of the $(n,n)$-Pad\'{e} approximation to $\mathrm{e}^z$ are located in the right complex plane, i.e., $\operatorname{Re} \widehat{s}_j>0$ for $j=1,\ldots,n$. Following Proposition~\ref{prop:poleslimit}, the poles $s_1(\omega),\ldots,s_n(\omega)$ of the best approximation $r_\omega(\mathrm{i} x)\approx \mathrm{e}^{\mathrm{i} \omega x}$ can be numbered s.t.\ $\omega\, s_j(\omega)  \to \widehat{s}_j$ for $\omega \to 0^+$ and $j=1,\ldots,n$. Hence, for a sufficiently small $\omega>0$ this implies $\operatorname{Re} s_j(\omega)>0$. 

Furthermore, for $\omega\to0^+$ the poles are unbounded in the limit, i.e., $|s_j(\omega)|\to \infty$. This result complies with Proposition~\ref{prop:rwto1omega0}, i.e., $r_\omega$ continuously extends to $\omega=0$ with $r_0\equiv 1$.
\end{corollary}

Poles of the unitary best approximation for $n=5$, $n=6$ and various values of $\omega$ are illustrated in Fig.~\ref{fig:polesbestapprox}.

In the following proposition, we show that the poles are located in the right complex plane which implies stability.
\begin{proposition}[Stability]\label{prop:stable}
For $\omega\in(0,(n+1)\pi)$ the poles $s_1,\ldots,s_n$ of the unitary best approximation $r(\mathrm{i} x)\approx \mathrm{e}^{\mathrm{i} \omega x}$, $r\in\mathcal{U}_n$, are located in the right complex plane, i.e., $\operatorname{Re} s_j>0$ for $j=1,\ldots,n$, and we have
\begin{equation}\tag{\ref{eq:rzless1leftplane}}
|r(z)| < 1,~~~~\text{for $z\in\mathbb{C}$ with $\operatorname{Re} z < 0$}. 
\end{equation}
\end{proposition}
\begin{proof}
In the present proof we use the notation $r$ and $r_\omega$ for the best approximation $r_\omega(\mathrm{i} x)\approx \mathrm{e}^{\mathrm{i} \omega x}$, and $s_j$ and $s_j(\omega)$ for its poles in an equivalent manner.
As noted in Corollary~\ref{cor:polestoinf}, for a sufficiently small $\omega>0$ the poles $s_1(\omega),\ldots,s_n(\omega)$ of $r_\omega$ are located in the right complex plane, i.e., $\operatorname{Re} s_j(\omega)>0$.

Theorem~\ref{thm:bestapprox} shows that the unitary best approximation~$r_\omega$ has minimal degree $n$ for $\omega\in(0,(n+1)\pi)$. Especially, this implies that $r_\omega$ is irreducible and has no poles located on the imaginary axis, cf.\ Proposition~\ref{prop:fulldegree}.
Following Proposition~\ref{prop:polescont} the poles $s_j(\omega)$ of $r_\omega$ depend continuously on $\omega\in(0,(n+1)\pi)$. We conclude that the paths of the poles $s_j(\omega)$ do not cross the imaginary axis for $\omega\in(0,(n+1)\pi)$. Since $\operatorname{Re} s_j(\omega) >0$ for sufficiently small~$\omega$, this remains true for all~$\omega\in(0,(n+1)\pi)$.

In accordance to Proposition~\ref{prop:unitaryridentities},
the absolute value of the unitary best approximation $r$ has the form
\begin{equation}\label{eq:ip.absvalrsj}
|r(z)| = \prod_{j=1}^n\left|\frac{z+\overline{s}_j}{z-s_j}\right|.
\end{equation}
Since $\operatorname{Re} s_j>0$, the quotient $(z+\overline{s}_j)/(z-s_j)$ corresponds to a Cayley transform (cf.~\eqref{eq:cayley12}), mapping the left complex plane $\{z \in \mathbb{C} : \operatorname{Re} z < 0\}$ to the disk $\{z \in \mathbb{C}: |z|<1\}$. Thus, the product in~\eqref{eq:ip.absvalrsj} is strictly smaller than one for $z$ in the left complex plane which shows~\eqref{eq:rzless1leftplane}.
\end{proof}

 \begin{figure}
\centering
\includegraphics{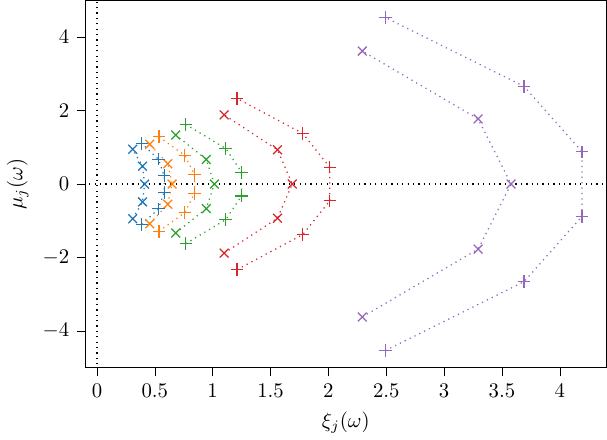}
\caption{Poles $s_j(\omega)=\xi_j(\omega)+\mathrm{i}\mu_j(\omega)$ of the unitary best approximation $r_\omega(\mathrm{i} x)\approx \mathrm{e}^{\mathrm{i}\omega x}$ in the complex plane, where $j=1,\ldots,n$ and $r_\omega\in\mathcal{U}_n$ for $n=5$ ($\times$) and $n=6$ ($+$), and $\omega = 10,8,6,4,2$. Poles for the same $n$ and $\omega$ are connected by a dotted line. The left-most poles correspond to $\omega=10$ and the right-most poles to $\omega=2$.}
\label{fig:polesbestapprox}
\end{figure}

\appendix

\section{Rational interpolation and the limit \texorpdfstring{$\omega\to0^+$}{w to 0+}}\label{sec:appendixA}

In the present section we show some results for the poles and the approximation error of rational interpolants to $\mathrm{e}^{\omega z}$ with interpolation nodes on the imaginary axis, i.e., $\mathrm{i} x_j\in\mathrm{i}\mathbb{R}$, in the limit $\omega\to0^+$. 
This interpolation setting naturally occurs when considering rational approximation to $\mathrm{e}^{\omega z}$ for $z$ on a subset of the imaginary axis. However, to simplify some technicalities in the present section we also consider rational interpolation to $\mathrm{e}^z$ with nodes $\mathrm{i} \omega x_j\in\mathrm{i}\mathbb{R}$ in an equivalent manner. Throughout the present section we assume the nodes $x_1,\ldots,x_{2n+1}\in[-1,1]$ to be given in ascending order, i.e.,
\begin{equation}\label{eq:zjimaginary}
-1 \leq x_1\leq x_2 \leq \ldots \leq x_{2n+1} \leq 1.
\end{equation}

As explained in full detail in Subsection~\ref{subsec:ratint}, rational interpolants are uniquely defined via solutions of an underlying linearized interpolation problem featuring the osculatory case and unattainable nodes. We first state the proof of Proposition~\ref{prop:ratintconfluentunitary}, showing that rational interpolants to $\mathrm{e}^{\omega z}$ with interpolation nodes on the imaginary axis are unitary.

\begin{proof}[\bf Proof of Proposition~\ref{prop:ratintconfluentunitary}]
Let $r=p/q$ be the rational interpolant to $\mathrm{e}^{\omega z}$ in the sense of~\eqref{eq:ratintcondconfluentboth} with interpolation nodes $\mathrm{i} x_1,\ldots,\mathrm{i} x_{2n+1}\in\mathrm{i}\mathbb{R}$. We define 
\begin{equation*}
\chi:=p^\dag p - q^\dag q,
\end{equation*}
which is a polynomial of degree $\leq 2n$. For arguments on the imaginary axis,
\begin{equation}\label{eq:ip.chiidentitiesintix}
\chi(\mathrm{i} x) = \overline{p(\mathrm{i} x)} p(\mathrm{i} x)  - \overline{q(\mathrm{i} x)} q(\mathrm{i} x)  = |p(\mathrm{i} x)|^2 - |q(\mathrm{i} x)|^2,~~~x\in\mathbb{R},
\end{equation}
since $p^\dag(\mathrm{i} x) =\overline{p(\mathrm{i} x)}$ and $q^\dag(\mathrm{i} x) =\overline{q(\mathrm{i} x)}$.
The identity~\eqref{eq:ratintcondconfluentlin2} entails
\begin{equation}\label{eq:ip.chiidentitiesint21}
p(\mathrm{i} x) = \mathrm{e}^{\mathrm{i} \omega x}q(\mathrm{i} x)+h(\mathrm{i} x),~~~\text{where}~~
h(\mathrm{i} x) := \prod_{j=1}^{2n+1}(\mathrm{i} x-\mathrm{i} x_j) v(\mathrm{i} x).
\end{equation}
This shows
\begin{equation}\label{eq:ip.chiidentitiesint1}
\overline{p(\mathrm{i} x)} p(\mathrm{i} x)
= \overline{q(\mathrm{i} x)} q(\mathrm{i} x) + \mathrm{e}^{-\mathrm{i} \omega x}\overline{q(\mathrm{i} x)}h(\mathrm{i} x) +  \mathrm{e}^{\mathrm{i} \omega x}q(\mathrm{i} x)\overline{h(\mathrm{i} x)} +  \overline{h(\mathrm{i} x)} h(\mathrm{i} x),
\end{equation}
where
\begin{equation}\label{eq:ip.chiidentitiesint2}
\overline{h(\mathrm{i} x)} = \prod_{j=1}^{2n+1}(\mathrm{i} x_j - \mathrm{i} x) \overline{v(\mathrm{i} x)}.
\end{equation}
Inserting~\eqref{eq:ip.chiidentitiesint1} in~\eqref{eq:ip.chiidentitiesintix}, we observe
\begin{equation*}
\chi(\mathrm{i} x) =  \mathrm{e}^{-\mathrm{i}\omega x}\overline{q(\mathrm{i} x)}h(\mathrm{i} x) +  \mathrm{e}^{\mathrm{i}\omega x}p(\mathrm{i} x)\overline{h(\mathrm{i} x)} +  \overline{h(\mathrm{i} x)} h(\mathrm{i} x).
\end{equation*}
The functions $h(\mathrm{i} x)$ and $\overline{h(\mathrm{i} x)}$, as in~\eqref{eq:ip.chiidentitiesint21} and~\eqref{eq:ip.chiidentitiesint2}, have zeros at the nodes $\mathrm{i} x_j$, and the multiplicity of these zeros corresponds to the multiplicity of the interpolation nodes. As a consequence, the polynomial $\chi$ has $2n+1$ zeros counting multiplicity which shows $\chi\equiv 0$. Thus, $|p(\mathrm{i} x)|=|q(\mathrm{i} x)|$ for $x\in\mathbb{R}$ and $r=p/q$ is unitary.
\end{proof}

Let $\widehat{r}_\omega$ denote the solution the rational interpolation problem~\eqref{eq:ratintcondconfluentboth} to $\mathrm{e}^z$ with interpolation nodes $\mathrm{i} \omega x_1,\ldots, \mathrm{i} \omega x_{2n+1}$ where $x_j$ satisfies~\eqref{eq:zjimaginary}. Following Proposition~\ref{prop:ratintconfluentunitary} the interpolant $\widehat{r}_\omega$ is unitary. Thus, it is of the form $\widehat{r}_\omega=\widehat{p}^\dag_\omega/\widehat{p}_\omega\in\mathcal{U}_n$ for some polynomial $\widehat{p}_\omega$ of degree $\leq n$, and the underlying linearized interpolation problem~\eqref{eq:ratintcondconfluentlin2} corresponds to 
\begin{equation}\label{eq:linintinproofphat}
\widehat{p}_\omega^\dag(z) - \widehat{p}_\omega(z) \mathrm{e}^z = \prod_{j=1}^{2n+1}(z-\mathrm{i} \omega x_j) v_\omega(z),
\end{equation}
for some analytic function $v_\omega$. We recall that $\widehat{p}_\omega$ is a unique solution to~\eqref{eq:linintinproofphat} up to a common factor of $\widehat{p}_\omega$ and $\widehat{p}^\dag_\omega$. For the case $\omega>0$ and $\widehat{p}_\omega$ as in~\eqref{eq:rintdenomdetTpfull}, we also introduce the notation
\begin{equation}\label{eq:rintdefpw}
p_\omega(z)=\widehat{p}_\omega(\omega z)~~~\text{and}~~r_\omega=p_\omega^\dag/p_\omega,
\end{equation}
i.e.,
\begin{equation}\label{eq:ip.rfromphat}
r_\omega(z) = \frac{p_\omega^\dag(z)}{p_\omega(z)} = \frac{\widehat{p}^\dag_\omega(\omega z)}{\widehat{p}_\omega(\omega z)}.
\end{equation}
Substituting $\omega z$ for $z$ in~\eqref{eq:linintinproofphat}, we observe that $p_\omega$ satisfies 
\begin{equation}\label{eq:linintinproofp}
p_\omega^\dag(z) - p_\omega(z) \mathrm{e}^{\omega z} = \omega^{2n+1}\prod_{j=1}^{2n+1}(z-\mathrm{i} x_j) v_\omega(\omega z),
\end{equation}
and thus $r_\omega$ corresponds to the unique rational interpolant to $\mathrm{e}^{\omega z}$ with interpolation nodes $\mathrm{i} x_1,\ldots,\mathrm{i} x_{2n+1}$.

We may also use the notation $\widehat{r}_\omega$ and $r_\omega$ for the case $\omega =0$. In this case, $\widehat{r}_\omega$ interpolates $\mathrm{e}^z$ at the confluent interpolation node $\mathrm{i} \omega x_1=\ldots=\mathrm{i} \omega x_{2n+1}=0$, and thus, $\widehat{r}_0$ refers to the $(n,n)$-Pad\'e approximant to $\mathrm{e}^z$. Similar to previous sections, we also write $\widehat{r}=\widehat{r}_0$ for the Pad\'e approximant. On the other hand, for $\omega =0$ the unitary function $r_\omega$ interpolates $\mathrm{e}^{\omega z}\equiv 1$ at nodes $\mathrm{i} x_1,\ldots,\mathrm{i} x_{2n+1}$, and thus $r_0 \equiv 1$. The relation~\eqref{eq:ip.rfromphat} remains true since $\widehat{p}^\dag_\omega(\mathrm{i}\omega x)/\widehat{p}_\omega(\mathrm{i}\omega x) = \widehat{r}(0)$ for $\omega = 0$, and since the Pad\'e approximant satisfies $\widehat{r}(0) = 1$.

\subsection{Divided differences}
Similar to polynomial interpolation, divided differences also have some relevance for rational interpolation. In the present subsection we recall classical definitions of divided differences and some properties which are auxiliary for the present work.

Let $z_j=\mathrm{i} \omega x_j$ denote underlying interpolation nodes. Assuming~\eqref{eq:zjimaginary}, we remark that equal interpolation nodes are contiguous, i.e., if $z_j=z_k$ for $j<k$ then $z_j=z_{j+1}=\ldots=z_k$. 
We define the divided differences for the exponential function in accordance to~\cite[Section~B.16]{Hi08} and others, namely,
\begin{subequations}\label{eq:ddrecursive}
\begin{align}\label{eq:dd1nodes}
\exp[z_k] &= \mathrm{e}^{z_k},\\
\label{eq:dd2nodes}
\exp[z_k,z_{k+1}] &=\left\{ 
\begin{array}{ll}
\frac{\exp(z_{k+1})-\exp(z_{k})}{z_{k+1} - z_{k}},~~~&z_k\neq z_{k+1},\\
\mathrm{e}^{z_k},~~~&z_k= z_{k+1},
\end{array}\right.\\
\exp[z_j,\ldots,z_{k+1}] &=\left\{ 
\begin{array}{ll}
\frac{\exp[z_{j+1},\ldots,z_{k+1}]-\exp[z_{j},\ldots,z_{k}]}{z_{k+1} - z_{j}},~~~&z_j\neq z_{k+1},\\
\frac{\exp(z_k)}{(k-j+1)!},~~~&z_j= z_{k+1}.
\end{array}\right.
\end{align}
\end{subequations}

For the special case of an interpolation node at zero of multiplicity $m+1$, e.g., $z_k=\ldots=z_{k+m}=0$ for some index $k$, we remark
\begin{equation}\label{eq:ddconf}
\exp[0^{m+1}]
= \frac{1}{m!},~~~\text{where}~~0^{m+1}:=\underbrace{0,\ldots,0}_{\text{$m+1$ many}}.
\end{equation}

Moreover, divided differences also satisfy the Genocchi--Hermite formula, for instance~\cite[eq.~(B.25)]{Hi08} or~\cite[eq.~(52)]{Bo05},
\begin{equation}\label{eq:ddGenocchiHermite}
\begin{aligned}
&\exp[z_j,\ldots,z_{j+k}]\\
&\quad=\int_0^1 \int_0^{t_1} \int_0^{t_2}\cdots \int_0^{t_{k-1}} 
\exp\left( z_j +\sum_{\ell=1}^k t_{\ell} (z_{j+\ell} - z_{j+\ell -1})  \right)\, \mathrm{d} t_k \cdots \mathrm{d} t_2 \mathrm{d} t_1.
\end{aligned}
\end{equation}

We proceed with some auxiliary results which hold true for divided differences of the exponential function, assuming the underlying interpolation nodes are located the imaginary axis.
The upper bound from~\cite[eq.~(B.28)]{Hi08} implies
\begin{equation}\label{eq:upperboundonddzk}
|\exp[\mathrm{i} \omega x_{k},\mathrm{i} \omega x_{k+1},\ldots,\mathrm{i} \omega x_{k+m}]| \leq \frac{1}{m!}.
\end{equation}

\begin{proposition} For given frequencies $\omega_1,\omega_2\in\mathbb{R}$ the exponential function satisfies
\begin{equation}\label{eq:boundonexpiw}
|\mathrm{e}^{\mathrm{i} \omega_2 x}-\mathrm{e}^{\mathrm{i} \omega_1 x}|
\leq |\omega_2-\omega_1|,~~~x\in[-1,1].
\end{equation}
\end{proposition}
\begin{proof}
Similar to~\eqref{eq:upperboundonddzk}, the divided difference for the nodes $z_1=\mathrm{i}\omega_1 x$ and $z_2=\mathrm{i}\omega_2 x$ satisfies
\begin{equation*}
|\exp[\mathrm{i} \omega_1 x,\mathrm{i} \omega_2 x]| \leq 1,~~~x\in\mathbb{R}.
\end{equation*}
Substituting~\eqref{eq:dd2nodes} for the divided difference, and making use $|x|\leq 1$, we conclude~\eqref{eq:boundonexpiw}.
\end{proof}

\begin{proposition}\label{prop:ddexpdifbound}
The divided differences for nodes $z_j=\mathrm{i} \omega x_j$ as in~\eqref{eq:zjimaginary} satisfy
\begin{equation}\label{eq:ddexpto1k}
\left|\exp[\mathrm{i}\omega x_{j},\mathrm{i}\omega x_{j+1},\ldots,\mathrm{i}\omega x_{j+k}]
- \frac{1}{k!}\right| \leq \frac{\omega}{k!}.
\end{equation}
\end{proposition}
\begin{proof}
Making use of the Genocchi--Hermite formula~\eqref{eq:ddGenocchiHermite}
for the divided difference of a confluent node at zero of multiplicity $k+1$ as in~\eqref{eq:ddconf}, we note
\begin{equation}\label{eq:GenicchiHermite1}
\frac{1}{k!}=\exp[0^{k+1}]=\int_0^1 \int_0^{t_1} \int_0^{t_2}\cdots \int_0^{t_{k-1}} 1\, \mathrm{d} t_k \cdots \mathrm{d} t_2 \mathrm{d} t_1.
\end{equation}
Thus, the Genocchi--Hermite formula~\eqref{eq:ddGenocchiHermite} shows
\begin{align*}
& \exp[\mathrm{i}\omega x_{j},\mathrm{i}\omega x_{j+1},\ldots,\mathrm{i}\omega x_{j+k}]
- \frac{1}{k!}\\
& \quad ~ = \int_0^1 \int_0^{t_1} \int_0^{t_2}\cdots \int_0^{t_{k-1}} 
\left( \exp\left( z_j +\sum_{\ell=1}^k t_{j+\ell} (z_{j+\ell} - z_{j+\ell -1})  \right) - 1 \right) \mathrm{d} t_k \cdots \mathrm{d} t_2 \mathrm{d} t_1\notag
\end{align*}
where $z_j=\mathrm{i} \omega x_j$. Using triangular inequalities, we observe that the absolute value of this difference is bounded by
\begin{equation}\label{eq:GenicchiHermitediff1}
\begin{aligned}
& \left|\exp[\mathrm{i}\omega x_{j},\mathrm{i}\omega x_{j+1},\ldots,\mathrm{i}\omega x_{j+k}]
- \frac{1}{k!}\right|\\
&\qquad\qquad\qquad\quad \leq
\max_{x \in[x_j,x_{j+k}]} |\mathrm{e}^{\mathrm{i}\omega x} - 1|
 \int_0^1 \int_0^{t_1} \int_0^{t_2}\cdots \int_0^{t_{k-1}} 
1\, \mathrm{d} t_k \cdots \mathrm{d} t_2 \mathrm{d} t_1.
\end{aligned}
\end{equation}
The integral therein corresponds to $1/k!$ as in~\eqref{eq:GenicchiHermite1}, and following~\eqref{eq:boundonexpiw}, the maximum therein satisfies
\begin{equation}\label{eq:GenicchiHermitediff2}
\max_{x \in[x_j,x_{j+k}]} |\mathrm{e}^{\mathrm{i}\omega x}-1|
\leq \max_{x \in[-1,1]} |\mathrm{e}^{\mathrm{i}\omega x}-1|
\leq \omega.
\end{equation}
Thus, combining~\eqref{eq:GenicchiHermitediff1} and~\eqref{eq:GenicchiHermitediff2} we conclude~\eqref{eq:ddexpto1k}.
\end{proof}

\subsection{Explicit representations for rational interpolation}

Explicit formulations for the numerator $p$ and denominator $q$ of a rational interpolant $r=p/q$ are provided in~\cite[Theorem~7.1.1]{BG96} together with an expression of the remainder, i.e., the right-hand side of the linearized problem~\eqref{eq:ratintcondconfluentlin2}. We may apply these results to provide explicit formulations for $\widehat{p}_\omega$ and $v_\omega$ in~\eqref{eq:linintinproofphat}. To this end, we first declare the notation
\begin{equation}\label{eq:ddshortw}
\exp_{k,j}^\omega :=
\exp[\mathrm{i} \omega x_k,\mathrm{i} \omega x_{k+1},
\ldots,\mathrm{i} \omega x_j],~~~k\leq j,
\end{equation}
for the divided differences of the underlying nodes, and
\begin{equation}\label{eq:rintdefHw}
H_\omega := 
\begin{pmatrix}
\exp_{n+1,n+2}^\omega&\exp_{n+1,n+3}^\omega&\cdots&\exp_{n+1,2n+1}^\omega\\
\exp_{n,n+2}^\omega&\exp_{n,n+3}^\omega&\cdots&\exp_{n,2n+1}^\omega\\
\vdots&\vdots&&\vdots\\
\exp_{1,n+2}^\omega&\exp_{1,n+3}^\omega&\cdots&\exp_{1,2n+1}^\omega
\end{pmatrix}\in\mathbb{C}^{(n+1)\times n}.
\end{equation}
Moreover, we explicitly define $\widehat{p}_\omega$ as
\begin{subequations}\label{eq:rintdenomdetTpfull}
\begin{align}\label{eq:rintdenomdetT}
&\widehat{p}_\omega(z) = \det\Theta_\omega(z),~~~\text{for}~~~\Theta_\omega(z)=[H_\omega|\theta_\omega(z)]\in\mathbb{C}^{(n+1)\times (n+1)},~~~\text{where}\\
&\label{eq:rintdefthetaw}
\theta_\omega(z) := \left(\prod_{k=1}^{n}( z-\mathrm{i}\omega x_k),\prod_{k=1}^{n-1}( z-\mathrm{i}\omega x_k),\ldots,1\right)^\top\in\mathbb{C}^{n+1}, 
\end{align}
\end{subequations}
and $[H_\omega|\theta_\omega(z)]$ is to be understood as the $(n+1)\times (n+1)$ complex
matrix obtained by concatenating the $(n+1)\times n$ matrix $H_\omega$ and $n+1$-dimensional vector $\theta_\omega(z)$.
Furthermore, we explicitly define $v_\omega$ as
\begin{subequations}\label{eq:ip.linerrgenall}
\begin{align}\label{eq:ip.linerrgen}
&v_\omega(z) = \det\Gamma_\omega(z),~~~\text{where}~~\Gamma_\omega(z)=[H_\omega|\gamma_\omega(z)]\in\mathbb{C}^{(n+1)\times (n+1)}~~\text{with},\\
&\label{eq:ip.rintdefgammaw}
\gamma_\omega(z) := 
\begin{pmatrix}
\exp[\mathrm{i}\omega x_{n+1},\ldots,\mathrm{i}\omega x_{2n+1},z]\\
\exp[\mathrm{i}\omega x_n,\ldots,\mathrm{i}\omega x_{2n+1},z]\\
\vdots\\
\exp[\mathrm{i}\omega x_1,\ldots,\mathrm{i}\omega x_{2n+1},z]
\end{pmatrix}\in\mathbb{C}^{n+1}.
\end{align}
\end{subequations}

\smallskip
Applying the results~\cite[Theorem~7.1.1]{BG96} to the present setting, we observe that $\widehat{r}_\omega=\widehat{p}_\omega^\dag/\widehat{p}_\omega$ with $\widehat{p}_\omega$ as in~\eqref{eq:rintdenomdetTpfull} interpolates $\mathrm{e}^z$ with interpolation nodes $\mathrm{i}\omega x_1,\ldots,\mathrm{i}\omega x_{2n+1}$. Thus, together with $v_\omega$ from~\eqref{eq:ip.linerrgenall}, the polynomial $\widehat{p}_\omega$ solves the linearized interpolation problem~\eqref{eq:linintinproofphat}.
Moreover, $r_\omega=p^\dag_\omega/p_\omega$ with $p_\omega(z) = \widehat{p}_\omega(\omega z)$ corresponds to the rational interpolant to $\mathrm{e}^{\omega z}$ with interpolation nodes $\mathrm{i} x_1,\ldots,\mathrm{i} x_{2n+1}$. Thus, $\widehat{p}_\omega(z) = \widehat{p}_\omega(\omega z)$ with $\widehat{p}_\omega$ as in~\eqref{eq:rintdenomdetTpfull} and $v_\omega$ as in~\eqref{eq:ip.linerrgenall} solve the linearized interpolation problem~\eqref{eq:ip.rfromphat}.

For the case $\omega=0$, i.e., the Pad\'e approximation, we remark that the notation in~\eqref{eq:rintdefHw},~\eqref{eq:rintdefthetaw} and~\eqref{eq:ip.rintdefgammaw} simplifies to
\begin{subequations}\label{eq:H0theta0gamma0}
\begin{align}\label{eq:rintdenomdetTPade}
&H_0 = 
\begin{pmatrix}
1&1/2!&\cdots&1/n!\\
1/2!&1/3!&\cdots&1/(n+1)!\\
\vdots&\vdots&&\vdots\\
1/(n+1)!&1/(n+2)!&\cdots&1/(2n)!
\end{pmatrix},\\
&\theta_0(z) = (z^n,z^{n-1},\ldots,1)^\top,~~~\text{and}\label{eq:theta0entries}\\
&\gamma_0(z) = (
\exp[0^{n+1},z],
\exp[0^{n+2},z],
\ldots,
\exp[0^{2n+1},z]
)^\top.
\end{align}
\end{subequations}
In a similar manner, for the Pad\'e approximation and its remainder  the expressions given by $H_0$, $\theta_0$ and $\gamma_0$ via~\eqref{eq:rintdenomdetTpfull} and~\eqref{eq:ip.linerrgenall} also appear in~\cite[eq.~(2.1)]{BG96}. The denominator $\widehat{p}$ of the Pad\'e approximation is unique up to a constant factor which is typically used to normalize $\widehat{p}(0)=1$. However, no normalizing factor is used in the present section.

\subsection{Convergence of poles for \texorpdfstring{$\omega\to0^+$}{w to 0+}}
We provide a convergence results for the poles of the rational interpolant $r_\omega$ to $\mathrm{e}^{\omega z}$ in Proposition~\ref{prop:ratintasympoles} further below.

\begin{proposition}\label{prop:ratintasymdenom}
Let $\widehat{r}=\widehat{p}^\dag/\widehat{p}$ denote the $(n,n)$-Pad\'e approximation to $\mathrm{e}^z$, and let $\widehat{r}_\omega\in\mathcal{U}_n$ denote the rational interpolant to $\mathrm{e}^z$ with interpolation nodes $\mathrm{i} \omega x_1,\ldots,\mathrm{i}\omega x_{2n+1}\in\mathrm{i}\mathbb{R}$. Then
\begin{enumerate}[label=(\roman*)]
\item\label{item:pzconverge} there exists a representation $\widehat{r}_\omega=\widehat{p}_\omega^\dag/\widehat{p}_\omega$ with
\begin{equation}\label{eq:pinttopade}
\widehat{p}_\omega(z) \to \widehat{p}(z),~~~z\in\mathbb{C},~~~\omega\to0^+,~~~\text{and}
\end{equation}
\item\label{item:rpolesconverge} the zeros and poles of $\widehat{r}_\omega$ converge to the zeros and poles of $\widehat{r}$, respectively, for $\omega\to0^+$.
\end{enumerate}
In particular, for a sufficiently small $\omega>0$ the rational interpolant $\widehat{r}_\omega$ has minimal degree $n$.
\end{proposition}
\begin{proof}
For $\omega\to0^+$ the sequence of interpolation nodes $\mathrm{i}\omega x_1,\ldots,\mathrm{i} \omega x_{2n+1}$ converges to the sequence of $2k+1$ zeros and the Pad\'e approximant $\widehat{r}$ corresponds to the rational interpolant to $\mathrm{e}^z$ for the latter set of interpolation nodes. Since the Pad\'e approximation is non-degenerate, i.e., $\widehat{r}$ has minimal degree $n$,~\cite[Lemma~1]{Gu90} implies that the numerator and denominator of $\widehat{r}_\omega$ converge to the numerator and denominator $\widehat{r}$ coefficient-wise. In particular, this shows~\eqref{eq:pinttopade}.

Moreover, since we have the non-degenerate case the denominator of the Pad\'e approximation $\widehat{p}$ has degree exactly $n$, and the zeros of $\widehat{p}_\omega$ converge to the zeros of $\widehat{p}$ for $\omega\to0^+$ up to ordering, see~\cite[Proposition~5.2.1]{Ar11} for instance. Analogously, the zeros of $\widehat{p}^\dag_\omega$ converge to the zeros of $\widehat{p}^\dag$. Since the zeros and poles of $\widehat{r}_\omega$ correspond to the zeros of $\widehat{p}_\omega^\dag$ and $\widehat{p}_\omega$, respectively, this shows~\ref{item:rpolesconverge}.

Since the $(n,n)$-Pad\'e approximation $\widehat{r}\in\mathcal{U}_n$ has minimal degree $n$, the set of its zeros is distinct to the set of its poles. Due to convergence of zeros and poles of $\widehat{r}_\omega$ as in~\ref{item:rpolesconverge}, we conclude that $\widehat{r}_\omega$ has minimal degree $n$ for a sufficiently small $\omega>0$.
\end{proof}

We recall that the rational interpolant $r_\omega\in\mathcal{U}_n$ to $\mathrm{e}^{\omega z}$ with interpolation nodes $\mathrm{i} x_1,\ldots,\mathrm{i} x_{2n+1}\in\mathrm{i}\mathbb{R}$
satisfies $r_\omega(z) = \widehat{p}_\omega^\dag(\omega z)/\widehat{p}_\omega(\omega z)$ as in~\eqref{eq:ip.rfromphat}.
As a consequence of Proposition~\ref{prop:ratintasymdenom}, the rational interpolant $\widehat{r}_\omega(z)=\widehat{p}_\omega^\dag(z)/\widehat{p}_\omega(z)$ has minimal degree $n$ for a sufficiently small $\omega>0$, and this carries over to $r_\omega$. Thus, in this case the rational interpolant $r_\omega$ has $n$ poles which correspond to non-removable singularities.
\begin{proposition}\label{prop:ratintasympoles}
Let $s_1(\omega),\ldots,s_n(\omega)\in\mathbb{C}$ denote the poles of the rational interpolant $r_\omega\in\mathcal{U}_n$ to $\mathrm{e}^{\omega z}$ with interpolation nodes $\mathrm{i} x_1,\ldots,\mathrm{i} x_{2n+1}\in\mathrm{i}\mathbb{R}$, and let $\widehat{s}_1,\ldots,\widehat{s}_n\in\mathbb{C}$ denote the poles of the $(n,n)$-Pad\'e approximation to $\mathrm{e}^z$, then
\begin{equation}\label{eq:sinttopade}
\omega s_\ell(\omega) \to \widehat{s}_\ell,~~~\ell=1,\ldots,n,~~\omega\to 0^+,
\end{equation}
up to ordering. Moreover, this result remains true if $r_\omega$ refers to a rational interpolant with interpolation nodes $\mathrm{i} x_1(\omega),\ldots,\mathrm{i} x_{2n+1}(\omega)$ where $x_j(\omega)\in[-1,1]$ depends on $\omega$.
\end{proposition}

\begin{proof} For the rational interpolant $r_\omega$ to $\mathrm{e}^z$ we consider $r_\omega = p_\omega^\dag /p_\omega$ with $p_\omega(z)=\widehat{p}_\omega(\omega z)$ as in~\eqref{eq:ip.rfromphat}, where we choose $\widehat{p}_\omega$ as in Proposition~\ref{prop:ratintasymdenom}.\ref{item:pzconverge} for the present proof. Let $\zeta_1(\omega),\ldots,\zeta_n(\omega)$, $\widehat{\zeta}_1(\omega),\ldots,\widehat{\zeta}_n(\omega)$, and $\widehat{\zeta}_1,\ldots,\widehat{\zeta}_n$ refer the zeros of $p_\omega$, $\widehat{p}_\omega$, and $\widehat{p}$, respectively. Since $p_\omega(z)=\widehat{p}_\omega(\omega z)$, the zeros $\zeta_\ell(\omega)$ and $\widehat{\zeta}_\ell(\omega)$ satisfy the relation 
\begin{equation}\label{eq:ip.zerosphattozerosp}
\omega \zeta_\ell(\omega) = \widehat{\zeta}_\ell(\omega),~~~\ell=1,\ldots,n.
\end{equation}
We recall that the zeros of $\widehat{p}_\omega$ and $\widehat{p}$ correspond to the poles of the rational interpolant $\widehat{r}_\omega=\widehat{p}_\omega^\dag/\widehat{p}_\omega$ and the Pad\'e approximant $\widehat{r}$, respectively. Thus, Proposition~\ref{prop:ratintasymdenom}.\ref{item:rpolesconverge} implies
\begin{equation}\label{eq:ip.phatzerosconverge}
\widehat{\zeta}_\ell(\omega) \to \widehat{\zeta}_\ell,~~~\ell=1,\ldots,n,~~\omega\to 0^+,
\end{equation}
up to ordering.
Since the zeros $\zeta_\ell(\omega)$ correspond to the poles of the rational interpolant $r_\omega$, the identity~\eqref{eq:ip.zerosphattozerosp} and the convergence result~\eqref{eq:ip.phatzerosconverge} show~\eqref{eq:sinttopade}.

\smallskip
The convergence result~\eqref{eq:sinttopade} holds true for interpolation nodes $\mathrm{i} x_1,\ldots,\mathrm{i} x_{2n+1}$ where $x_j\in[-1,1]$, i.e., $(x_1,\ldots,x_{2n+1}) \in [-1,1]^{2n+1}$ where $[-1,1]^{2n+1}$ denotes the $2n+1$-dimensional Cartesian product of the interval $[-1,1]$. Thus, the convergence result~\eqref{eq:sinttopade} holds true independently of the choice of underlying interpolation nodes due to compactness arguments. More precisely, for all $\varepsilon>0$ and $\ell=1,\ldots,n$ we have $|\omega s_\ell(\omega) - \widehat{s}_\ell|<\varepsilon$ up to ordering for a sufficiently small $\omega>0$ and independent of the underlying interpolation nodes $\mathrm{i} x_j$ with $x_j\in[-1,1]$. Thus, the statement~\eqref{eq:sinttopade} also holds true if underlying interpolation nodes depend on $\omega$ which concludes our last assertion.
\end{proof}

\subsection{An asymptotic error expansion}
We provide an asymptotic expression and lower bound for the error $\|r_\omega - \exp(\omega \cdot)\|$ for $\omega\to0^+$ in propositions~\ref{prop:ratintasymerr} and~\ref{prop:ratintasymlowererrbound} below. First, we state some auxiliary results.

\begin{proposition}\label{prop:detAmdetB}
The determinants of two matrices $A$ and $B\in\mathbb{C}^{n+1\times n+1}$ satisfy
\begin{equation}\label{eq:detAdetBdif0}
|\det A - \det B|
\leq \sum_{\sigma\in\mathrm{S}_{n+1}}  \sum_{k=1}^{n+1}\left(
\left|A_{\sigma(k)k} - B_{\sigma(k)k}\right|
\prod_{j=1}^{k-1}\left|A_{\sigma(j)j}\right|
\prod_{j=k+1}^{n+1}\left|B_{\sigma(j)j}\right|\right),
\end{equation}
where $\mathrm{S}_{n+1}$ denotes the set of permutations of $\{1,\ldots,n+1\}$. 
\end{proposition}
\begin{proof}
We recall the Leibniz formula for the determinant of a matrix $A \in\mathbb{C}^{(n+1)\times(n+1)}$, namely,
\begin{equation*}
\det A = \sum_{\sigma\in\mathrm{S}_{n+1}} \operatorname{sgn}(\sigma) 
\prod_{j=1}^{n+1} A_{\sigma(j)j}.
\end{equation*}
where $\mathrm{S}_{n+1}$ denotes the set of permutations of $\{1,\ldots,n+1\}$, and $\operatorname{sgn}(\sigma) = \pm1$ denotes the sign of the permutation $\sigma$.
Making use of the Leibniz formula for the determinants of the two matrices $A$ and $B$, we observe
\begin{equation}\label{eq:deviationthetaprod0}
\begin{aligned}
|\det A-\det B|
&=\left|  \sum_{\sigma\in\mathrm{S}_{n+1}} \operatorname{sgn}(\sigma) \left( 
 \prod_{j=1}^{n+1} A_{\sigma(j)j}-
 \prod_{j=1}^{n+1} B_{\sigma(j)j}
\right) \right|\\
&\leq \sum_{\sigma\in\mathrm{S}_{n+1}} \left|
 \prod_{j=1}^{n+1} A_{\sigma(j)j}-
 \prod_{j=1}^{n+1} B_{\sigma(j)j}\right|.
\end{aligned}
\end{equation}
The difference in the last line therein expands to
\begin{align*}
 \prod_{j=1}^{n+1} A_{\sigma(j)j}-
 \prod_{j=1}^{n+1} B_{\sigma(j)j}
 = \sum_{k=1}^{n+1}
&\left(\vphantom{\prod_{j}^{k}}\left(A_{\sigma(k)k} - B_{\sigma(k)k}\right)\right.\\
&\qquad\cdot
\left.\prod_{j=1}^{k-1}A_{\sigma(j)j}
\prod_{j=k+1}^{n+1}B_{\sigma(j)j}
\right),
\end{align*}
and together with~\eqref{eq:deviationthetaprod0}, this shows~\eqref{eq:detAdetBdif0}.
\end{proof}

\begin{proposition}\label{prop:ratintasymThetaGamma}
For $x\in[-1,1]$ and $\omega\leq 1$ we have the upper bounds
\begin{align}
\label{eq:thetaerr}
|\det\Theta_\omega(\mathrm{i} \omega x) - \det\Theta_0(\mathrm{i} \omega x)|
&\leq \omega c_{\theta,n},
~~~\text{and}\\
\label{eq:gammaerr}
|\det\Gamma_\omega(\mathrm{i} \omega x) - \det\Gamma_0(\mathrm{i} \omega x)|
&\leq \omega c_{\gamma,n},
\end{align}
where $c_{\theta,n}$ and $c_{\gamma,n}$ denote the constants
\begin{equation*}
c_{\theta,n} = (n+2^{n+1})(n+1) \prod_{j=1}^{n-1}\frac{1}{j!}~~~\text{and}~~
c_{\gamma,n} = (n+2) \prod_{j=1}^{n}\frac{1}{j!}.
\end{equation*}
\end{proposition}
\begin{proof}
We recall
\begin{equation*}
\Theta_\omega(\mathrm{i}\omega x)=[H_\omega|\theta_\omega(\mathrm{i}\omega x)]~~~\text{and}~~
\Gamma_\omega(\mathrm{i}\omega x)=[H_\omega|\gamma_\omega(\mathrm{i}\omega x)]\in\mathbb{C}^{n+1\times n+1}.
\end{equation*}
from~\eqref{eq:ip.linerrgenall} and~\eqref{eq:ip.linerrgenall} for $\omega\geq 0$, and we also refer to~\eqref{eq:H0theta0gamma0} for the case $\omega=0$. We first remark some results for the matrix entries of $H_\omega$ and $H_0$ which are used to derive the upper bounds~\eqref{eq:thetaerr} and~\eqref{eq:gammaerr} further below.

The matrix $H_0$ has the entries
\begin{equation}\label{eq:Theta0jl}
(H_0)_{\ell j} = \frac{1}{(\ell+j-1)!},~~~\ell=1,\ldots,n+1,~~j=1,\ldots,n.
\end{equation}
For the matrix $H_\omega$ the entries correspond to divided differences of $\mathrm{e}^{z}$ at nodes $\mathrm{i} \omega x_k$, namely, with the notation from~\eqref{eq:ddshortw},
\begin{equation}\label{eq:Thetawjlfirstn}
(H_\omega)_{\ell j}
=\exp_{n-\ell+2,n+j+1}^\omega =\exp[\mathrm{i} \omega x_{n-\ell+2},\mathrm{i} \omega x_{n-\ell+3},\ldots,\mathrm{i} \omega x_{n+j+1}],
\end{equation}
where $\ell=1,\ldots,n+1$ and $j=1,\ldots,n$. Thus, the entry $(H_\omega)_{\ell j}$ corresponds to a divided difference over $\ell+j$ interpolation nodes. Due to~\eqref{eq:upperboundonddzk}, the entries in~\eqref{eq:Thetawjlfirstn} satisfy
\begin{equation}\label{eq:HwljtoH0lj}
|(H_\omega)_{\ell j}|
\leq\frac{1}{(\ell+j-1)!}
= (H_0)_{\ell j}.
\end{equation}
Moreover, making use of Proposition~\ref{prop:ddexpdifbound} we observe
\begin{equation}\label{eq:HwminusH0}
\left|(H_\omega)_{\ell j} - (H_0)_{\ell j}\right|
\leq \omega (H_0)_{\ell j}.
\end{equation}
Let $\sigma\in\mathrm{S}_{n+1}$ denote a permutation similar as in Proposition~\ref{prop:detAmdetB}, then the product over entries $(H_0)_{\sigma(j)j}$ from~\eqref{eq:Theta0jl} satisfies
\begin{equation}\label{eq:defkappaH0}
\prod_{j=1}^{n} (H_0)_{\sigma(j)j} \leq \prod_{j=1}^{n}\frac{1}{j!}.
\end{equation}

We proceed to apply Proposition~\ref{prop:detAmdetB} for the matrices $A=\Theta_\omega(\mathrm{i} \omega x)$ and $B=\Theta_0(\mathrm{i} \omega x)$ to show~\eqref{eq:thetaerr}. We introduce the notation $K_{k,\sigma}$ to re-write the inner sum in~\eqref{eq:detAdetBdif0}, resulting in
\begin{subequations}
\begin{equation}\label{eq.ip.detAdetBsum1nwithK}
|\det A - \det B|
\leq \sum_{\sigma\in\mathrm{S}_{n+1}} \left(\sum_{k=1}^n K_{\sigma,k} + K_{\sigma,n+1} \right),
\end{equation}
where
\begin{equation}
K_{\sigma,k} := 
\left|A_{\sigma(k)k} - B_{\sigma(k)k}\right|
\prod_{j=1}^{k-1}\left|A_{\sigma(j)j}\right|
\prod_{j=k+1}^{n+1}\left|B_{\sigma(j)j}\right|,~~~k=1,\ldots,n+1.
\end{equation}
\end{subequations}
Thus, for a fixed permutation $\sigma\in\mathrm{S}_{n+1}$, the inner sum over $k=1,\ldots,n$ in~\eqref{eq.ip.detAdetBsum1nwithK} corresponds to
\begin{equation}\label{eq.ip.detAdetBsum1n}
\sum_{k=1}^n K_{\sigma,k} 
= \sum_{k=1}^{n} \left(\left|A_{\sigma(k)k} - B_{\sigma(k)k}\right|
\prod_{j=1}^{k-1}\left|A_{\sigma(j)j}\right|
\prod_{j=k+1}^{n+1}\left|B_{\sigma(j)j}\right|\right).
\end{equation}
Substituting $A=\Theta_\omega(\mathrm{i} \omega x)$ and $B=\Theta_0(\mathrm{i} \omega x)$ in the right-hand side therein, we obtain
\begin{equation}\label{eq:firstntermsE10}
\sum_{k=1}^{n}\left(
\left|(H_\omega)_{\sigma(k)k} - (H_0)_{\sigma(k)k}\right|
\prod_{j=1}^{k-1}\left|(H_\omega)_{\sigma(j)j}\right|
\prod_{j=k+1}^{n}\left|(H_0)_{\sigma(j)j}\right| \left|(\theta_0(\mathrm{i} \omega x))_{\sigma(n+1)}\right|\right).
\end{equation}

As a consequence of~\eqref{eq:HwljtoH0lj},~\eqref{eq:HwminusH0} and~\eqref{eq:defkappaH0} the sum in~\eqref{eq:firstntermsE10}, and respectively,~\eqref{eq.ip.detAdetBsum1n}, is bounded from above by
\begin{equation}\label{eq:firstntermsE11}
\sum_{k=1}^n K_{\sigma,k} \leq n\omega \left|(\theta_0(\mathrm{i} \omega x))_{\sigma(n+1)}\right| \prod_{j=1}^{n}\frac{1}{j!}.
\end{equation}
Since we assume $\omega\leq 1$ and $x\in[-1,1]$, the entries of $\theta_0$ as in~\eqref{eq:theta0entries} satisfy 
\begin{equation}\label{eq:theta0np1colbounded}
|(\theta_0(\mathrm{i} \omega x))_{\ell}|
= |(\omega x)^{n+1-\ell}|\leq 1,~~~~\ell=1,\ldots,n+1,
\end{equation}
and~\eqref{eq:firstntermsE11} further simplifies to
\begin{equation}\label{eq:firstntermsE12}
\sum_{k=1}^n K_{\sigma,k} \leq n\omega \prod_{j=1}^{n}\frac{1}{j!}.
\end{equation}

We proceed to treat the second inner term in~\eqref{eq.ip.detAdetBsum1nwithK} for a fixed $\sigma\in\mathrm{S}_{n+1}$, i.e.,
\begin{equation}\label{eq.ip.detAdetBtermnp1}
K_{\sigma,n+1} = \left|A_{\sigma(n+1),n+1} - B_{\sigma(n+1),n+1}\right|
\prod_{j=1}^{n}\left|B_{\sigma(j)j}\right|.
\end{equation}
Substituting $A=\Theta_\omega(\mathrm{i} \omega x)$ and $B=\Theta_0(\mathrm{i} \omega x)$ therein, and making use of~\eqref{eq:defkappaH0}, we observe
\begin{equation}\label{eq:firstntermsE20}
K_{\sigma,n+1}
\leq \left|(\theta_\omega(\mathrm{i} \omega x))_{\sigma(n+1)} - (\theta_0(\mathrm{i} \omega x))_{\sigma(n+1)}\right|
\prod_{j=1}^{n}\frac{1}{j!}.
\end{equation}
When the index $\sigma(n+1)$ corresponds to $\ell = 1,\ldots,n$, the deviation of the respective entries of $\theta_\omega$~\eqref{eq:rintdefthetaw} and $\theta_0$~\eqref{eq:theta0entries} is bounded by
\begin{equation}\label{eq:ip.thetadiff}
\begin{aligned}
|(\theta_\omega(\mathrm{i} \omega x))_{\ell} - (\theta_0(\mathrm{i} \omega x))_{\ell}|
&=\left|(\mathrm{i}\omega)^{n+1-\ell}\prod_{k=1}^{n+1-\ell}( x-x_k)-(\mathrm{i} \omega x)^{n+1-\ell}\right|\\
&
\leq (2^{n+1-\ell}+1) \omega^{n+1-\ell} \leq 2^{n+1} \omega,
\qquad\ell=1,\ldots,n.
\end{aligned}
\end{equation}
Since
\begin{equation*}
|(\theta_\omega(\mathrm{i} \omega x))_{n+1} - (\theta_0(\mathrm{i} \omega x))_{n+1}| = 0,
\end{equation*}
we may also apply the upper bound in~\eqref{eq:ip.thetadiff} for the case $\sigma(n+1)=n+1$. Thus, the upper bound in~\eqref{eq:firstntermsE20} simplifies to
\begin{equation}\label{eq:firstntermsE21}
K_{\sigma,n+1} \leq 2^{n+1}\omega\prod_{j=1}^{n}\frac{1}{j!}.
\end{equation}
Combining~\eqref{eq.ip.detAdetBsum1nwithK} for $A=\Theta_\omega(\mathrm{i} \omega x)$ and $B=\Theta_0(\mathrm{i} \omega x)$ with~\eqref{eq:firstntermsE12} and~\eqref{eq:firstntermsE21}, we arrive at
\begin{equation}\label{eq:detAdetBcase10}
|\det \Theta_\omega(\mathrm{i} \omega x) - \det \Theta_0(\mathrm{i} \omega x)|
\leq \omega(n+2^{n+1})\prod_{j=1}^{n}\frac{1}{j!} \sum_{\sigma\in\mathrm{S}_{n+1}} 1
\end{equation}
Since the number of elements in $\mathrm{S}_{n+1}$ corresponds to $(n+1)!$, i.e.,
\begin{equation}\label{eq:sumSnpermutation}
\sum_{\sigma\in\mathrm{S}_{n+1}} 1 = (n+1)!,
\end{equation}
the upper bound~\eqref{eq:detAdetBcase10} shows~\eqref{eq:thetaerr}.

\smallskip
We proceed to show~\eqref{eq:gammaerr} in a similar manner. For $A=\Gamma_\omega(\mathrm{i} \omega x)$ and $B=\Gamma_0(\mathrm{i} \omega x)$ and a fixed permutation $\sigma\in\mathrm{S}_{n+1}$, the inner sum in~\eqref{eq.ip.detAdetBsum1n} corresponds to
\begin{equation*}
\begin{aligned}
\sum_{k=1}^n K_{\sigma,k} = 
\sum_{k=1}^{n}&\left(
\left|(H_\omega)_{\sigma(k)k} - (H_0)_{\sigma(k)k}\right| \phantom{\prod_{j=1}^{k-1}}\right.\\
&\quad \left.
\prod_{j=1}^{k-1}\left|(H_\omega)_{\sigma(j)j}\right|
\prod_{j=k+1}^{n}\left|(H_0)_{\sigma(j)j}\right| \left|(\gamma_0(\mathrm{i} \omega x))_{\sigma(n+1)}\right|\right).
\end{aligned}
\end{equation*}
Similar to~\eqref{eq:firstntermsE11}, this sum is bounded by
\begin{equation}\label{eq:firstntermsE11g}
\sum_{k=1}^n K_{\sigma,k} \leq n\omega \left|(\gamma_0(\mathrm{i} \omega x))_{\sigma(n+1)}\right| \prod_{j=1}^{n}\frac{1}{j!}.
\end{equation}
Entries of $\gamma_0$~\eqref{eq:theta0entries} are of the form
\begin{equation}\label{eq:gamma0entries}
(\gamma_0(\mathrm{i} \omega x))_{\ell}=\exp[0^{n+\ell},\mathrm{i} \omega x],~~~\ell=1,\ldots,n+1.
\end{equation}
Thus, using~\eqref{eq:upperboundonddzk} we observe
\begin{equation}\label{eq:gamma0entriesbounded}
|(\gamma_0(\mathrm{i} \omega x))_\ell|
=|\exp[0^{n+\ell},\mathrm{i} \omega x]| \leq \frac{1}{(n+\ell)!}.
\end{equation}
For $\ell=1,\ldots,n+1$, we have $1/(n+\ell)!\leq 1/(n+1)!$, and thus,
and the upper bound in~\eqref{eq:firstntermsE11g} is further bounded by
\begin{equation}\label{eq:firstntermsE12g}
\sum_{k=1}^n K_{\sigma,k} \leq \frac{n\omega}{(n+1)!} \prod_{j=1}^{n}\frac{1}{j!}  .
\end{equation}
On the other hand, proceeding similar to~\eqref{eq:firstntermsE20} the term in~\eqref{eq.ip.detAdetBtermnp1} is bounded from above by
\begin{equation}\label{eq:firstntermsE20g}
K_{\sigma,n+1} \leq \left|(\gamma_\omega(\mathrm{i} \omega x))_{\sigma(n+1)} - (\gamma_0(\mathrm{i} \omega x))_{\sigma(n+1)}\right|
\prod_{j=1}^{n}\frac{1}{j!},
\end{equation}
where entries of $\gamma_\omega$~\eqref{eq:ip.rintdefgammaw} correspond to
\begin{equation*}
(\gamma_\omega(\mathrm{i} \omega x))_\ell=\exp[\mathrm{i} \omega x_{n+2-\ell},\ldots,\mathrm{i} \omega x_{2n+1},\mathrm{i} \omega x],~~~\ell=1,\ldots,n+1.
\end{equation*}
Utilizing~\eqref{eq:ddexpto1k} for the divided differences therein, we observe that these entries satisfy
\begin{equation*}
\left| (\gamma_\omega(\mathrm{i} \omega x))_\ell - \frac{1}{(n+\ell)!} \right|\leq \frac{\omega}{(n+\ell)!},~~~x\in[-1,1].
\end{equation*}
In a similar manner, the entries $(\gamma_0(\mathrm{i} \omega x))_{\ell}$ as in~\eqref{eq:gamma0entries} satisfy
\begin{equation*}
\left| (\gamma_0(\mathrm{i} \omega x))_{\ell} - \frac{1}{(n+\ell)!} \right|\leq \frac{\omega}{(n+\ell)!},~~~x\in[-1,1].
\end{equation*}
Thus, using triangular inequality we arrive at
\begin{equation*}
|(\gamma_\omega(\mathrm{i} \omega x))_{\ell}-(\gamma_0(\mathrm{i} \omega x))_{\ell}| \leq \frac{2\omega}{(n+\ell)!},
\end{equation*}
and~\eqref{eq:firstntermsE20g} is further bounded by 
\begin{equation}\label{eq:firstntermsE21g}
K_{\sigma,n+1} \leq \frac{2\omega}{(n+1)!}\prod_{j=1}^{n}\frac{1}{j!}.
\end{equation}
Combining~\eqref{eq.ip.detAdetBsum1nwithK} for $A=\Gamma_\omega(\mathrm{i} \omega x)$ and $B=\Gamma_0(\mathrm{i} \omega x)$ with~\eqref{eq:firstntermsE12g} and~\eqref{eq:firstntermsE21g}, we observe
\begin{equation}\label{eq:detAdetBcase10g}
|\det \Gamma_\omega(\mathrm{i} \omega x) - \det \Gamma_0(\mathrm{i} \omega x)|
\leq \frac{\omega(n+2)}{(n+1)!}\prod_{j=1}^{n}\frac{1}{j!}\sum_{\sigma\in\mathrm{S}_{n+1}} 1.
\end{equation}
Combining this with~\eqref{eq:sumSnpermutation}, we conclude~\eqref{eq:gammaerr}.

\end{proof}

\begin{proposition}
For $x\in[-1,1]$ we have
\begin{equation}\label{eq:ip.remPadefromratint}
\frac{\det\Gamma_\omega(\mathrm{i}\omega x)}{\widehat{p}_\omega(\mathrm{i}\omega x)}
= \frac{(-1)^{n+1} (n!)^2}{(2n)!(2n+1)!} + \mathcal{O}(\omega),~~~\omega\to0^+.
\end{equation}
\end{proposition}
\begin{proof}
We recall that the Pad\'e approximation $\widehat{r}=\widehat{p}^\dag/\widehat{p}$ corresponds to the rational interpolant to $\mathrm{e}^z$ with the interpolation node of multiplicity $2n+1$.
For the Pad\'e approximant, the linearized interpolation problem~\eqref{eq:linintinproofphat} with negative sign for $z=\mathrm{i} \omega x$ writes
\begin{equation}\label{eq:ip.lininterpolationPade}
\widehat{p}(\mathrm{i} \omega x) \mathrm{e}^{\mathrm{i} \omega x} - \widehat{p}^\dag(\mathrm{i} \omega x)  = -(\mathrm{i} \omega x)^{2n+1} v_\omega(\mathrm{i} \omega x),
\end{equation}
where we assume that $\widehat{p}$ corresponds to the representation~\eqref{eq:rintdenomdetTpfull}.
Since $\widehat{p}$ is non-zero on the imaginary axis, we may divide by $\widehat{p}(\mathrm{i} \omega x)$ and substitute $\det\Gamma_0(\mathrm{i}\omega x)$ from~\eqref{eq:ip.linerrgenall} for $v_\omega(\mathrm{i} \omega x)$ to obtain
\begin{equation*}
\mathrm{e}^{\mathrm{i} \omega x} - \widehat{r}(\mathrm{i}\omega x)
=  -(\mathrm{i}\omega x)^{2n+1}\frac{\det\Gamma_0(\mathrm{i}\omega x)}{\widehat{p}(\mathrm{i}\omega x)},
\end{equation*}
Comparing with the asymptotic error~\eqref{eq:leadingordererrorpade} for $z=\mathrm{i}\omega x$,
\begin{equation}
(\mathrm{i} \omega x)^{2n+1}\frac{\det\Gamma_0(\mathrm{i}\omega x)}{\widehat{p}(\mathrm{i}\omega x)}
= (-1)^{n+1}  \frac{(n!)^2 (\mathrm{i} \omega x)^{2n+1}}{(2n)!(2n+1)!} + \mathcal{O}(\omega^{2n+2}),~~~\omega\to {0+},
\end{equation}
we derive
\begin{equation}
\frac{\det\Gamma_0(\mathrm{i}\omega x)}{\widehat{p}(\mathrm{i}\omega x)}
=  \frac{(-1)^{n+1} (n!)^2}{(2n)!(2n+1)!} + \mathcal{O}(\omega).
\end{equation}
Thus, we have
\begin{equation}\label{eq:ip.remPadefromratintp1}
\frac{\det\Gamma_\omega(\mathrm{i}\omega x)}{\widehat{p}_\omega(\mathrm{i}\omega x)}
- \frac{(-1)^{n+1} (n!)^2}{(2n)!(2n+1)!} 
= \frac{\det\Gamma_\omega(\mathrm{i}\omega x)}{\det\Theta_\omega(\mathrm{i}\omega x)}
- \frac{\det\Gamma_0(\mathrm{i}\omega x)}{\det\Theta_0(\mathrm{i}\omega x)} + \mathcal{O}(\omega).
\end{equation}
To conclude~\eqref{eq:ip.remPadefromratint}, we proceed to show
\begin{equation}\label{eq:gammaadivthetaOnow}
\left|\frac{\det\Gamma_\omega(\mathrm{i}\omega x)}{\widehat{p}_\omega(\mathrm{i}\omega x)} - \frac{\det\Gamma_0(\mathrm{i}\omega x)}{\widehat{p}(\mathrm{i}\omega x)}\right| \leq \mathcal{O}(\omega),~~~\omega\to 0^+.
\end{equation}
We first recall $\widehat{p}_\omega(\mathrm{i}\omega x)=\det \Theta_\omega(\mathrm{i} \omega x)$ and $\widehat{p}(\mathrm{i}\omega x)=\det \Theta_0(\mathrm{i} \omega x)$. To simplify our notation we ignore the function argument $\mathrm{i}\omega x$ in the following.
The deviation in~\eqref{eq:gammaadivthetaOnow} corresponds to
\begin{equation*}
\frac{\det\Gamma_\omega}{\det\Theta_\omega} - \frac{\det\Gamma_0}{\det\Theta_0}
= \frac{(\det\Gamma_\omega - \det\Gamma_0)\det\Theta_0 + \det\Gamma_0(\det\Theta_0 - \det\Theta_\omega)}{\det\Theta_0 \det\Theta_\omega}.
\end{equation*}
Using the upper bounds~\eqref{eq:thetaerr} and~\eqref{eq:gammaerr}, we observe
\begin{equation*}
\left|(\det\Gamma_\omega - \det\Gamma_0)\det\Theta_0 + \det\Gamma_0(\det\Theta_0 - \det\Theta_\omega) \right|
\leq \omega (c_{\gamma,n}|\det\Theta_0| + c_{\theta,n}|\det\Gamma_0|).
\end{equation*}
For a fixed degree $n$ and $\omega$ sufficiently small the upper bound~\eqref{eq:thetaerr} entails
\begin{equation*}
|\det\Theta_\omega| \geq  |\det\Theta_0| - c_{\theta,n} \omega
\geq |\det\Theta_0|/2,
\end{equation*}
which further leads to
\begin{equation*}
\left|\frac{\det\Gamma_\omega}{\det\Theta_\omega} - \frac{\det\Gamma_0}{\det\Theta_0}\right| \leq 
\frac{ 2\omega (c_{\gamma,n}|\det\Theta_0| + c_{\theta,n}|\det\Gamma_0|)}
{|\det\Theta_0|^2}
= \mathcal{O}(\omega).
\end{equation*}
This shows~\eqref{eq:gammaadivthetaOnow}, and together with~\eqref{eq:ip.remPadefromratintp1} we conclude~\eqref{eq:ip.remPadefromratint}.
\end{proof}

To proceed we recall the notation
\begin{equation*} \tag{\ref{eq:errnormnotation}}
\|r_\omega-\exp(\omega\cdot)\|
= \max_{x\in[-1,1]} |r_\omega(\mathrm{i} x)-\mathrm{e}^{\mathrm{i} \omega x}|.
\end{equation*}

\begin{proposition}\label{prop:ratintasymerr}
The rational interpolant $r_\omega$ to $\mathrm{e}^{\omega z}$ with interpolation nodes $\mathrm{i} x_1,\ldots,\mathrm{i} x_{2n+1}\in\mathrm{i}\mathbb{R}$, where $x_j\in[-1,1]$, has the asymptotic error
\begin{equation}\label{eq:interr}
\|r_\omega-\exp(\omega\cdot)\| 
= \max_{x\in[-1,1]}\left|\prod_{j=1}^{2n+1}(x-x_j) \right|  \frac{(n!)^2 \omega^{2n+1}}{(2n)!(2n+1)!}+\mathcal{O}(\omega^{2n+2}),
~~~\omega\to 0^+.
\end{equation}
\end{proposition}

\begin{proof}
Let $\widehat{r}_\omega=\widehat{p}^\dag_\omega/\widehat{p}_\omega$ denote the rational interpolant to $\mathrm{e}^{z}$ with interpolation nodes $\mathrm{i}\omega x_1,\ldots,\mathrm{i}\omega x_{2n+1}\in\mathrm{i}\mathbb{R}$. With the remainder $v_\omega = \det\Gamma_\omega$ from~\eqref{eq:ip.linerrgenall}, the linearized interpolation problem~\eqref{eq:linintinproofphat} for $z=\mathrm{i}\omega x$ writes
\begin{equation}\label{eq:linerrorphat}
\widehat{p}_\omega^\dag (\mathrm{i} \omega x) - \mathrm{e}^{\mathrm{i} \omega x} \widehat{p}_\omega (\mathrm{i} \omega x) = (\mathrm{i} \omega)^{2n+1}\prod_{j=1}^{2n+1} (x-x_j)\det\Gamma_\omega(\mathrm{i} \omega x).
\end{equation}
Since $\widehat{p}_\omega$ converges to the denominator of the Pad\'e approximation $\widehat{p}$ which is non-zero on the imaginary axis, and since $r_\omega(z) = \widehat{p}_\omega^\dag(\omega z)/\widehat{p}_\omega (\omega z)$ as in~\eqref{eq:ip.rfromphat}, the linearized error~\eqref{eq:linerrorphat} entails
\begin{equation*}
r_\omega (\mathrm{i} x) - \mathrm{e}^{\mathrm{i} \omega x}  = (\mathrm{i} \omega)^{2n+1}\prod_{j=1}^{2n+1} (x-x_j)\frac{\det\Gamma_\omega(\mathrm{i} \omega x)}{\widehat{p}_\omega (\mathrm{i} \omega x)}.
\end{equation*}
Making use of~\eqref{eq:ip.remPadefromratint} for a given $x\in[-1,1]$ and taking the absolute value, we observe
\begin{equation*}
|r_\omega (\mathrm{i} x) - \mathrm{e}^{\mathrm{i} \omega x}| = \omega^{2n+1}\left|\prod_{j=1}^{2n+1} (x-x_j)\right| \frac{(n!)^2}{(2n)!(2n+1)!} + \mathcal{O}(\omega),~~~\omega\to0^+.
\end{equation*}
Taking the maximum over $x\in[-1,1]$, we conclude~\eqref{eq:interr}.
\end{proof}

\begin{proposition}\label{prop:ratintasymlowererrbound}
Let $r_\omega$ denote a rational interpolant to $\mathrm{e}^{\omega z}$ with interpolation nodes $\mathrm{i} x_1(\omega),\ldots,\mathrm{i} x_{2n+1}(\omega)$, where $x_j(\omega)\in[-1,1]$ may depend on $\omega$, then
\begin{equation}\label{eq:asymerrlowerbound}
\|r_\omega-\exp(\omega\cdot)\|
\geq\frac{2(n!)^2 }{(2n)!(2n+1)!} \left(\frac{\omega}{2}\right)^{2n+1}+\mathcal{O}(\omega^{2n+2}),~~~\omega\to0^+.
\end{equation}
\end{proposition}
\begin{proof}
To proof this assertions in full detail we first remark that, using the classical definition of the $\mathcal{O}$-notation the asymptotic bound~\eqref{eq:asymerrlowerbound} is equivalent to
\begin{equation}\label{eq:prunitarybestasymp2}
\limsup_{\omega \to 0^+} \frac{\frac{2(n!)^2}{(2n)!(2n+1)!} \left( \frac{\omega }{2} \right)^{2n+1}  
- \|r_\omega-\exp(\omega\cdot)\|}{\omega^{2n+2}} <\infty.
\end{equation}

We first consider rational interpolants $\widetilde{r}_\omega$ to $\mathrm{e}^{\omega z}$ with fixed interpolation nodes.
Following Proposition~\ref{prop:ratintasymerr}, the rational interpolant $\widetilde{r}_\omega$ to $\mathrm{e}^{\mathrm{i} \omega x}$ with given interpolation nodes $\mathrm{i} x_1,\ldots,\mathrm{i} x_{2n+1}$, where $x_j\in[-1,1]$, has the asymptotic error representation
\begin{equation}\label{eq:asymerrinterinproof1}
\|\widetilde{r}_\omega-\exp(\omega\cdot)\| 
= \max_{x\in[-1,1]}\left|\prod_{j=1}^{2n+1}(x-x_j) \right|  \frac{(n!)^2 \omega^{2n+1}}{(2n)!(2n+1)!}+\mathcal{O}(\omega^{2n+2}).
\end{equation}
Substituting the lower bound from~\eqref{eq:Chebmin} in~\eqref{eq:asymerrinterinproof1}, we observe
\begin{equation}\label{eq:asymbound2fixednodes}
\|\widetilde{r}_\omega-\exp(\omega\cdot)\| 
\geq \frac{2(n!)^2}{(2n)!(2n+1)!} \left( \frac{\omega }{2} \right)^{2n+1}+\mathcal{O}(\omega^{2n+2}),
\end{equation}
which proves~\eqref{eq:prunitarybestasymp2} for the case of fixed nodes. It remains to show that this holds true when interpolation nodes depend on $\omega$.
For the case of fixed interpolation nodes, the asymptotic bound~\eqref{eq:asymbound2fixednodes} corresponds to
\begin{equation}\label{eq:asymerrinterinproof2}
\limsup_{\omega \to 0^+} \frac{\frac{2(n!)^2}{(2n)!(2n+1)!} \left( \frac{\omega }{2} \right)^{2n+1} - \|\widetilde{r}_\omega-\exp(\omega\cdot)\| }{\omega^{2n+2}}
=: K(x_1,\ldots,x_n) < \infty.
\end{equation}
Thus, the limit superior $K(x_1,\ldots,x_n)$ therein is bounded from above for any set of fixed nodes $x_1,\ldots,x_{2n+1}$ with $x_j\in[-1,1]$, i.e., $(x_1,\ldots,x_{2n+1}) \in [-1,1]^{2n+1}$ where $[-1,1]^{2n+1}$ denotes the $2n+1$-dimensional Cartesian product of the interval $[-1,1]$. 
Since $[-1,1]^{2n+1}$ is a compact set, we find a set of interpolation nodes $(\widehat{x}_{1},\ldots,\widehat{x}_{2n+1}) \in [-1,1]^{2n+1}$ for which $K(x_1,\ldots,x_n)$ attains its maximum, i.e.,
\begin{equation*}
\widehat{K}
:= \max_{(x_{1},\ldots,x_{2n+1}) \in [-1,1]^{2n+1}} K(x_1,\ldots,x_n) 
=  K(\widehat{x}_{1},\ldots,\widehat{x}_{2n+1})<\infty.
\end{equation*}
Thus, for the maximum $\widehat{K}$ as above, a given set of nodes $x_1,\ldots,x_n\in[-1,1]^{2n+1}$, and $\varepsilon>0$ there exists $\widetilde{\omega}_\varepsilon>0$ s.t.\ the respective rational interpolant $\widetilde{r}_\omega$ satisfies
\begin{equation}\label{eq:asymerrinterinproof3}
\frac{\frac{2(n!)^2}{(2n)!(2n+1)!} \left( \frac{\omega }{2} \right)^{2n+1} - \|\widetilde{r}_\omega-\exp(\omega\cdot)\| }{\omega^{2n+2}}
< \widehat{K} +\varepsilon,~~~\omega\in(0,\widetilde{\omega}_\varepsilon).
\end{equation}
Due to compactness arguments, we find $\omega_\varepsilon$ s.t.\ this holds true for any set of nodes in $[-1,1]^{2n+1}$.

Under these considerations, we prove~\eqref{eq:prunitarybestasymp2} by contradiction. Let $x_1(\omega),\ldots,x_{2n+1}(\omega)$ be a given set of nodes where $x_j(\omega)\in[-1,1]$ depends on $\omega$, and let $r_\omega$ correspond to the rational interpolant to $\mathrm{e}^{\omega z}$ with interpolation nodes $\mathrm{i} x_1(\omega),\ldots,\mathrm{i} x_{2n+1}(\omega)$.
We assume~\eqref{eq:prunitarybestasymp2}, and respectively,~\eqref{eq:asymerrinterinproof2}, does not hold true. Then there exists a sequence $\{\omega_j>0\}_{j\in\mathbb{N}}$ with $\omega_j\to 0$ for $j\to \infty$ and
\begin{equation*}
\frac{\frac{2(n!)^2}{(2n)!(2n+1)!} \left( \frac{\omega_j }{2} \right)^{2n+1}  
- \| r_{\omega_j} - \exp(\omega_j\cdot)\| }{\omega_j^{2n+2}} \to \infty,~~~~j\to \infty.
\end{equation*}
However, this is contradictory to~\eqref{eq:asymerrinterinproof3} since the upper bound therein holds true for $\omega\in(0,\omega_\varepsilon)$ uniformly in the choice of the interpolation nodes.
Thus,~\eqref{eq:prunitarybestasymp2} holds true for the case that interpolation nodes depend on $\omega$, which completes this proof.
\end{proof}

\section{Technicalities for \texorpdfstring{$\omega\to(n+1)\pi^-$}{w to (n+1)pi-}}\label{sec:proofswtonp1pi}

In the present appendix we provide proofs to propositions~\ref{prop:rforwtonp1pi},~\ref{prop:polesconvwtonp1pi} and~\ref{prop:nodesconvwtonp1pi}, which correspond to the case $\omega\to(n+1)\pi^-$. The idea behind these proofs is motivated by the representation $r(\textrm{i} x) = \textrm{e}^{\textrm{i} g(x)}$ of unitary rational functions and the fact that the arc tangent functions in $g(x)$ converge to step functions in specific cases. In particular, for $\varepsilon\to0^+$ and $y\in\mathbb{R}$,
\begin{equation}\label{eq:arctantostepfct}
\arctan \frac{y}{\varepsilon} \to
\left\{\begin{array}{rll}
-\frac{\pi}{2},&~~~&y<0,\\
0,&& y=0,~~\text{and}\\
\frac{\pi}{2},&&y<0.
\end{array}\right.
\end{equation}
This relation is especially utilized in the auxiliary propositions~\ref{prop:appendixBauxgeps1},~\ref{prop:appendixBauxgeps2} and~\ref{prop:appendixBauxgeps3} in various forms.

Numerical illustrations of phase functions together with the corresponding phase errors are also shown in figures~\ref{fig:ip.geps} and~\ref{fig:ip.gnp1pi} to visually support proofs of the present section.
In Fig.~\ref{fig:ip.geps}, the function $g_\varepsilon$ from~\eqref{eq:defgxi} below, which corresponds to a phase function with pre-assigned poles, is illustrated together with its phase error. This function behaves similar to the phase function $g_\omega$ of the unitary best approximation in the case $\omega\approx(n+1)\pi$, for which a numerical illustration is provided in Fig.~\ref{fig:ip.gnp1pi}.

\subsection*{Pre-assigned poles and the proof to  Proposition~\ref{prop:rforwtonp1pi}}
Proposition~\ref{prop:rforwtonp1pi} shows how to construct a unitary rational function $r$ which has an approximation error strictly smaller than two for $\omega$ sufficiently close to $(n+1)\pi$.

We recall that a unitary rational function has a representation $r(\textrm{i} x) = \textrm{e}^{\textrm{i} g(x)}$ where $g$ is of the form~\eqref{eq:defg} and $s_j=\xi_j+\textrm{i} \mu_j$ therein correspond to the poles of $r$. For the proof of Proposition~\ref{prop:rforwtonp1pi} we use pre-assigned poles. In particular, this proof relies on phase functions of the form
\begin{equation}\label{eq:defgxi}
g_\varepsilon(x) := 2\sum_{j=1}^n \arctan \frac{x-\mu_j}{\varepsilon},~~~\text{with}~~\mu_j:=-1+\frac{2j}{n+1},
\end{equation}
and transformations of this function. We highlight that $\xi_j$ is replaced by $\varepsilon>0$ in this representation, where $\varepsilon$ is be specified in the proof of Proposition~\ref{prop:rforwtonp1pi} further below.

We proceed with an auxiliary proposition.
\begin{proposition}[Approaching step functions, part 1]
\label{prop:appendixBauxgeps1}
Let $\varepsilon>0$ be sufficiently small, then the following inequalities hold true for $g_\varepsilon$ in~\eqref{eq:defgxi}.
\begin{itemize}
\begin{subequations}
\item For $x\leq \mu_1-\tfrac{1}{n+1}$,
\begin{equation}\label{eq:gxi1}
-n\pi<g_\varepsilon (x) \leq -n\pi + 3\varepsilon n(n+1),
\end{equation}
\item for $x\in[\mu_k - \tfrac{1}{n+1},\mu_k + \tfrac{1}{n+1}]$ with $k=1,\ldots,n$,
\begin{equation}\label{eq:ip.gxi2}
\left| g_\varepsilon(x) -\left( 2\arctan\frac{x-\mu_k}{\varepsilon} + (-n+2k-1) \pi \right)\right | \leq  3\varepsilon n(n+1),
\end{equation}
\item and for $x\geq \mu_n+\tfrac{1}{n+1}$,
\begin{equation}\label{eq:gxi3}
n\pi - 3 \varepsilon n(n+1) \leq g_\varepsilon(x)\leq n\pi.
\end{equation}
\end{subequations}
\end{itemize}
\end{proposition}
\begin{proof}
These results rely on properties of the arc tangent function. Since $\arctan y$ is monotonically increasing in $y$ we note that for $j\in\{1,\ldots,n\}$ and $x\geq\mu_j+\tfrac{1}{n+1}$,
\begin{equation*}
2\arctan \frac{x-\mu_j}{\varepsilon } \geq 2\arctan \frac{1}{(n+1)\varepsilon } = \pi - 2\arctan((n+1)\varepsilon).
\end{equation*}
Asymptotically for $\varepsilon\to0^+$, this lower bound satisfies
\begin{equation*}
\pi - 2\arctan((n+1)\varepsilon)
= \pi - 2\varepsilon (n+1) + \mathcal{O}(\varepsilon^3).
\end{equation*}
Thus, for a sufficiently small $\varepsilon>0$,
\begin{subequations}\label{eq:ip.wnp1c0}
\begin{equation}\label{eq:ip.wnp1c12}
2\arctan \frac{x-\mu_j}{\varepsilon }
\geq \pi - 3\varepsilon (n+1),~~~\text{for $x\geq\mu_j+\frac{1}{n+1}$},
\end{equation}
and $j\in\{1,\ldots,n\}$ uniformly.
Analogously, for $x\leq\mu_j-\frac{1}{n+1}$ we observe
\begin{equation*}
2\arctan \frac{x-\mu_j}{\varepsilon } \leq -\pi +2\arctan((n+1)\varepsilon),
\end{equation*}
and under the conditions of~\eqref{eq:ip.wnp1c12} we obtain
\begin{equation}\label{eq:ip.wnp1c11}
2\arctan \frac{x-\mu_j}{\varepsilon }
\leq -\pi + 3\varepsilon (n+1),~~~\text{for $x\leq\mu_j-\frac{1}{n+1}$},
\end{equation}
\end{subequations}
and $j\in\{1,\ldots,n\}$ uniformly.

\bigskip
We proceed to simplify $g_\varepsilon(x)$ for $x \in [\mu_k - \tfrac{1}{n+1},\mu_k + \tfrac{1}{n+1}]$ with $k\in\{1,\ldots,n\}$ fixed.
The function $g_\varepsilon$ in~\eqref{eq:defgxi} corresponds to
\begin{equation}\label{eq:ip.gxi1}
g_\varepsilon(x) = 2\sum_{j=1}^{k-1} \arctan \frac{x-\mu_j}{\varepsilon}
+ 2\arctan\frac{x-\mu_k}{\varepsilon}
+ 2\sum_{j=k+1}^{n} \arctan \frac{x-\mu_j}{\varepsilon} .
\end{equation}
For the sum containing $j=1,\ldots,k-1$ we note $x > \mu_j + \tfrac{1}{n+1}$, and thus,~\eqref{eq:ip.wnp1c12} implies
\begin{equation*}
2\sum_{j=1}^{k-1} \arctan \frac{x-\mu_j}{\varepsilon} = (k-1)\pi - \delta_a(x),~~~\text{where}~~0<\delta_a(x)\leq 3\varepsilon (k-1)(n+1),
\end{equation*}
and for the sum with $j=n-k,\ldots,n$ we note $x < \mu_j-\tfrac{1}{n+1}$, and thus,~\eqref{eq:ip.wnp1c11} implies
\begin{equation*}
2\sum_{j=k+1}^{n} \arctan \frac{x-\mu_j}{\varepsilon} = -(n-k)\pi + \delta_b(x),~~~\text{where}~~0<\delta_b(x)\leq 3\varepsilon (n-k)(n+1).
\end{equation*}
Thus, for $x\in[\mu_k - \tfrac{1}{n+1},\mu_k + \tfrac{1}{n+1}]$, the function $g_\varepsilon$ satisfies to~\eqref{eq:ip.gxi2}.

Moreover, for $x\leq \mu_1-1/(n+1)$ we may apply~\eqref{eq:ip.wnp1c11} for each sum term in~\eqref{eq:defgxi}, which yields~\eqref{eq:gxi1}. 
In a similar manner we show~\eqref{eq:gxi3} for the case $x\geq \mu_n+1/(n+1)$ by applying~\eqref{eq:ip.wnp1c12} for each sum term in~\eqref{eq:defgxi}.
\end{proof}

\begin{proof}[\bf Proof of Proposition~\ref{prop:rforwtonp1pi}]
We first fix $\omega = (n+1)\pi$, and we proceed to specify $\varepsilon>0$ s.t.\ $g_\varepsilon$ as in~\eqref{eq:defgxi} satisfies
\begin{equation}\label{eq:ip.condphaseerreps}
\max_{x\in[-1+n\varepsilon,1-n\varepsilon]}|g_\varepsilon(x)-\omega x | < \pi.
\end{equation}
The function $g_\varepsilon(x)-\omega x$ is also referred to as phase error in the present proof. For an illustration of $g_\varepsilon$ from~\eqref{eq:defgxi} and its phase error we also refer to Fig.~\ref{fig:ip.geps}.

In the present proof we require $\varepsilon$ sufficiently small s.t.\ the inequalities of Proposition~\ref{prop:appendixBauxgeps1} hold true.
In addition, we assume
\begin{equation}\label{eq:ip.wnp1c1}
\varepsilon < \frac{\min\left\{2\arctan \frac{1}{n}, 1\right\}}{3n(n+1)}.
\end{equation}
\begin{subequations}\label{eq:ip.wnp1c1entails}
We highlight some consequences of~\eqref{eq:ip.wnp1c1}. In particular,
\begin{equation}\label{eq:ip.wnp1c2}
 3 \varepsilon n(n+1) < 2\arctan \frac{1}{n} = \pi - 2\arctan n.
\end{equation}
Moreover, since $\arctan 1/n < \arctan n $, the assumption~\eqref{eq:ip.wnp1c1} further entails 
\begin{equation}\label{eq:ip.wnp1c3}
3\varepsilon n(n+1) < 2\arctan n,
\end{equation}
and $ 3 \varepsilon n(n+1) <1 $ implies
\begin{equation}\label{eq:ip.wnp1c4}
(3+\pi) \varepsilon n(n+1) < 3\pi \varepsilon n(n+1) <\pi.
\end{equation}
\end{subequations}

We proceed to prove~\eqref{eq:ip.condphaseerreps} by showing
\begin{equation}\label{eq:ip.condphaseerrepsx}
|g_\varepsilon(x) - \omega x|<\pi,
\end{equation}
separately for $x$ located in the sub-intervals specified below.  Namely,
\begin{equation*}
\begin{array}{ll}
\text{case 1},~~~&  \text{for $x \in[-1+ n\varepsilon, \mu_1 - \tfrac{1}{n+1}]$ and},\\
\text{case 2},&\text{for $x \in [\mu_k - \tfrac{1}{n+1},\mu_k + \tfrac{1}{n+1}]$ with $k\in\{1,\ldots,n\}$ where further}\\
\text{case 2a},&\text{covers $x\in[\mu_k-\tfrac{1}{n+1},\mu_k-\varepsilon n]$,}\\
\text{case 2b},&\text{covers $x\in[\mu_k-\varepsilon n,\mu_k]$,}\\
\text{case 2c},&\text{covers $x\in[\mu_k,\mu_k+\varepsilon n]$, and}\\
\text{case 2d},&\text{covers $x\in[\mu_k+\varepsilon n,\mu_k + \tfrac{1}{n+1}]$, and finally,}\\
\text{case 3},& \text{for $x \in[\mu_n + \tfrac{1}{n+1}, 1 - n\varepsilon]$}.\\
\end{array}
\end{equation*}
We remark that~\eqref{eq:ip.wnp1c1} entails $ \varepsilon n < 1/(n+1) $, and thus, these sub-intervals are non-empty.

 \begin{figure}
\centering
\includegraphics{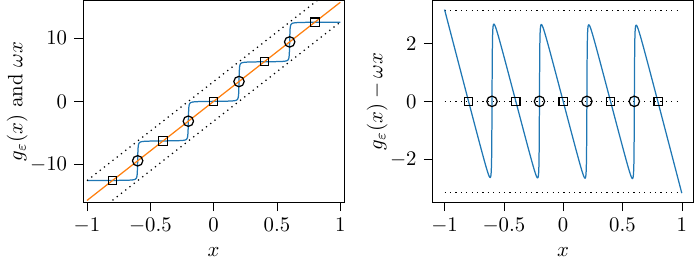}
\caption{
This figure illustrates the function $g_\varepsilon$~\eqref{eq:defgxi} and its phase error for $n=4$, $\varepsilon=2\cdot10^{-3}$, and $\omega=(n+1)\pi$. {\bf left:} The curved line corresponds to $g_\varepsilon(x)$, the middle diagonal line illustrates $\omega x$, and the dotted-limes show $\omega x+\pi$ and $\omega x-\pi$.
In combination the symbols ($\square$) and ($\circ$) mark $\omega x$ for the points $t_j=-1+j/(n+1)$ with $j=1,\ldots,2n+1$, where the ($\circ$) symbols particularly mark $\omega \mu_j$ for the points $\mu_j$ from~\eqref{eq:defgxi}. {\bf right:} This plot shows the phase error $g_\varepsilon(x)-\omega x$. The dotted lines illustrate $\pi$, $0$ and $-\pi$ as a reference. Similar to the left plot, the ($\square$) and ($\circ$) symbols mark the points $-1+j/(n+1)$ on the $x$-axis. As discussed in the proof of Proposition~\ref{prop:rforwtonp1pi}, for a sufficiently small $\varepsilon>0$ and $x\in[-1+n\varepsilon,1-n\varepsilon]$ the phase error is bounded by $\pi$ in absolute value.}
\label{fig:ip.geps}
\end{figure}

\noindent{\bf Case 1.} We consider $ x \in[-1+ n\varepsilon, \mu_1 - \tfrac{1}{n+1}]$.
Evaluating $\omega x$ at the boundaries of this interval we observe
\begin{equation}\label{eq:ip.atantransepsi}
-(n+1)\pi + n(n+1)\pi \varepsilon \leq \omega x\leq  -n\pi.
\end{equation}
Moreover, the inequalities~\eqref{eq:gxi1} apply for the current case.
Combining the lower bound in~\eqref{eq:gxi1} with the upper bound in~\eqref{eq:ip.atantransepsi}, we arrive at
\begin{equation}\label{eq:ip.case1gxwlowerb}
g_\varepsilon (x) - \omega x \geq 0.
\end{equation}
Combining the upper bound in~\eqref{eq:gxi1} with the lower bound in~\eqref{eq:ip.atantransepsi}, we conclude
\begin{equation*}
g_\varepsilon (x) - \omega x \leq \pi + (3-\pi)\varepsilon n(n+1) < \pi.
\end{equation*}
Together with~\eqref{eq:ip.case1gxwlowerb}, this shows~\eqref{eq:ip.condphaseerrepsx} for case 1.

\noindent{\bf Case 2.}
We recall that we assume $\varepsilon$ to sufficiently small s.t.\ Proposition~\ref{prop:appendixBauxgeps1} holds true.
Thus,~\eqref{eq:gxi1} applies for the case $x \in [\mu_k - \tfrac{1}{n+1},\mu_k + \tfrac{1}{n+1}]$ with $k\in\{1,\ldots,n\}$.

\noindent{\bf Case 2a.} We now more specifically consider $x\in[\mu_k-\tfrac{1}{n+1},\mu_k-\varepsilon n]$. Evaluating $\omega x$ at the boundaries of this interval, we observe
\begin{equation}\label{eq:ip.case2wx}
(-n+2k-2)\pi \leq \omega x \leq (-n+2k-1)\pi - \varepsilon n(n+1)\pi.
\end{equation}
We remark
\begin{equation}\label{eq:ip.arctannepsi}
2\arctan \frac{x-\mu_k}{\varepsilon } = -2\arctan n, ~~~\text{for $x=\mu_k-n\varepsilon$}.
\end{equation}
Thus, since arc tangents function is strictly larger than $-\pi/2$ and monotonically increasing, the current case entails
\begin{equation}\label{eq:ip.arctannepsibounds}
-\pi < 2\arctan \frac{x-\mu_k}{\varepsilon } \leq -2\arctan n.
\end{equation}
Combining~\eqref{eq:ip.gxi2} and~\eqref{eq:ip.arctannepsibounds}, we observe
\begin{equation}\label{eq:ip.case2gxupperb}
g_\varepsilon(x) \leq (-n+2k-1) \pi - 2\arctan n +3\varepsilon n(n+1),
\end{equation}
and
\begin{equation}\label{eq:ip.case2gxlowerb}
g_\varepsilon(x) > (-n+2k-2) \pi  - 3 \varepsilon n(n+1).
\end{equation}
Combining the lower bound in~\eqref{eq:ip.case2wx} with~\eqref{eq:ip.case2gxupperb}, and making use of the assumption~\eqref{eq:ip.wnp1c3}, we conclude
\begin{equation*}
g_\varepsilon(x)-\omega x
\leq  \pi - 2\arctan n +3\varepsilon n(n+1) < \pi.
\end{equation*}
Furthermore, the upper bound in~\eqref{eq:ip.case2wx} together with~\eqref{eq:ip.case2gxlowerb} yields
\begin{equation*}
g_\varepsilon(x)-\omega x  > -\pi +(\pi- 3) \varepsilon n(n+1) > -\pi.
\end{equation*}
Thus, we conclude~\eqref{eq:ip.condphaseerrepsx} for case 2a.

\noindent{\bf Case 2b.} We consider $x\in[\mu_k-\varepsilon n,\mu_k]$. In this case, $\omega x$ is bounded by
\begin{equation}\label{eq:ip.case3wx}
(-n+2k-1)\pi - \varepsilon n(n+1)\pi \leq \omega x \leq (-n+2k-1)\pi.
\end{equation}
Making use of~\eqref{eq:ip.arctannepsi} and that $\arctan$ is monotonically increasing, we have
\begin{equation*}
-2\arctan n \leq 2\arctan \frac{x-\mu_k}{\varepsilon } \leq 0.
\end{equation*}
Combining this with~\eqref{eq:ip.gxi2}, we arrive at
\begin{equation}\label{eq:ip.case3gxupperb}
g_\varepsilon(x) \leq (-n+2k-1) \pi +3\varepsilon n(n+1),
\end{equation}
and
\begin{equation}\label{eq:ip.case3gxlowerb}
g_\varepsilon(x) \geq (-n+2k-1) \pi - 2\arctan n - 3 \varepsilon n(n+1).
\end{equation}
Combining~\eqref{eq:ip.case3gxupperb} with the lower bound in~\eqref{eq:ip.case3wx}, and making use of the assumption~\eqref{eq:ip.wnp1c4},
\begin{equation}\label{eq:ip.case3phaseerrupperb}
g_\varepsilon(x)-\omega x \leq 3\varepsilon n(n+1)  + \varepsilon n(n+1)\pi < \pi.
\end{equation}
In a similar manner,~\eqref{eq:ip.case3gxlowerb}, the upper bound in~\eqref{eq:ip.case3wx} and
assumption~\eqref{eq:ip.wnp1c3} entail
\begin{equation}\label{eq:ip.case3phaseerrlowerb}
g_\varepsilon(x)-\omega x \geq  - 2\arctan n - 3 \varepsilon n(n+1) > -\pi.
\end{equation}
The inequalities~\eqref{eq:ip.case3phaseerrupperb} and~\eqref{eq:ip.case3phaseerrlowerb} show that~\eqref{eq:ip.condphaseerrepsx} holds true for $x$ associated with case~2b.

We only provide a sketch of the proof for the remaining cases, since these cases are treated analogously to previous cases.
\begin{itemize}
\item {\bf Case 2c} with $x\in[\mu_k,\mu_k+\varepsilon n]$. Similar to case 2b, the upper bound~\eqref{eq:ip.condphaseerrepsx} follows as a result of
\begin{align*}
&(-n+2k-1)\pi \leq \omega x \leq (-n+2k-1)\pi + \varepsilon n(n+1)\pi,\\
&g_\varepsilon(x)\leq (-n+2k-1) \pi + 2\arctan n + 3 \varepsilon n(n+1) ,~~~\text{and}\\
&g_\varepsilon(x)\geq (-n+2k-1) \pi - 3 \varepsilon n(n+1) .
\end{align*}

\item {\bf Case 2d,} $x\in[\mu_k+\varepsilon n,\mu_k+\tfrac{1}{n+1}]$. Similar to case 2a, for the respective $x$ the upper bound~\eqref{eq:ip.condphaseerrepsx} follows from
\begin{align*}
&(-n+2k-1)\pi + \varepsilon n(n+1)\pi \leq \omega x \leq(-n+2k)\pi,\\
&g_\varepsilon(x)\leq (-n+2k) \pi + 3 \varepsilon n(n+1)  ,~~~\text{and}\\
&g_\varepsilon(x)\geq (-n+2k-1) \pi + 2\arctan n - 3 \varepsilon n(n+1).
\end{align*}
\item {\bf Case 3.} Consider $x \in [\mu_n + \tfrac{1}{n+1}, 1 - n\varepsilon]$.
Similar to case 1,~\eqref{eq:ip.condphaseerrepsx} holds true as a result of
\begin{align*}
& n\pi \leq \omega x \leq (n+1)\pi - n(n+1)\pi \varepsilon,~~~\text{and}\\
&n\pi - 3 \varepsilon n(n+1) \leq g_\varepsilon(x)\leq n\pi~~~\text{from~\eqref{eq:gxi3}}.
\end{align*}
\end{itemize}
For a sufficiently small $\varepsilon>0$ we have shown that~\eqref{eq:ip.condphaseerrepsx} holds true $x\in[-1+n\varepsilon,1-n\varepsilon]$, which implies that~\eqref{eq:ip.condphaseerreps} holds true for $g_\varepsilon$ and $\omega=(n+1)\pi$ fixed. For the re-scaled parameter $\widetilde{\omega} := (1-n\varepsilon)(n+1)\pi$ and the transformed function $\widetilde{g}_\varepsilon(x) := g_\varepsilon\left(x(1-n\varepsilon)\right)$, the upper bound~\eqref{eq:ip.condphaseerreps} entails
\begin{equation}\label{eq:ip.phaseerrtransformpi}
\max_{x\in[-1,1]} |\widetilde{g}_\varepsilon(x)-\widetilde{\omega} x| 
= \max_{x\in[-1+n\varepsilon,1-n\varepsilon]} |g_\varepsilon(x)-\omega x| < \pi.
\end{equation}
The transformed function $\widetilde{g}_\varepsilon$ has the representation
\begin{equation}\label{eq:ip.defgxitilde}
\widetilde{g}_\varepsilon(x) 
= 2\sum_{j=1}^n \arctan \frac{x-\widetilde{\mu}_j}{\widetilde{\varepsilon}},~~~\text{where}~~
\widetilde{\mu}_j = \frac{\mu_j}{1-n\varepsilon},~~~\text{and}~~
\widetilde{\varepsilon} = \frac{\varepsilon}{1-n\varepsilon}.
\end{equation}
Thus, $\widetilde{r}_\varepsilon(\textrm{i} x):= \textrm{e}^{\textrm{i} \widetilde{g}_\varepsilon(x)}$ satisfies $\widetilde{r}_\varepsilon\in\mathcal{U}_n$, and following Proposition~\ref{prop:approxertophaseerrinequ}, the upper bound~\eqref{eq:ip.phaseerrtransformpi} implies $\|\widetilde{r}_\varepsilon-\exp(\omega \cdot)\|<2 $.
Especially, $\widetilde{r}_\varepsilon$ has the poles~\eqref{eq:ip.rtildewtonp1pipoles} and satisfies the assertion of the present proposition, which completes the proof.
\end{proof}

\subsection*{Unitary best approximations in the limit \texorpdfstring{$\omega\to(n+1)\pi^-$}{w to (n+1)pi-}}

For the remainder of the present appendix we consider the unitary best approximation $r_\omega\in\mathcal{U}_n$ to $\mathrm{e}^{\mathrm{i}\omega x}$ for a fixed degree $n$ and $\omega\in(0,(n+1)\pi)$. We recall $ r_\omega(\textrm{i} x) = \textrm{e}^{\textrm{i} g_\omega(x)}$, where the phase function $g_\omega$ relies on the poles of $r_\omega$, i.e., $s_j(\omega) = \xi_j(\omega) + \textrm{i} \mu_j(\omega)$.
For the phase $g_\omega(x)$ we recall the representation~\eqref{eq:defg},
\begin{equation}\label{eq:ip.gwagain}
g_\omega(x) = 2\sum_{j=1}^n\arctan \frac{x-\mu_j(\omega)}{\xi_j(\omega)}.
\end{equation}
Comparing with~\eqref{eq:defg}, we remark that $\theta=0$ for the unitary best approximation due to symmetry properties. Moreover, we recall $\xi_j(\omega)\neq 0$ for $j=1,\ldots ,n$ since the best approximation has minimal degree $n$.

Moreover, following Proposition~\ref{prop:omega0}, for $\omega\in(0,(n+1)\pi)$ the phase error of the unitary best approximation is strictly smaller than $\pi$, i.e.,
\begin{equation}\label{eq:ip.g1mgm1step0}
|g_\omega(x)-\omega x| < \pi,~~~\text{for $x\in[-1,1]$}.
\end{equation}
To simplify our notation we also write
\begin{equation}\label{eq:stepfctomegatodelta}
\omega = (n+1)\pi - \delta_\omega,
\end{equation}
where the case $\omega\in(0,(n+1)\pi)$ implies $\delta_\omega\in(0,(n+1)\pi)$ and $\omega\to (n+1)\pi^-$ is equivalent to $\delta_\omega\to 0^+$.

We need the following auxiliary propositions before proving Proposition~\ref{prop:polesconvwtonp1pi}.
\begin{proposition}\label{prop:gbminusgalowerbound}
For given points $\alpha,\beta\in[-1,1]$ the upper bound~\eqref{eq:ip.g1mgm1step0} implies
\begin{equation}\label{eq:gbminusgalowerbound}
g_\omega(\beta)-g_\omega(\alpha) > ((\beta-\alpha)(n+1) -2)\pi - (\beta-\alpha) \delta_\omega.
\end{equation}
\end{proposition}
\begin{proof}
Substituting $x=\beta$ and $x=\alpha$ in~\eqref{eq:ip.g1mgm1step0} we observe $g_\omega(\beta) -\omega \beta > -\pi$ and $g_\omega(\alpha) -\omega \alpha <\pi$, respectively. Thus,
\begin{equation*}
g_\omega(\beta) - g_\omega(\alpha) > (\beta-\alpha) \omega -2\pi.
\end{equation*}
Substituting~\eqref{eq:stepfctomegatodelta} for $\omega$ therein, we conclude~\eqref{eq:gbminusgalowerbound}.
\end{proof}

\begin{proposition}[Approaching step functions, part 2]
\label{prop:appendixBauxgeps2}
Let points $b_1<b_2$, a distance $\zeta>0$ and a tolerance $\kappa>0$ be given. Let $\{\omega_k\}_{k\in\mathbb{N}}$ denote a sequence with $\omega_k\to (n+1)\pi^-$ for $k\to\infty$. Let $g_\omega$ be given as in~\eqref{eq:ip.gwagain} with underlying poles $s_1(\omega),\ldots,s_n(\omega)$. Moreover, for the poles $s_j(\omega)=\xi_j(\omega)+\textrm{i}\mu_j(\omega)$, $j=1,\ldots,n$, we assume
$\xi_j(\omega_k)\to0^+$ for $k\to\infty$, and $\mu_j(\omega_k) \notin [b_2-\zeta,b_1+\zeta]$ for $k\in\mathbb{N}$. Then, for a sufficiently large $k$, we have
\begin{equation}\label{eq:ip.gwb2b1diff2}
0<g_{\omega_k}(\beta)-g_{\omega_k}(\alpha)< \kappa,~~~\text{for $\alpha<\beta$ with $[\alpha,\beta]\subset[b_1-\zeta/2,b_2+\zeta/2]$}.
\end{equation}
\end{proposition}
\begin{proof}
Since we assume $\xi_j(\omega_k)\to0^+$ for $k\to\infty$, we may assume $\xi_j(\omega_k)>0$ for a sufficiently large $k$. In this case all terms in the sum of $g_{\omega_k}(x)$~\eqref{eq:ip.gwagain} are monotonically increasing in $x$, and certainly, the lower bound in~\eqref{eq:ip.gwb2b1diff2} holds true since $\alpha<\beta$.

Moreover, arc tangents functions of the form $\arctan( y/\xi_j(\omega_k))$ for $y\in\mathbb{R}$ approach Heaviside step functions similar to~\eqref{eq:arctantostepfct}. In particular, for given $\widehat{\kappa}>0$ and $\zeta>0$ we find sufficiently large $k$ s.t.
\begin{subequations}
\begin{equation}\label{eq:ip.arctanyoverxijygrz}
\frac{\pi}{2} - \widehat{\kappa} < \arctan \frac{y}{\xi_j(\omega_k)} < \frac{\pi}{2} ,
~~~\text{for $y>\frac{\zeta}{2}$},
\end{equation}
and
\begin{equation}\label{eq:ip.arctanyoverxijylez}
-\frac{\pi}{2} < \arctan \frac{y}{\xi_j(\omega_k)} < -\frac{\pi}{2} + \widehat{\kappa},
~~~\text{for $y< -\frac{\zeta}{2}$}.
\end{equation}
\end{subequations}

Since we assume $\mu_j(\omega_k)\not\in [b_1-\zeta,b_2+\zeta]$ for $j=1,\ldots,n$ and $k\in\mathbb{N}$, points $x\in[b_1-\zeta/2,b_2+\zeta/2]$ have at least a distance of $\zeta/2$ to any $\mu_j(\omega_k)$. Thus, for sufficiently large $k$ and $\mu_j(\omega_k) < b_1-\zeta$ the inequalities in~\eqref{eq:ip.arctanyoverxijygrz} yield
\begin{subequations}\label{eq:ip.arctanyoverxijygrz2}
\begin{equation}
\frac{\pi}{2} - \widehat{\kappa} < \arctan \frac{x-\mu_j(\omega_k)}{\xi_j(\omega_k)} < \frac{\pi}{2} ,~~~\text{for $x\in\left[b_1-\frac{\zeta}{2},b_2+\frac{\zeta}{2}\right]$},
\end{equation}
and in a similar manner, for $\mu_j(\omega_k) > b_2+\zeta$ the inequalities~\eqref{eq:ip.arctanyoverxijylez} imply
\begin{equation}
-\frac{\pi}{2} < \arctan \frac{x-\mu_j(\omega_k)}{\xi_j(\omega_k)} < -\frac{\pi}{2} + \widehat{\kappa},~~~\text{for $x\in\left[b_1-\frac{\zeta}{2},b_2+\frac{\zeta}{2}\right]$}.
\end{equation}
\end{subequations}

In particular, for a sufficiently large $k$ these inequalities hold true uniformly in $j$. Assume $k$ fixed s.t.~\eqref{eq:ip.arctanyoverxijygrz2} holds true, and assume $\widehat{m}$ points $\mu_j(\omega_k)$ are smaller than $b_1-\zeta$ and $n-\widehat{m}$ are larger than $b_2+\zeta$.
Then, evaluating the inequalities~\eqref{eq:ip.arctanyoverxijygrz2} for the sum terms in $g_{\omega_k}$, we observe
\begin{equation*}
(2\widehat{m}-n)\pi - 2\widehat{m}\widehat{\kappa} 
< g_{\omega_k}(x)
< (2\widehat{m}-n)\pi + 2(n-\widehat{m})\widehat{\kappa} ,~~~\text{for $x\in\left[b_1-\frac{\zeta}{2},b_2+\frac{\zeta}{2}\right]$}.
\end{equation*}
In particular, this holds true for points $\alpha$ and $\beta\in [b_1-\zeta/2, b_2+\zeta/2]$, and we arrive at
\begin{equation*}
g_{\omega_k}(\beta) - g_{\omega_k}(\alpha) < 2n\widehat{\kappa}.
\end{equation*}
For a given tolerance $\kappa>0$, we may use $\widehat{\kappa}=\kappa/(2n)$ to concludes~\eqref{eq:ip.gwb2b1diff2}.
\end{proof}

\begin{proposition}\label{prop:arctantermconditions}
For $\delta\in(0,\pi/2)$ and $\xi\neq 0$, necessary conditions for
\begin{equation}\label{eq:arctan1pmu1mmuinpial}
\arctan\frac{1-\mu}{\xi } + \arctan\frac{1+\mu}{\xi } > \pi - \delta
\end{equation}
are $\mu\in(-1,1)$, and
\begin{equation}\label{eq:arctan1pmu1mmuxino}
0 < \xi < 2 \tan\frac{\delta}{2}.
\end{equation}
\end{proposition}
\begin{proof}
The assumption $\delta\in(0,\pi/2)$ entails $\pi-\delta>\pi/2$. Since the arc tangents function attains values strictly smaller than $\pi/2$, a necessary condition for~\eqref{eq:arctan1pmu1mmuinpial} is
\begin{equation}\label{eq:ip.muxicond00}
\arctan \frac{1-\mu}{\xi} >0~~~\text{and}~~\arctan \frac{1+\mu}{\xi} >0,
\end{equation}
Since the arc tangents function is odd and monotonically increasing we have $\arctan y>0$ if and only if $y>0$, and~\eqref{eq:ip.muxicond00} is equivalent to
\begin{equation}\label{eq:ip.muxicond0}
\frac{1-\mu}{\xi} >0~~~\text{and}~~\frac{1+\mu}{\xi} >0.
\end{equation}
In case of $\xi<0$, these inequalities are equivalent to $ 1 <\mu$ and $ \mu < -1 $, which is contradictory. On the other hand, for $\xi>0$ the conditions~\eqref{eq:ip.muxicond0} are equivalent to $\mu<1$ and $\mu >-1$. Thus,~\eqref{eq:arctan1pmu1mmuinpial} requires $\xi>0$ and $\mu\in(-1,1)$, and we consider this case in the remainder of the proof. Making use of the identity
\begin{equation*}
\arctan\frac{1-\mu}{\xi } + \arctan\frac{1+\mu}{\xi } 
= \pi - \left(\arctan\frac{\xi}{1-\mu} + \arctan\frac{\xi}{1+\mu}\right),
\end{equation*}
we observe that~\eqref{eq:arctan1pmu1mmuinpial} holds true if and only if
\begin{equation}\label{eq:ip.arctanboundalpha}
\arctan\frac{\xi}{1-\mu} + \arctan\frac{\xi}{1+\mu}<\delta.
\end{equation}
We note that
\begin{equation*}
\arctan\frac{\xi}{1\pm\mu} = \frac{\delta}{2}~~~\text{for}~~
\xi = (1\pm\mu) \tan\frac{\delta}{2}
< 2\tan\frac{\delta}{2}.
\end{equation*}
Since the arc tangents function is strictly monotonically increasing,
\begin{equation*}
\arctan\frac{\xi}{1\pm\mu} > \frac{\delta}{2}~~~\text{for}~~
\xi \geq 2\tan\frac{\delta}{2},
\end{equation*}
and thus, $\xi < 2\tan (\delta/2)$ is a necessary conditions for~\eqref{eq:ip.arctanboundalpha} to hold true, which carries over to~\eqref{eq:arctan1pmu1mmuinpial}.
\end{proof}

\begin{proof}[\bf Proof of Proposition~\ref{prop:polesconvwtonp1pi}]
We write $s_j(\omega)=\xi_j(\omega) + \textrm{i} \mu_j(\omega)$ for the poles of the unitary best approximation $r_\omega$ with $r_\omega(\mathrm{i} x)=\mathrm{e}^{\mathrm{i} g_\omega(x)}$.
We first show $\xi_j(\omega)\to 0^+$ for $j=1,\ldots,n$ and $\omega\to(n+1)\pi^-$, respectively, $\delta_\omega\to0^+$.
Following Proposition~\ref{prop:gbminusgalowerbound}, the phase of the best approximation satisfies
\begin{equation}\label{eq:ip.g1mgm1}
g_\omega(1)-g_\omega(-1) >  2n\pi - 2\delta_\omega.
\end{equation}
Substituting~\eqref{eq:ip.gwagain} for $g_\omega$, and making use of the identity
\begin{equation*}
- \arctan \frac{-1-\mu_j(\omega)}{\xi_j(\omega)} = \arctan \frac{1+\mu_j(\omega)}{\xi_j(\omega)},
\end{equation*}
we evaluate the left-hand side of~\eqref{eq:ip.g1mgm1} to
\begin{equation}\label{eq:ip.g1mgm1step2}
g_\omega(1)-g_\omega(-1) = 2\sum_{j=1}^n \left( \arctan \frac{1-\mu_j(\omega)}{\xi_j(\omega)} + \arctan \frac{1+\mu_j(\omega)}{\xi_j(\omega)}\right).
\end{equation}
The $\arctan$ function attains values smaller than $\pi/2$, which implies that for a fixed $j\in\{1,\ldots,n\}$,
\begin{equation*}
g_\omega(1)-g_\omega(-1) < 2(n-1)\pi + 2\left(\arctan \frac{1-\mu_j(\omega)}{\xi_j(\omega)} + \arctan \frac{1+\mu_j(\omega)}{\xi_j(\omega)}\right).
\end{equation*}
Thus, a necessary condition for~\eqref{eq:ip.g1mgm1} is that each sum term in~\eqref{eq:ip.g1mgm1step2} is strictly larger than $\pi - \delta_\omega$, i.e.,
\begin{equation}\label{eq:arctan1pmu1mmuinpialsum}
\arctan \frac{1-\mu_j(\omega)}{\xi_j(\omega)} + \arctan \frac{1+\mu_j(\omega)}{\xi_j(\omega)}
> \pi - \delta_\omega,~~~j=1,\ldots,n.
\end{equation}
We may assume $\delta_\omega \in (0,\pi/2)$ to deduce necessary conditions for~\eqref{eq:arctan1pmu1mmuinpialsum} from Proposition~\ref{prop:arctantermconditions}, i.e., $\mu_j(\omega)\in(-1,1)$ and
\begin{equation}\label{eq:ip.xiuniformbounded}
0 < \xi_j(\omega) < 2\tan\frac{\delta_\omega}{2}.
\end{equation}
We recall that $\delta_\omega\to 0^+ $ is equivalent to $\omega\to (n+1)\pi^-$. Thus,~\eqref{eq:ip.xiuniformbounded} entails
\begin{equation}\label{eq:ip.xiwtozero}
\xi_j(\omega) \to 0^+,~~~\text{for $j=1,\ldots,n$ and $\omega\to (n+1)\pi^-$.}
\end{equation}

\medskip
To proceed with the proof of $\mu_j(\omega)\to -1+2j/(n+1)$ we first show the following statement.
\begin{enumerate}[label=($\star$)]
\item\label{item:ip.muininterval} Assume a point $a\in [-1+1/(n+1), 1-1/(n+1)]$ and a tolerance $\zeta>0$ to be given. To simplify our notation we define $b_1 := a-\frac{1}{n+1}$ and $b_2 := a+\frac{1}{n+1}$.
Then, for a sufficiently small $\delta_\omega>0$ there exists at least one index $j\in\{1,\ldots,n\}$ s.t.\ $\mu_j(\omega)$ has a distance smaller than $1/(n+1)+\zeta$ to the point $a$, i.e., $\mu_j(\omega) \in [b_1-\zeta, b_2+\zeta ]$.
\end{enumerate}
We prove this statement by contradiction. Let the point $a$, respectively, $b_1$ and $b_2$, and $\zeta>0$ be fixed as above and assume the statement~\ref{item:ip.muininterval} is wrong. Then there exists a sequence $\{\omega_k\in(0,(n+1)\pi)\}_{k\in\mathbb{N}}$ with $\omega_k\to (n+1)\pi^-$ for $k\to \infty$ and $\mu_j(\omega_k)\not\in  [b_1-\zeta, b_2+\zeta]$ for $j=1,\ldots,n$ and $k\in\mathbb{N}$.

For $a\in [-1+1/(n+1), 1-1/(n+1)]$ at least one of the boundaries of the interval $[b_1-\zeta,b_2+\zeta]$ is located in $[-1,1]$. We assume this holds true for the right boundary of this interval in the following, i.e., $b_2+\zeta\leq 1$. Otherwise, similar arguments hold true for $b_1-\zeta\geq -1$.
Due to Proposition~\ref{prop:gbminusgalowerbound}, the phase $g_\omega$ satisfies
\begin{equation}\label{eq:ip.gwb2b1diff1}
g_{\omega}(b_2+\zeta/2)-g_{\omega}(b_1)
> \frac{\zeta (n+1)\pi}{2} + \mathcal{O}(\delta_\omega),
~~~\text{for $\omega\to(n+1)\pi^-$}.
\end{equation}

Since $\xi_j(\omega)\to 0^+$ for $\omega\to (n+1)\pi^-$ and $j=1,\ldots,n$ as in~\eqref{eq:ip.xiwtozero}, we particularly have $\xi_j(\omega_k)\to 0^+$ for $\{\omega_k\}_{k\in\mathbb{N}}$ from above and $k\to \infty$. Moreover, this holds true uniformly in $j=1,\ldots,n$.
Proposition~\ref{prop:appendixBauxgeps2} implies that, for a given $\widehat{\kappa}>0$ and $k$ sufficiently large,
\begin{equation*}
g_{\omega_k}(b_2+\zeta/2)-g_{\omega_k}(b_1)< \widehat{\kappa}.
\end{equation*}
We may choose $\widehat{\kappa}$ arbitrary small, in particular, $\widehat{\kappa}< \zeta (n+1)\pi/2$, and thus, this yields a contradiction to~\eqref{eq:ip.gwb2b1diff1} for $k\to\infty$, respectively, $\delta_{\omega_k} \to 0^+$.
We conclude that~\ref{item:ip.muininterval} holds true.

To proceed we assume that $\mu_1(\omega),\ldots,\mu_n(\omega)$ are numbered in ascending order, and we introduce the notation $\ell_j(\omega)$ for the distance between pairs of neighboring points $\mu_j(\omega)$. Namely,
\begin{equation*}
\begin{aligned}
&\ell_1(\omega):=\mu_1(\omega)+1,~~~\ell_{n+1}(\omega):=1-\mu_n(\omega),~~~\text{and}\\
& \ell_{j+1}(\omega):=\mu_{j+1}(\omega)-\mu_j(\omega),~~~\text{for $j=1,\ldots,n-1$}.
\end{aligned}
\end{equation*}
As a consequence of the statement~\ref{item:ip.muininterval}, for a given $\zeta>0$ and $\omega$ sufficiently close to $(n+1)\pi$ the distances between the points $\mu_j$ (and at the boundary) are bounded by
\begin{equation}\label{eq:ip.musdistance0}
\ell_j(\omega) < \frac{2}{n+1}+2\zeta,~~~\text{for $j=1,\ldots,n+1$}.
\end{equation}
We remark that a stricter version of this results holds true at the boundary, i.e., replacing $2\zeta$ by $\zeta$, however, the upper bounds~\eqref{eq:ip.musdistance0} are sufficiently strict in the present proof. Since the distances $\ell_1,\ldots,\ell_{n+1}$ span the interval $[-1,1]$, we further have
\begin{equation}\label{eq:ip.sumljistwo}
2 = \sum_{j=1}^{n+1}\ell_j(\omega).
\end{equation}
Thus, for $\ell_1,\ldots,\ell_{n+1}$ the upper bounds~\eqref{eq:ip.musdistance0} imply
\begin{equation*}
\ell_k(\omega) = 2 - \sum_{j\neq k}\ell_j(\omega) 
> 2 - \frac{2n}{n+1}-2n\zeta,~~~k=2,\ldots,n.
\end{equation*}
Together with~\eqref{eq:ip.musdistance0}, and since $2-2n/(n+1) = 2/(n+1)$, this yields the inequalities
\begin{equation}\label{eq:ip.ljsenclosed}
\frac{2}{n+1} - 2n\zeta
< \ell_k(\omega)
< \frac{2}{n+1} + 2\zeta,~~~\text{for $k=1,\ldots,n+1$}.
\end{equation}
Since~\eqref{eq:ip.ljsenclosed} holds true for arbitrary small $\zeta>0$ if $\omega$ is sufficiently close to $(n+1)\pi$, we conclude that distances between $\mu_j$'s and at the boundary satisfy
\begin{equation*}
\ell_k(\omega)\to \frac{2}{n+1},~~~k=1,\ldots,n+1,~~~\text{for $\omega\to(n+1)\pi^-$}.
\end{equation*}
In particular, for the $k=1$ and $k=n+1$ this implies
\begin{equation*}
\mu_1(\omega) \to -1 + \frac{2}{n+1},~~~\text{and}~~
\mu_n(\omega) \to 1 - \frac{2}{n+1},~~~\text{for $\omega\to(n+1)\pi^-$},
\end{equation*}
and repeating the same arguments for $\mu_2$, $\mu_{n-1}$, etc., we arrive at
\begin{equation*}
\mu_j(\omega) \to -1+\frac{2j}{n+1},~~~\text{for $j=1,\ldots,n$, and $\omega\to(n+1)\pi^-$}.
\end{equation*}
Together with~\eqref{eq:ip.xiwtozero}, this completes the proof of Proposition~\ref{prop:polesconvwtonp1pi}.
\end{proof}

We proceed with some notation before proving Proposition~\ref{prop:nodesconvwtonp1pi}.
Let $\zeta>0$ be given and sufficiently small.
To simplify our notation, we define
\begin{equation}\label{eq:ip.defajbjforinodes}
\begin{aligned}
&t_{j}=-1+\frac{j}{n+1},~~~a_{j}=t_j-\zeta,~~~b_j = t_j+\zeta,~~~\text{for $j=1,\ldots,2n+1$,}\\
&\text{and}~~~b_0=-1+\zeta,~~~\text{and}~~a_{2n+2} = 1-\zeta.
\end{aligned}
\end{equation}
For a sketch of these points we refer to Fig.~\ref{fig:sketchabeta}.
While $\zeta>0$ can be arbitrary small, we assume at least $\zeta<1/(2(n+1))$ s.t.\ $b_{j-1}<a_j$ for $j=1,\ldots,2n+2$. 

\begin{figure}
\centering
\includegraphics{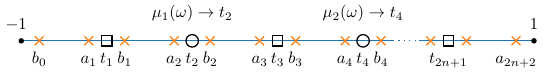}
\caption{
A sketch for the points $t_j$, $a_j$ and $b_j$ from~\eqref{eq:ip.defajbjforinodes}. In addition, the points $\mu_j(\omega)$ refer to the imaginary parts of the poles of the unitary best approximation to $\textrm{e}^{\textrm{i}\omega x}$, and following Proposition~\ref{prop:polesconvwtonp1pi}, the imaginary parts $\mu_j(\omega)$ converge to $t_{2j}$ for $j=1,\ldots,n$ and $\omega\to(n+1)\pi^-$. The ($\square$) and ($\circ$) symbols are used for the points $t_j=-1+j/(n+1)$ in the present figure and in figures~\ref{fig:ip.geps} and~\ref{fig:ip.gnp1pi} in a similar manner.}
\label{fig:sketchabeta}
\end{figure}

\begin{proposition}[Approaching step functions, part 3]
\label{prop:appendixBauxgeps3}
Let $r_\omega(\textrm{i} x)=\textrm{e}^{\textrm{i} g_\omega(x)}$ denote the best approximation to $\textrm{e}^{\textrm{i} \omega x}$ in $\mathcal{U}_n$ for a given degree $n$ and $\omega\in(0,(n+1)\pi)$. Then, for a given tolerance $\kappa>0$ and $\omega$ sufficiently close to $(n+1)\pi$,
\begin{equation}\label{eq:ip.gomegaonstraigs}
|g_\omega(x)-(-n+2k-2)\pi| < \kappa,~~~\text{for $x\in[b_{2(k-1)},a_{2k}]$ and $k=1,\ldots,n+1$},
\end{equation}
where $a_{2k}$ and $b_{2(k-1)}$ refer to~\eqref{eq:ip.defajbjforinodes} for a sufficiently small $\zeta>0$.
\end{proposition}
\begin{proof}
Let $\zeta>0$ correspond to $\zeta$ in~\eqref{eq:ip.defajbjforinodes}.
For the poles $s_j(\omega) = \xi_j(\omega) + \textrm{i} \mu_j(\omega)$, 
Proposition~\ref{prop:polesconvwtonp1pi} shows
\begin{equation*}
\xi_j(\omega)\to 0^+~~~\text{and}~~\mu_j(\omega) \to t_{2j},~~~\text{for $j=1,\ldots,n$ and $\omega\to(n+1)\pi^-$.}
\end{equation*}
In particular, for $\omega$ sufficiently close to $(n+1)\pi$ the points $\mu_j(\omega)$ are enclosed by $a_{2j}$ and $b_{2j}$ with an additional distance of $\zeta/2$, i.e.,
\begin{equation*}
a_{2j} + \frac{\zeta}{2} < \mu_j(\omega)
< b_{2j} - \frac{\zeta}{2},~~~j=1,\ldots,n.
\end{equation*}
Convergence of arc tangents function to a Heaviside step function~\eqref{eq:arctantostepfct} implies that, for a given $\widehat{\kappa}>0$ and $\omega$ sufficiently close to $(n+1)\pi$,
\begin{subequations}\label{eq:ip.atantostepfctpart31}
\begin{equation}\label{eq:ip.atantostepfctpart31a}
\frac{\pi}{2} - \widehat{\kappa}
< \arctan \frac{x-\mu_j(\omega)}{\xi_j(\omega)}
< \frac{\pi}{2},
~~~\text{for $x > b_{2j}$ with $j=1,\ldots,n$},
\end{equation}
and
\begin{equation}\label{eq:ip.atantostepfctpart31b}
-\frac{\pi}{2} 
< \arctan \frac{x-\mu_j(\omega)}{\xi_j(\omega)}
< -\frac{\pi}{2} + \widehat{\kappa},
~~~\text{for $x < a_{2j}$ with $j=1,\ldots,n$}.
\end{equation}
\end{subequations}
In the following we assume the points $\mu_1(\omega),\ldots,\mu_n(\omega) $ are in ascending order. For $\omega$ sufficiently close to $(n+1)\pi$ and $k\in\{1,\ldots,n+1\}$ and $x\in[b_{2(k-1)},a_{2k}]$, the inequality~\eqref{eq:ip.atantostepfctpart31a} applies for $\mu_1(\omega),\ldots,\mu_{k-1}(\omega)$ and~\eqref{eq:ip.atantostepfctpart31b} applies for $\mu_k(\omega),\ldots,\mu_{n}(\omega)$. Thus, substituting the respective inequality~\eqref{eq:ip.atantostepfctpart31} for each arc tangents term in $g_\omega(x)$ and $x\in[b_{2(k-1)},a_{2k}]$, we conclude
\begin{equation*}
(-n+2k-2)\pi - 2(k-1)\widehat{\kappa}
< g_\omega(x)
< (-n+2k-2)\pi + 2(n-k+1)\widehat{\kappa}.
\end{equation*}
Thus, for a given $\kappa>0$ and $\widehat{\kappa} = \kappa/(2n)$ we conclude~\eqref{eq:ip.gomegaonstraigs}.
\end{proof}

\begin{figure}
\centering
\includegraphics{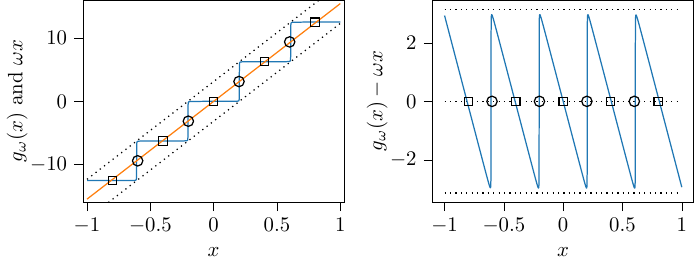}
\caption{
This figure illustrates the phase function $g_\omega$ of the unitary best approximation $r(\textrm{i} x)\approx \textrm{e}^{\textrm{i} \omega x}$, $r\in\mathcal{U}_n$ and its phase error for $n=4$ and $\omega=15.5$. Thus, in this example $\omega$ is close to $(n+1)\pi\approx15.71$. {\bf left:} The curved line corresponds to $g_\omega(x)$, the middle diagonal line illustrates $\omega x$, and the dotted-limes show $\omega x+\pi$ and $\omega x-\pi$.
In combination the symbols ($\square$) and ($\circ$) mark $\omega x$ for the points $t_j=-1+j/(n+1)$ with $j=1,\ldots,2n+1$, whereof $g$ behaves similar to a step functions with steps around the points $-1+2j/(n+1)$ marked by ($\circ$) for $j=1,\ldots,2n$. {\bf right:} This plot shows the phase error $g_\omega(x)-\omega x$ over $x$. The dotted lines illustrate $\pi$, $0$ and $-\pi$ as a reference. Similar to the left plot, the ($\square$) and ($\circ$) symbols mark the points $-1+j/(n+1)$ on the $x$-axis. The phase error attains absolute values close to $\pi$, but still strictly smaller than $\pi$, and can be described as a saw tooth function.}
\label{fig:ip.gnp1pi}
\end{figure}

To visually support the proof of Proposition~\ref{prop:nodesconvwtonp1pi} which is stated below, we also refer to Fig.~\ref{fig:ip.gnp1pi} where we illustrate the phase function and phase error of the unitary best approximation for a numerical example, i.e., $n=4$ and $\omega=15.5$. Thus $(n+1)\pi\approx 15.71$, and we have the case $\omega\approx (n+1)\pi$. In the left plot therein, we observe that the phase function behaves similar to a step function with steps around the points $t_{2j}=-1+2j/(n+1)$ for $j=1,\ldots,n$.

\begin{proof}[\bf Proof of Proposition~\ref{prop:nodesconvwtonp1pi}]
We proceed with the notation $\delta_\omega$ as in~\eqref{eq:stepfctomegatodelta} and $t_j$, $a_j$, $b_j$ and $\zeta>0$ as in~\eqref{eq:ip.defajbjforinodes}. Evaluating $\omega x$ at points $x=a_j$, we observe 
\begin{align}
\omega a_j &= ((n+1)\pi - \delta_\omega)\left(-1+\frac{j}{n+1} - \zeta\right)\notag\\
&= (-n+j-1)\pi - \zeta (n+1)\pi + \mathcal{O}(\delta_\omega),~~~j=1,\ldots,2n+2, \label{eq:inodesconv.omegaaj}
\end{align}
for $\delta_\omega\to0^+$.
In a similar manner, we have
\begin{equation}\label{eq:inodesconv.omegabj}
\omega b_j
= (-n+j-1)\pi + \zeta (n+1)\pi + \mathcal{O}(\delta_\omega),~~~j=0,\ldots,2n+1.
\end{equation}

{\bf Interpolation nodes.} From Corollary~\ref{cor:errattainsmax} we recall that the interpolation nodes $\mathrm{i}x_1(\omega),\ldots,\mathrm{i}x_{2n+1}(\omega)$ are directly related to the zeros $x_1(\omega),\ldots,x_{2n+1}(\omega)$ of the phase error.
Provided $k\in\{0,\ldots,n\}$ is fixed, the identity~\eqref{eq:inodesconv.omegaaj} for $j=2k+1$ implies
\begin{equation*}
(-n+2k)\pi = \omega a_{2k+1} + \zeta (n+1)\pi + \mathcal{O}(\delta_\omega),
\end{equation*}
and combining this with~\eqref{eq:ip.gomegaonstraigs} from Proposition~\ref{prop:appendixBauxgeps3}, we observe
\begin{subequations}\label{eq:ip.phaserroronstraights}
\begin{equation}\label{eq:ip.phaserroronstraightsa}
| g_\omega(a_{2k+1})-\omega a_{2k+1} - \zeta (n+1)\pi|
< \kappa  +\mathcal{O}(\delta_\omega),~~~k=0,\ldots,n.
\end{equation}
In a similar manner,~\eqref{eq:ip.gomegaonstraigs} and~\eqref{eq:inodesconv.omegabj} entail
\begin{equation}\label{eq:ip.phaserroronstraightsb}
| g_\omega(b_{2k+1})-\omega b_{2k+1} + \zeta (n+1)\pi|
< \kappa  +\mathcal{O}(\delta_\omega),~~~k=0,\ldots,n,
\end{equation}
\end{subequations}
As a consequence of Proposition~\ref{prop:appendixBauxgeps3}, the inequalities in~\eqref{eq:ip.phaserroronstraights} hold true for arbitrary small $\kappa>0$.
In particular, for $\delta_\omega$ sufficiently small~\eqref{eq:ip.phaserroronstraights} reveals $ g_\omega(x)-\omega x>0$ for $x=a_{2k+1}$ and $ g_\omega(x)-\omega x <0$ for $x=b_{2k+1}$.
Thus, the phase error $g_\omega(x)-\omega x$ has a zero between $a_{2k+1}$ and $b_{2k+1}$. We refer to this zero as $x_{2k+1}(\omega)$, i.e.,
\begin{equation}\label{eq:ip.zerosxjoddenclose}
x_{2k+1}(\omega) \in (a_{2k+1},b_{2k+1})~~~k=0,\ldots,n,
\end{equation}
and this holds true for all $\omega$ sufficiently close to $(n+1)\pi$.

We proceed to consider the phase error for $x=a_{2k}$ and $x=b_{2k}$ with $k=1,\ldots,n$. Combining~\eqref{eq:ip.gomegaonstraigs} and~\eqref{eq:inodesconv.omegaaj} we observe
\begin{subequations}\label{eq:ip.phaserrornearmax}
\begin{equation}\label{eq:ip.phaserrornearmaxa}
|g_\omega(a_{2k})-\omega a_{2k} + \pi - \zeta (n+1)\pi|
< \kappa + \mathcal{O}(\delta_\omega),
\end{equation}
and in a similar manner,~\eqref{eq:ip.gomegaonstraigs} and~\eqref{eq:inodesconv.omegabj} entail
\begin{equation}
|g_\omega(b_{2k})-\omega b_{2k}  - \pi + \zeta (n+1)\pi|
< \kappa + \mathcal{O}(\delta_\omega).
\end{equation}
\end{subequations}
Thus, for a sufficiently small $\kappa$ and $\delta_\omega$, the phase error is strictly negative at $a_{2k}$ and strictly positive at $b_{2k}$, and certainly, the phase error attains a zero in between these points. We refer to this zero as $x_{2k}$, i.e.,
\begin{equation}\label{eq:ip.zerosxjevenenclose}
x_{2k}(\omega)\in(a_{2k}, b_{2k}),~~~k=1,\ldots,n,
\end{equation}
and this holds true for all $\omega$ sufficiently close to $(n+1)\pi$. We recall that the phase error of the unitary best approximation $r_\omega\in\mathcal{U}_n$ has exactly $2n+1$ zeros for $\omega\in(0,(n+1)\pi)$. Following~\eqref{eq:ip.zerosxjoddenclose} and~\eqref{eq:ip.zerosxjevenenclose}, the phase error has zeros $x_j\in(a_j,\beta_j)$ for $j=1,\ldots,2n+1$ and $\omega$ sufficiently close to $(n+1)\pi$. In this case, the points $x_1,\ldots,x_{2n+1}$ exactly correspond to the zeros of the phase error.

Since we find a zero $x_j\in(a_j,\beta_j)$, and this holds true for an arbitrary small $\zeta>0$, and $ a_j<t_j<\beta_j$ with $ a_j,\beta_j\to t_j$ for $\zeta\to0^+$, we conclude
\begin{equation}\label{eq:ip.xjconvergetoetaj}
x_{j}(\omega) \to t_j = -1+\frac{j}{n+1},~~~j=1,\ldots,2n+1,~~~\delta_\omega\to0^+.
\end{equation}
In particular, the zeros of the phase error $x_j(\omega)$ provide the interpolation nodes $\mathrm{i}x_j(\omega)$ of $r_\omega$, and~\eqref{eq:ip.xjconvergetoetaj} shows the first assertion of Proposition~\ref{prop:nodesconvwtonp1pi}.

{\bf Equioscillation points.}
Consider the parameters $\zeta$ and $\kappa$, and $\delta_\omega$ sufficiently small s.t.\ the zeros $x_1(\omega),\ldots,x_{2n+1}(\omega)$ of the phase error satisfy $x_j(\omega)\in(a_j,b_j)$ as in~\eqref{eq:ip.zerosxjoddenclose} and~\eqref{eq:ip.zerosxjevenenclose}. Let $\eta_1(\omega),\ldots,\eta_{2n+2}(\omega)$ denote the equioscillation points of the phase error with $\eta_1(\omega)=-1$ and $\eta_{2n+2}(\omega)=1$ (see Theorem~\ref{thm:bestapprox}). The remaining $2n$ equioscillation points are located in $(-1,1)$, more precisely, these points are enclosed by pairs of neighboring zeros of the phase error, i.e.,
\begin{equation}\label{eq:ip.eopointsenclosed}
\eta_j(\omega)\in(x_{j-1}(\omega),x_j(\omega))~~~\text{for $j=2,\ldots,2n+1$}.
\end{equation}

Let $K_\omega$ denote the maximum of the phase error in absolute value over $[-1,1]$, i.e.,
\begin{equation*}
K_\omega := \max_{x\in[-1,1]} |g_\omega(x) - \omega x|.
\end{equation*}
We recall that the phase error attains its extreme values at equioscillation points, namely, from Proposition~\ref{prop:eo1max},
\begin{equation}\label{eq:ip.phaseerrateo}
g_\omega(\eta_{j}(\omega))-\omega \eta_{j}(\omega) = (-1)^{j+1} K_\omega ,~~~\text{for $j=1,\ldots,2n+2$}.
\end{equation}

As a consequence of propositions~\ref{prop:approxertophaseerrinequ} and~\ref{prop:omega0}, $K_\omega\in(0,\pi)$ if and only if $\omega\in(0,(n+1)\pi)$. In particular, $K_\omega \geq \pi$ for $\omega\geq (n+1)\pi$. Due to continuity of $K_\omega$ in $\omega$ (see Corollary~\eqref{cor:phaseerrcont}), we conclude $K_\omega\to \pi^-$ for $\omega\to(n+1)\pi^-$, equivalently, $\delta_\omega\to0^+$. In particular, we may choose $\delta_\omega$ sufficiently small s.t.\ for a given $\zeta>0$,
\begin{equation}\label{eq:ip.Koclosetopi}
\pi - \frac{\zeta n\pi}{2} < K_\omega < \pi.
\end{equation}
For a fixed $k\in\{1,\ldots,n+1\}$, the absolute value of the phase error satisfies the upper bound
\begin{equation}\label{eq:phaseerrorextendedboundn2kpi}
|g_\omega(x) - \omega x| < |g_\omega(x) - (-n+2k-2)\pi| + |\omega x - (-n+2k-2)\pi|,
\end{equation}
due to the triangular inequality.
Combining~\eqref{eq:inodesconv.omegaaj} for $j=2(k-1)$ and~\eqref{eq:inodesconv.omegabj} for $j=2k$, we observe
\begin{equation*}
(-n+2k-3)\pi + \zeta (n+1)\pi + \mathcal{O}(\delta_\omega) < \omega x <  (-n+2k-1)\pi - \zeta (n+1)\pi + \mathcal{O}(\delta_\omega),
\end{equation*}
for $x\in(b_{2(k-1)},a_{2k})$, and thus,
\begin{equation}\label{eq:wxextendedboundn2kpi}
|\omega x - (-n+2k-2)\pi| < \pi - \zeta (n+1)\pi + \mathcal{O}(\delta_\omega),~~~x\in(b_{2(k-1)},a_{2k}).
\end{equation}
Applying~\eqref{eq:ip.gomegaonstraigs} and~\eqref{eq:wxextendedboundn2kpi} to~\eqref{eq:phaseerrorextendedboundn2kpi}, we arrive at
\begin{equation*}
|g_\omega(x) - \omega x| < \pi - \zeta(n+1)\pi +\kappa + \mathcal{O}(\delta_w),~~~x\in(b_{2(k-1)},a_{2k}),~~k=1,\ldots,n+1.
\end{equation*}
For $\zeta>0$ given, we may choose $\kappa \leq \zeta n\pi$ s.t.\ this upper bound simplifies to
\begin{equation*}
|g_\omega(x) - \omega x| < \pi - \zeta n\pi + \mathcal{O}(\delta_w),~~~\text{for $x\in(b_{2(k-1)},a_{2k})$},
\end{equation*}
for $k=1,\ldots,n+1$. In particular, we may choose $\delta_\omega>0$ sufficiently small s.t.
\begin{equation}\label{eq:ip.phaserrdstpiinab}
|g_\omega(x) - \omega x| < \pi - \frac{\zeta n\pi}{2},~~~\text{for $x\in(b_{2(k-1)},a_{2k})$}.
\end{equation}
Combining~\eqref{eq:ip.phaseerrateo},~\eqref{eq:ip.Koclosetopi} and~\eqref{eq:ip.phaserrdstpiinab}, we conclude that for $\zeta>0$ given and $\delta_\omega$ sufficiently small the intervals $(b_{2(k-1)},a_{2k})$, $k=1,\ldots,n+1$, contain no equioscillation points. Thus,~\eqref{eq:ip.eopointsenclosed} specifies to
\begin{equation*}
\eta_{2k}(\omega) \in [a_{2k}, x_{2k}(\omega))~~~\text{and}~~\eta_{2k+1}(\omega) \in (x_{2k}(\omega),b_{2k}],~~~k=1,\ldots,n.
\end{equation*}
In particular, we observe
\begin{equation*}
\eta_{2k}(\omega), \eta_{2k+1}(\omega) \in[a_{2k},b_{2k}],~~~k=1,\ldots,n,
\end{equation*}
and since
\begin{equation*}
a_{2k}, b_{2k}\to t_{2k} = -1 + \frac{2k}{n+1},~~~\text{for $k=1,\ldots,n$, and $\zeta\to 0^+$,}
\end{equation*}
we conclude the assertion on equioscillation points for the limit $\omega\to(n+1)\pi^-$.
\end{proof}

\newcommand{\etalchar}[1]{$^{#1}$}

\end{document}